\documentclass[aap,preprint]{imsart}
\UseRawInputEncoding
\RequirePackage{amsthm,amsmath,amsfonts,amssymb}
\RequirePackage[numbers]{natbib}
\RequirePackage[colorlinks,citecolor=blue,urlcolor=blue]{hyperref}
\RequirePackage{graphicx}
\usepackage{array, caption, floatrow, tabularx, makecell, booktabs}
\usepackage{algorithmic}
\usepackage[section]{algorithm}
\usepackage{enumitem}
\usepackage[T1]{fontenc}
\newlist{steps}{enumerate}{1}
\setlist[steps, 1]{label = Step \arabic*:}
\usepackage{comment}
\usepackage{mathrsfs}
\usepackage{mathtools}
\usepackage{comment}
\usepackage{chngcntr}
\counterwithin{figure}{section}
\counterwithin{table}{section}
\startlocaldefs
\numberwithin{equation}{section}

\newtheorem{theorem}{Theorem}[section]
\newtheorem{lemma}[theorem]{Lemma}
\newtheorem{corollary}{Corollary}[section]
\theoremstyle{remark}

\newtheorem{assumption}{Assumption}[section]
\newtheorem{remark}{Remark}[section]

\newtheorem{experiment}{Experiment}[section]

 \newtheorem{innercustomgeneric}{\customgenericname}
\providecommand{\customgenericname}{}
\newcommand{\newcustomtheorem}[2]{%
  \newenvironment{#1}[1]
  {%
   \renewcommand\customgenericname{#2}%
   \renewcommand\theinnercustomgeneric{##1}%
   \innercustomgeneric
  }
  {\endinnercustomgeneric}
}

\newcustomtheorem{assumption1}{Assumption}
\DeclareMathOperator{\dist}{dist}
\DeclareMathOperator{\bias}{Bias}
\DeclareMathOperator{\var}{Var}

\DeclarePairedDelimiter\floor{\lfloor}{\rfloor}
\endlocaldefs

\newcommand{\rf}[1]{{\color{black}  #1}} 
\newcommand{\rs}[1]{{\color{black}  #1}} 

\begin{document}
\begin{frontmatter}
\title{Simplest random walk for approximating Robin boundary value problems and ergodic limits of reflected diffusions}
\runtitle{Simplest random walk for reflected diffusions}


\begin{aug}
\author[A]{\fnms{Benedict} \snm{Leimkuhler}\ead[label=e1]{B.Leimkuhler@ed.ac.uk}},
\author[B]{\fnms{Akash} \snm{Sharma}\ead[label=e2]{Akash.Sharma1@nottingham.ac.uk}}
\and
\author[B]{\fnms{Michael V.} \snm{Tretyakov}\ead[label=e3]{Michael.Tretyakov@nottingham.ac.uk}}
\address[A]{School of Mathematics and the Maxwell Institute for Mathematical Sciences, University of Edinburgh\printead[presep={,\ }]{e1}}

\address[B]{School of Mathematical Sciences, University of Nottingham \printead[presep={,\ }]{e2,e3}}

\end{aug}

\begin{abstract}
 A simple-to-implement weak-sense numerical method to approximate reflected stochastic differential equations (RSDEs) is proposed and analysed. It is proved that the method has the first order of weak convergence. Together with the Monte Carlo technique, it can be used to numerically solve linear parabolic and elliptic PDEs with Robin boundary condition. One of the key results of this paper is the use of the proposed method for computing ergodic limits, i.e. expectations with respect to the invariant law of RSDEs, both inside a domain in  $\mathbb{R}^{d}$ and on its boundary. This allows to efficiently sample from distributions with compact support.
    Both time-averaging and ensemble-averaging estimators are considered and analysed. A number of extensions are considered including a second-order weak approximation, the case of arbitrary oblique direction of reflection, and a new adaptive weak scheme to solve a Poisson PDE with Neumann boundary condition. The presented theoretical results are supported by several numerical experiments.
\end{abstract}



\end{frontmatter}

\section{Introduction}
This paper is devoted to weak approximation of stochastic differential equations (SDEs) with reflecting boundary conditions in a multidimensional domain $G\subset \mathbb{R}^{d}$. Many models from physics, biology, engineering, and finance can be described using SDEs. In some of those scenarios reflected SDEs (RSDEs) can be used as a modelling tool. Let us list a few examples from different fields. Applications of reflected diffusion processes in stock management as well as in quality control were considered in \cite{23}. Heavy traffic behaviour of queuing systems is modelled  using the reflected Brownian motion in \cite{24}.
In \cite{79} (see also \cite{80}), it is demonstrated that the solution (known as the optimal portfolio process) of consumption investment problems with transaction cost are governed by RSDEs. 
In \cite{leite2019}, the authors used RSDEs to model dynamics of reacting chemical species with the constraint that concentration of species cannot be negative.
Related issues in sampling measures and modelling natural phenomena in constrained space using reflective boundaries arise in many different connections, for example in molecular modeling \cite{ william2019},
biological models \cite{bio2,bio3}, 
continuum mechanics \cite{cont}, chemistry \cite{chem1,chem2}, and statistical inference \cite{stats1,stats3}.
We expect our study therefore to be of wide interest.

The Feynman-Kac formula gives the probabilistic representation of solutions of parabolic and elliptic PDEs with Neumann/Robin boundary condition as the expectation of a functional of the reflected diffusion process. Solving such PDEs using deterministic methods, for example finite difference methods, requires \rf{approximating} solutions in the whole domain, therefore the computational cost increases exponentially with \rf{increasing} dimension. The use of the Monte Carlo method to find solutions of such PDEs is preferred when the solution is not needed in the whole domain but only at certain points. Further, in the case of the Monte Carlo methods, independent trajectories can be simulated using parallel computers.

 Another important application of RSDEs is in making use of a stochastic gradient system with reflection for drawing samples from higher dimensional distributions with compact support (see \cite{55} and Section~\ref{section6} here). This application is one of the main objectives of this paper in the setting of ergodic limits  (Section~\ref{section6}).

Let $G \subset \mathbb{R}^{d}$ be a bounded domain with boundary $\partial G$, $ Q := [T_{0}, T)\times G$ be a cylinder in $\mathbb{R}^{d+1}$, and $S$ be the lateral surface of $Q$. Let $b: \bar{Q} \rightarrow \mathbb{R}^{d}$ and $\sigma : \bar{Q} \rightarrow \mathbb{R}^{d \times d}$.
Consider the RSDEs
\begin{align}\label{eq:1}
    dX(s) &= b(s,X(s))ds + \sigma(s,X(s))dW(s) + \nu(X(s))I_{\partial G}(X(s)) dL(s),
\end{align}
$$\, \rf{X(t_0)=x,}\;\; T_{0} \leq t_{0} \leq T, \; \;  t_{0} \leq s \leq T, \; \;  x \in \bar{G}, $$
where $W(s)$ is a $d$-dimensional standard Wiener process defined on a filtered probability space $(\Omega, \mathscr{F}, (\mathscr{F}_{s})_{s\geq 0}, \mathbb{P}) $; $\nu(z) $, $ z \in \partial G$, is the inward normal vector to the boundary $  \partial G$; and $I_{\partial G}(z)$ is the indicator function of $z \in \partial G$. 
Further, $L(s)$ is the  local time \rf{of the process $X(s)$ on the boundary} $ \partial G$  adapted to the filtration $(\mathscr{F}_{s})_{s\geq 0}$. A local time is a scalar increasing process continuous in $s$ which increases only when $X({s}) \in \partial G$ (see the precise definition in e.g. \cite{11,17,18}):
\begin{align*}
    L(t) = \int_{\rf{t_0}}^{t}I_{\partial G}\big(X(s)\big)dL(s),\;\;\; a.s.
\end{align*}
\rf{We also note \cite{11} that in the integral form of (\ref{eq:1}) the term
\[
K(t)=\int_{t_0}^{t}\nu(X(s))I_{\partial G}\big(X(s)\big)dL(s)
\]
is a $d$-dimensional bounded variation process which increases only when $X(s) \in \partial G$.}

The following questions are considered in this work:
\begin{itemize}
    \item \textit{How to numerically (in the weak sense) solve RSDEs and the  related linear parabolic equation with Robin boundary condition?}

    The proposed simple-to-implement weak approximation of (\ref{eq:1}) numerically solves a linear advection-diffusion equation with Neumann boundary condition, and its extension solves an advection-diffusion equation with a decay/growth term and with the non-homogeneous Robin (in other words, third) boundary condition  (see Sections~\ref{section2}-\ref{section4}).

    \item \textit{How to compute  ergodic limits in the domain $G$ as well as on the boundary $\partial G$? }

    We introduce time-averaging and ensemble-averaging estimators to numerically calculate expectations with respect to the invariant density of a reflected diffusion $X(t)$ which lies in $\bar{G}$. We also propose estimators to compute \rs{integrals} with respect to the normalised restriction of the invariant density of $X(t)$ on $\partial G$ (see Section~\ref{section6}).

    \item \textit{How to sample from distributions with compact support using Brownian dynamics?}

    This directly follows from the previous point. Drawing samples from compactly supported targeted distributions has many applications, especially in machine learning and molecular dynamics. The proposed algorithm applied to a stochastic gradient system (Brownian dynamics) with reflection  efficiently samples from distributions whose support is a compact set $\bar{G}$ as well as from distributions whose support is a $d-1$ dimensional hyper-surface $\partial G$ (see Subsection~\ref{section2.4.3}).

    \item \textit{How to numerically solve a linear elliptic equation with Robin boundary condition?}

    The probabilistic representation of the solution of the elliptic Robin problem involves integration of functionals of $X(t)$ on $[0,\infty)$. The weak method of Section~\ref{section2} is applied and analysed in the case of the elliptic problem (see Section \ref{section5}). The special case of the Poisson equation with Neumann boundary condition is treated separately, for which a new adaptive time-stepping algorithm is proposed (see Subsection~\ref{subsection5.2}).  
    
\end{itemize}

\rf{The} approaches to numerically approximate the solutions of RSDEs driven by Wiener processes have taken three directions. The first two approaches are penalty methods \cite{9,8,49} and projection methods \cite{14,15,49,bayer2010}.
Introduce the projection map onto $%
\bar{G}:$
$
\Pi (x)=\arg \min_{y\in \bar{G}}\left\vert x-y\right\vert \,,\;\;x\in
\mathbb{R}^{d}
$.
We note that if $x\in \bar{G}$ then $\Pi (x)=x.$ In projection schemes, the map $
\Pi (x)$ is applied at every step of a numerical scheme (e.g., the Euler method) approximating the RSDEs with the local time term omitted. To construct penalty schemes, one replaces the reflection term in the RSDEs with $\beta _{\lambda }(X(s))ds$, where $\beta _{\lambda }(x):= (x-\Pi (x))/\lambda
\,,\;\;x\in \mathbb{R}^{d}$, and  $\lambda$ is a positive constant.
We note that $\beta_{\lambda } (x)=0$ for $x\in \bar{G}$. Then these resulting SDEs are approximated, e.g. by the Euler scheme (see a brief description of penalty and projection methods in \cite[Section 5.6]{49} and the references therein).
The convergence of penalty schemes is mainly shown in the mean-square sense (see \cite{9,8}). In this paper, we are interested in weak approximation.

 Liu \cite{14} proposed a weak scheme by combining projection and orthogonal transformation of the diffusion matrix and obtained  weak first-order convergence. \rf{Costantini}, Pacchiarotti and Sartoretto \cite{15} obtained weak order $1/2$ through a projection scheme. Expectations of complex functionals of local time can be evaluated using their scheme.

Methods which are neither penalty methods nor projection methods, we term reflection methods \cite{asmussen1995,1,65,64,2,49}.  Milstein \cite{1} (see also \cite{49}) proposed a weak scheme with first order of accuracy to solve the Robin boundary value problem for parabolic PDEs. The scheme is not easy to implement as it requires to change the local coordinates when the discretized sample path reaches the proximity of the boundary (see  its implementation in \cite{Bernal2019}). Gobet \cite{65,64} suggested a reflected scheme in half space which locally approximates $\bar{G}$ and gained a half order of weak convergence for any oblique direction of reflection. The order is improved to one for the co-normal reflection. Bossy, Gobet and Talay \cite{2} analyzed an Euler scheme combined with a symmetrized procedure for simulation of reflected diffusion with oblique reflection and obtained first order of accuracy. 
 Since their scheme is based on Gaussian increments, they propose to restart the simulation if a discretized sample path takes a very large increment outside the domain. 
Approximating ergodic limits inside the domain $G$ using the scheme of \cite{2} was considered in \cite{43}, where the corresponding convergence in time step was proved with an order lower than $1$. The other possible limitation of the method is that it is not known how to adapt it to compute expectations of integrals with respect to local time.

In this paper, we propose a new reflection method to numerically approximate RSDEs. The method does not require any orthogonal transformation of diffusion matrix or change of local coordinates thus it is easy to implement. This new  method is based on the idea of symmetrized reflection on the boundary, however we apply the weak Euler scheme which uses bounded random variables making sure that discretized sample paths cannot move beyond the boundary outside the domain by more that ${\cal{O}}(\sqrt{h})$, where $h$ is the time step. We note that although our method is based on symmetrized reflection like \cite{2}, our approach to prove convergence is entirely different - we use appropriate PDEs \rf{as in typical} proofs of weak convergence \cite{1,41,49,42}. Further, the path we take for the analysis allows us to compute expectations of functionals of local time accompanied by optimal error estimates. Moreover, our approach works for time averaging estimation of ergodic limits,  both in the domain $G$ and on the boundary $\partial G$, and the corresponding first-order convergence in the time step is proved in both cases. We also extend our algorithm to approximate SDEs with reflection in any oblique inward direction with first order of convergence. In addition, we modify our algorithm by introducing a new procedure near the boundary which results in a second-order method. 

As already highlighted earlier, our method can be used to solve elliptic PDEs with Robin boundary condition, where the two cases are considered separately: the first  with a decay term (equations (\ref{eqn2.10})-(\ref{eqn2.11})), where we achieve first order accuracy; the second case (equations (\ref{gpe1})-(\ref{gpe2})) is without a decay term (the Poisson problem), which causes additional difficulties. We introduce an  adaptive time-stepping scheme based on a novel idea of double time-discretization  to solve the Poisson problem with first order  accuracy.

 In, for example, molecular dynamics (see \cite{100} and references therein) and Bayesian statistics (see e.g., \cite{101,102,103} and references therein), it is usually necessary to compute \rf{the} expectation of a given function with respect to the invariant law of the diffusion (ergodic limit). 
 There are two common approaches to computing ergodic limits. One is to simulate a numerical trajectory over a long period of time and take the average at discretized points (time-averaging estimation, see e.g. \cite{42, 40, 41, 43} and references therein). The other is to simulate many independent numerical realizations of the associated SDEs and average them at a sufficiently large finite time (ensemble-averaging estimation, see e.g. \cite{40}).

 In this work we are interested in finding how close  numerical time-averaging and ensemble-averaging estimators are to the corresponding expectations with respect to the RSDEs' stationary measure. In \cite{41} it \rf{was} shown how an appropriate Poisson equation can be used to obtain the long time average of the functional of diffusion processes. In the same spirit, here we make use of the Poisson equation with Neumann boundary condition for analysis of numerical time-averaging estimators. For proving accuracy of the ensemble-averaging estimators, we exploit a parabolic PDE with Neumann boundary condition.

In many applications it is required to sample from distributions with compact support, $\bar{G}$ (see e.g. \cite{75,73,74}). Although there are methods which have been proposed to sample from distributions subject to certain constraints, most of them are not well studied  theoretically except e.g. \cite{55,56}, where \cite{55} uses a projection scheme and \cite{56} exploits a penalty method for sampling from log-concave distributions with compact support. Both works establish bounds \rf{on \rs{a} distance (in total variation norm or Wasserstein distance) between the stationary measure} of a Markov chain generated by a numerical method and the corresponding stationary measure of the underlying reflected Brownian dynamics. 
We are interested in establishing  closeness (including convergence order) of estimators to the ergodic limits with respect to generic ergodic RSDEs allowing us to sample from arbitrary distributions with compact support $\bar G$.

Finally, in this paper, we also develop a methodology to sample from a targeted probability distribution which lies on the hyper-surface $\partial G$ that can be considered as the boundary of a bounded domain $G$. There are geometric Monte Carlo methods for sampling from distributions lying on $\partial G$ (see e.g. \cite{ 57,58} and the references therein). Our approach is different and is based on weak approximation of the local time of the reflected diffusion $X(t)$ on the boundary $\partial  G$ (see Subsection~\ref{section2.4.3} for details).

\section{Numerical method to approximate reflected SDEs} \label{section2}
In this section, we first discuss the existence and uniqueness of solutions of RSDEs (\ref{eq:1}).

In the paper we will use the following functional spaces. Let $C^{\frac{p+\epsilon}{2},p+\epsilon } (\bar{Q})$ (or $C^{p+\epsilon}(\bar{G})$) be a H{\"o}lder space containing functions $u(t,x)$ (or $u(x)$) whose partial derivatives 
 $\frac{\partial^{i+|j|}u }{\partial t^{i} \partial x^{j_{1}}\cdots \partial  x^{j_{d}}} $ (or
 $\frac{\partial^{|j|}u }{ \partial x^{j_{1}}\cdots \partial  x^{j_{d}}} $)  with
$2i + |j| < p+ \epsilon$ (or $|j| < p+\epsilon$) are continuous in $\bar{Q}$ (or $\bar{G}$) with finite norm $\mid\cdot \mid_{Q}^{(p+\epsilon)}$ (or $\mid\cdot\mid_{G}^{(p+\epsilon)}$), where $i \in \mathbb{N}\cup \{0\}$, $p \in \mathbb{N}\cup \{0\}$, $  0 <\epsilon < 1$, $j$ is a multi-index, and $\mid \cdot \mid^{p+\epsilon}$ is the H{\"o}lder norm (see details in \cite[pp. 7-8]{3}). 
However, for brevity of the notation, in what follows we will omit $\epsilon$ and write $C^{\frac{p}{2},p} (\bar{Q})$ (or $C^{p}(\bar{G})$) instead of $C^{ \frac{p+\epsilon}{2},p+\epsilon} (\bar{Q})$ (or $C^{p+\epsilon}(\bar{G})$), which should not lead to any confusion. \rs{The} notation $f(t,z) \in C^{\frac{p }{2}, p}(\bar{S})$ will have the same meaning as explained above.
In what follows we will omit $\epsilon$ in the notation of the H{\"o}lder norm  by writing $\mid \cdot \mid^{p}$ instead of $\mid \cdot \mid^{p+\epsilon}$. We also denote by $C(\bar{Q})$ the   set of functions which are continuous in $ \bar{Q}$.

We make the following assumptions regarding the problem (\ref{eq:1}).
\begin{assumption}\label{as:3}
The boundary $\partial G$ of domain $ G $ belongs to $C^{4}$.
\end{assumption}
\begin{assumption}\label{as:2}
The coefficients $b(t,x)$ and $\sigma(t,x)$ are $C^{1,2}(\bar{Q})$ functions.
\end{assumption}
If Assumptions \ref{as:3} and \ref{as:2} hold, there exists a unique strong solution to the RSDEs (\ref{eq:1}). The meaning of existence of unique strong solution is same as given in \cite[p. 149]{17}. We note that the unique solution of (\ref{eq:1}) exists when $b(t,x)$ and $\sigma(t,x)$ are just Lipschitz continuous in $x$ and the domain $G$ is either a convex \rf{set} or $C^{2}$ smooth  \cite{10,11, 12}. Assumptions~\ref{as:3}-\ref{as:2} are used in subsequent sections to prove optimal (highest possible) order of convergence for the proposed algorithm. At the same time, we note that the algorithm can be used in practice under weaker assumptions.

 Now, we construct a new algorithm which approximates RSDEs (\ref{eq:1}). Let $(t_{0},x) \in Q$. We introduce the uniform discretization of the time interval $[t_{0}, T]$ so that $t_{0} < \dots < t_{N}= T$, $ h := (T-t_{0})/N$ and $t_{k+1} = t_{k} + h$.

 We consider a Markov chain $(t_{k}, X_{k})_{k\geq 0}$ with $X_{0}= x$ approximating the solution $X_{t_0,x}(t)$ of the RSDEs (\ref{eq:1}). Since $X(t)$ cannot take values outside $\bar{G}$, the Markov chain should remain in $\bar{G}$ as well. To this end, the chain has an auxiliary (intermediate) step every time it moves from the time layer $t_{k}$ to $t_{k+1}$. We denote this auxiliary step by $X_{k+1}^{'}$. In moving from $X_{k}$ to $X_{k+1}^{'}$, we apply the weak Euler scheme
\begin{equation}\label{eq:11}
    X_{k+1}^{'}  = X_{k} + hb_{k} + h^{1/2}\sigma_{k}\xi_{k+1},
\end{equation}
where $b_{k} = b(t_{k},X_{k})$, $\sigma_{k} = \sigma(t_{k},X_{k})$ and $\xi_{k+1} = (\xi^{1}_{k+1}, \dots ,\xi^{d}_{k+1})^{\top}$, $\xi^{i}_{k+1}$, $i= 1,\dots,d$, $k=0,\dots,N-1$, are mutually independent random variables taking values  $\pm 1$ with probability $1/2$.
\\ \\

 \begin{minipage}[b]{\textwidth}
   \centering
        \includegraphics[scale =0.6]{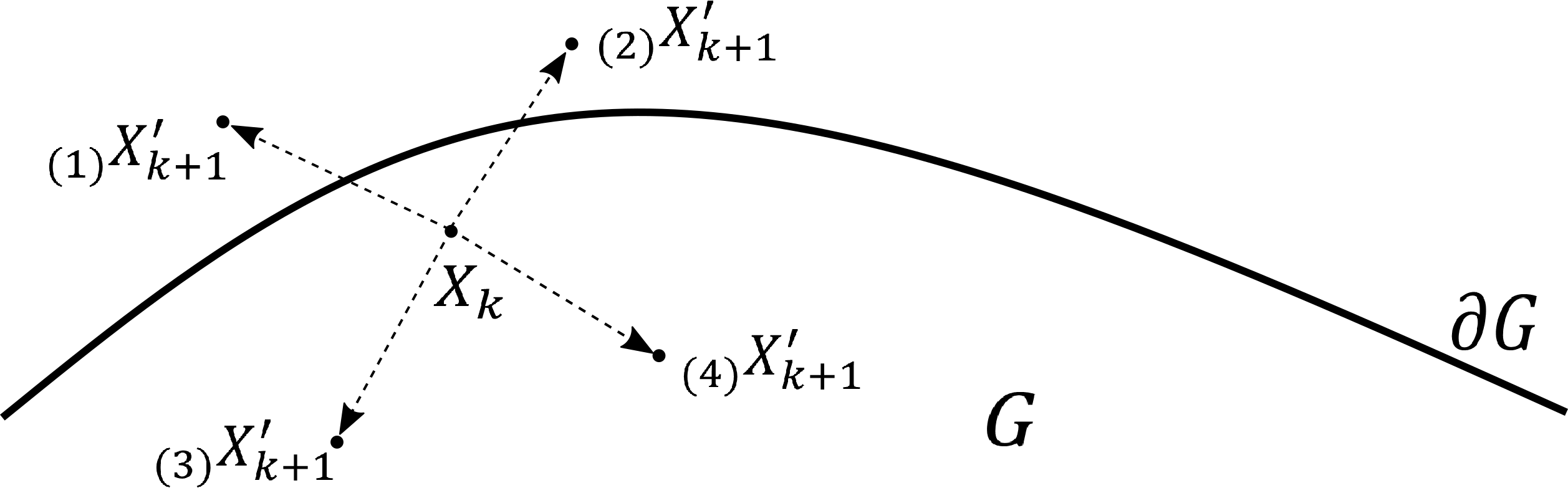}
    \captionof{figure}{ Four possible realizations ${}_{\footnotesize(i)}X_{k+1}^{'} $ of $X_{k+1}^{'}$ given $X_{k}$ in two dimensions.}\label{figure2.1}
\end{minipage}\qquad \qquad

Taking this auxiliary step $X_{k+1}^{'}$ while moving from $X_{k}$ to $X_{k+1}$ \rf{represents} cautious behaviour and gives us an opportunity to check whether the realized value of $X_{k+1}^{'}$ is inside the domain $G$ or not (see Fig.~\ref{figure2.1}). If $X_{k+1}^{'} \in \bar{G}$ then on the same time layer we assign values to $X_{k+1}$ as
\begin{align}
    X_{k+1} &= X_{k+1}^{'}. \label{eq:12}
\end{align}
 However, if the realized value of $X_{k+1}^{'}$ goes outside of $\bar{G}$ then we need an additional construction so that $ X_{k+1} \in G$ (see Fig.~\ref{figure2.2}). First, we find the projection of $X_{k+1}^{'}$ onto $\partial G$ which we denote as $  X_{k+1}^{\pi} $ and we calculate $ r_{k+1} = \dist (X_{k+1}^{'}, X_{k+1}^{\pi})$ which is the shortest distance between $X_{k+1}^{'}$ and $X_{k+1}^{\pi}$. Note that $\dist(X_{k}, X_{k+1}^{'}) = \mathcal{O}(h^{1/2})$, therefore under Assumption \ref{as:3} and for sufficiently small $h$, the projection $X_{k+1}^{\pi}$ of $X_{k+1}^{'}$ on $\partial G$ is unique \cite[Proposition 1]{2}. Moreover, the projection $X_{k+1}^{\pi}$ and the shortest distance $r_{k+1}$ satisfy the following equation, $X_{k+1}^{\pi} = X_{k+1}^{'} + r_{k+1}\rf{\nu(X_{k+1}^{\pi})} $, \rf{where $ \nu(X_{k+1}^{\pi}) $ is the inward normal vector to the boundary $\partial G$ at the projection  $X_{k+1}^{\pi}$.} 
 Thereafter, we add $ r_{k+1}\nu(X_{k+1}^{\pi}) $ to $ X_{k+1}^{\pi}$ to arrive at a point which we take as $ X_{k+1}$. This transition from intermediate step $X_{k+1}^{'}$ to $X_{k+1}$ makes sure that $X_{k+1} \in G $. We also  highlight that 
 $X_{k+1}^{'}$ and $ X_{k+1}$ are symmetric around $X_{k+1}^{\pi}$ along the direction $\nu(X_{k+1}^{\pi})$. Therefore, combining the above steps of calculating $ X_{k+1}^{\pi}$ from $X_{k+1}^{'}$ and then $X_{k+1}$ from $ X_{k+1}^{\pi}$, we have
\begin{align}
    X_{k+1} = X_{k+1}^{'} + 2r_{k+1}\nu(X_{k+1}^{\pi}). \label{eq:15}
\end{align}
\begin{minipage}[b]{\textwidth}
   \centering
        \includegraphics[scale =0.6]{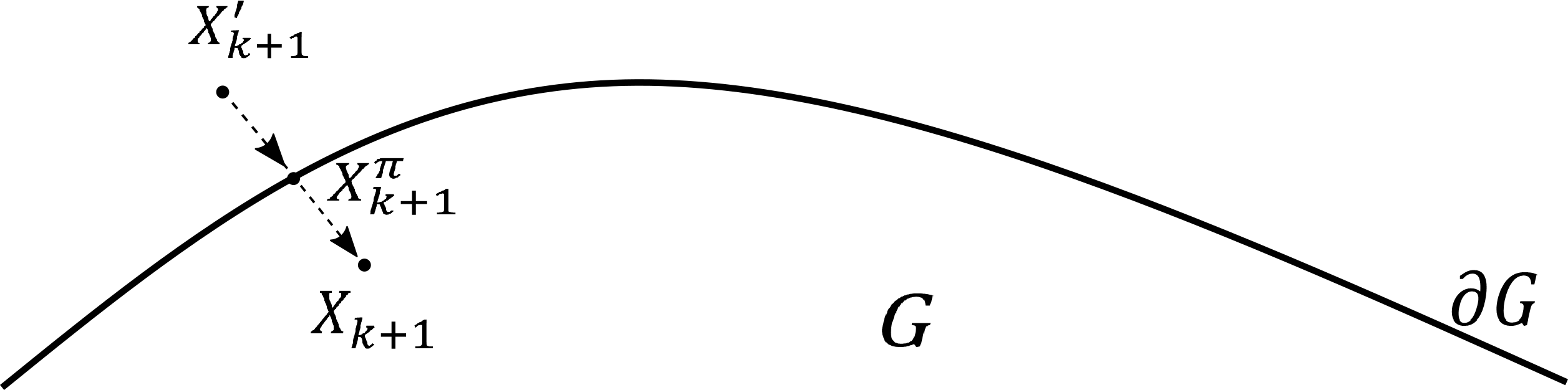}
    \captionof{figure}{One step transition in two dimensions from $X_{k+1}^{'}$ to $X_{k+1}$ using projection $X_{k+1}^{\pi}$ of $X_{k+1}^{'}$ on $\partial G$.} \label{figure2.2}
\end{minipage}\qquad \qquad \qquad
 We formally write our algorithm as Algorithm \ref{algorithm2.1}.
 \vspace{-2ex}
\begin{algorithm}
\caption{Algorithm to approximate normal reflected diffusion\label{algorithm2.1}}

\begin{steps}
\begin{em}
\item \begin{textit} Set $X_{0} = x$, $X_{0}^{'}= x$, $k=0$.\end{textit}
\item \begin{textit}Simulate $\xi_{k+1}$ and find $X_{k+1}^{'}$ using (\ref{eq:11}).\end{textit}

\item \textbf{If} $ X_{k+1}^{'} \in \bar{G} $ then $X_{k+1}= X_{k+1}^{'}$, \textbf{else} \\
\begin{enumerate}[label= (\roman*)]
\item  find the projection $X_{k+1}^{\pi}$ of $X_{k+1}^{'}$ on $\partial G$,
\item calculate $ r_{k+1} = \dist(X_{k+1}^{'}, X_{k+1}^{\pi}) $ and find $X_{k+1}$ according to (\ref{eq:15}).
\end{enumerate}
\item \textbf{If} $k+1=N$ then \textbf{stop}, \textbf{else} put $k := k+1$ and \textbf{return} to Step 2.
\end{em}
\end{steps}

\end{algorithm}
\vspace{-4ex}
\begin{remark}
To approximate the RSDEs inside the domain $G$ in Algorithm \ref{algorithm2.1}, we \rf{exploit} a particular method \rf{with discrete random variables used for approximating the Wiener increments}, the weak Euler scheme (\ref{eq:11}). Instead of (\ref{eq:11}), one can use any approximation with local weak order $2$ and bounded increments (e.g., the walk over spheres as in \cite{1,49} (see also \cite{Bernal2019}) or any other method \rf{with discrete random variables}) and complement it with the symmetric reflection (\ref{eq:15}). Such versions of  Algorithm~\ref{algorithm2.1} have the same convergence properties as the ones proved in this paper for Algorithm~\ref{algorithm2.1}. We restrict ourselves here to the weak Euler scheme (\ref{eq:11}) because it is the simplest scheme and also for definiteness.
\end{remark}

\section{Solving parabolic PDEs with Robin boundary condition}\label{section4}

 In Subsection \ref{section3.1}, we introduce a parabolic PDE with Robin boundary condition along with  assumptions required for  the existence of its solution and with the link to the reflected diffusion process via the Feynman-Kac formula. Subsection~\ref{section3.2} describes an extension of Algorithm~\ref{algorithm2.1} to solve the Robin parabolic problem. We state the main convergence theorem of the proposed algorithm in Subsection~\ref{subsection3.3} and prove it in Subsection~\ref{subsection3.4}.

\subsection{Probabilistic representations}\label{section3.1}

Consider the parabolic PDE
\begin{equation}\label{eq:2}
\frac{\partial u}{\partial t} + \frac{1}{2} \sum\limits_{i,j=1}^{d} a^{ij}(t,x) \frac{\partial u}{\partial x^{i} \partial x^{j}} + \sum \limits_{i=1}^{d} b^{i}(t,x)\frac {\partial u}{\partial x^{i}} + c(t,x)u + g(t,x) = 0, \; \; \;  (t,x) \in Q,
\end{equation}
with terminal condition
\begin{equation}\label{eq:3}
u(T,x) = \varphi(x), \; \; x \in \bar{G},
\end{equation}
and Robin boundary condition
\begin{equation}\label{eq:4}
\frac{ \partial u}{\partial \nu} + \gamma(t,z)u = \psi(t,z), \; \; \; (t,z) \in S, 
\end{equation}
where $\nu = \nu(z)$ is the direction of the inner normal to the surface $\partial G$ at a point $ z\in \partial G$.

We can write equation (\ref{eq:2}) in a more compact form as
\begin{equation}\label{eq:5}
    \frac{\partial u}{\partial t} + (b\cdot\nabla)u + \frac{1}{2}(a:\nabla\nabla)u + cu + g = 0,
\end{equation}
where $(:)$ denotes the Frobenius product of two matrices, $(\cdot)$ denotes the scalar product of two vectors, $c =c(t,x)$ and $g= g(t,x)$ are scalar functions,  and $ a(t,x) = \{ a^{ij}(t,x) \} $ is a $d \times d$ symmetric matrix,
 i.e. $ a : \bar{Q} \rightarrow \mathbb{R}^{d\times d}$.

 In addition to Assumptions \ref{as:3}-\ref{as:2}, we also make the following assumptions.
\begin{assumption}\label{a01}
$ g(t,x) \in C^{1,2}(\bar{Q})$, $\varphi(x) \in C^{4}(\bar{G})$, and $ \psi(t,z) \in C^{1.5,3}(\bar{S})$.
\end{assumption}
\begin{assumption}\label{asn2}
$c(t,x) \in C^{1,2}(\bar{Q})$ and $\gamma(t,z) \in C^{1.5,3}(\bar{S})$.
\end{assumption}
\begin{assumption}\label{as:1}
The symmetric matrix $a = \{a^{ij}\}$ satisfies the condition of uniform ellipticity in $\bar{Q}$, i.e., there exists a positive constant $a_{0}$ such that for all $ y \in \mathbb{R}^{d}$:
\begin{equation}
    a_{0}\mid y \mid^{2} \; \leq \; (a(t,x)y\cdot y) , \;\;\;  (t,x) \in \bar{Q}.
    \label{eq:ell}
\end{equation}
\end{assumption}

\noindent Define a sequence of functions $v_{k}$ recursively as
\begin{align*}
    v_{0} = \varphi,\;\;\;\;
    v_{k+1} = \rs{-}\sum\limits_{i=0}^{k}\binom{k}{i}\mathscr{A}^{(i)}v_{k-i}  - \frac{\partial^{k}g}{\partial t^{k}}(T,x),
\end{align*}
with
\begin{equation*}
    \mathscr{A}^{(i)}v = \frac{1}{2}\sum\limits_{j,l=1}^{d}\frac{\partial^{i}a^{jl}}{\partial t^{i}}(T,x)\frac{\partial^{2}v}{\partial x^{j}{\partial x^{l}}} \rs{+} \sum\limits_{j=1}^{d}\frac{\partial^{i}b^{j}}{\partial t^{i}}(T,x)\frac{\partial v}{\partial x^{j}} \rs{+} \frac{\partial^{i}c}{\partial t^{i}}(T,x)v.
\end{equation*}

\begin{assumption}\label{ass4}
The compatibility condition of order 1 is fulfilled for the problem (\ref{eq:2})-(\ref{eq:4}), i.e.,  the following relationship holds:
\begin{equation*}
    \bigg[ \frac{\partial v_{k}}{\partial \nu} + \sum\limits_{i=0}^{k} \binom{k}{i}\frac{\partial^{i}\gamma}{\partial t^{i}}\bigg|_{t=T}v_{k-i}\bigg]\bigg|_{\partial G} = \frac{\partial^{k} \psi}{\partial t^{k}}\bigg|_{t=T}, \;\;\;\;\; k=0,1.
\end{equation*}
\end{assumption}

It is known \cite{3} that if Assumptions \ref{as:3}-\ref{as:2} and \ref{a01}-\ref{ass4} are satisfied then the problem (\ref{eq:2})-(\ref{eq:4}) has a unique solution $u(t,x) \in C^{2,4}(\bar{Q})$ satisfying the inequality
\begin{equation} \label{eq:6}
    \mid u\mid_{\bar{Q}}^{(4)}\; \leq \;C(T)\Big(\mid g\mid_{\bar{Q}}^{(2)} + \mid \varphi\mid_{\bar{G}}^{(4)} + \mid \psi\mid_{\bar{S}}^{(3)}\Big),
\end{equation}
where $C(T)$ is a positive constant dependent on $T$. We also note that $u(t,x)$ satisfies the PDE (\ref{eq:2}) in $\bar Q$ under Assumptions~\ref{as:3}-\ref{as:2} and \ref{a01}-\ref{ass4}.

Let a matrix $ \sigma(t,x)$ be found from the equation
$$ \sigma(t,x)\sigma(t,x)^{\top} = a(t,x).$$ As is known \cite{16,18,17}, the probabilistic representation of the solution of problem (\ref{eq:2})-(\ref{eq:4}) is given by
\begin{equation}\label{eq:7}
u(t_{0},x) = \mathbb{E}\big(\varphi(X_{t_{0},x}(T))Y_{t_{0},x,1}(T) + Z_{t_{0},x,1,0}(T)\big),
\end{equation}
where $X_{t_{0},x}(s)$, $ Y_{t_{0},x,y}(s) $, $Z_{t_{0},x,y,z}(s)$, $ s \geq t_{0}$, is the solution of the Cauchy problem for the system of RSDEs
\begin{align}
 dX(s) &= b(s,X(s))ds + \sigma (s,X(s))dW(s) + \nu (X(s))I_{\partial G}(X(s)) dL(s),    \label{eq:8} \\
 dY(s) &= c(s,X(s))Y(s)ds + \gamma(s,X(s))I_{\partial G}(X(s))Y(s)dL(s),  \label{eq:9}\\
 dZ(s) &= g(s,X(s))Y(s)ds - \psi(s,X(s))I_{\partial G}(X(s))Y(s)dL(s),     \label{eq:10}
 \end{align}
 with $X(t_{0}) = x$, $Y(t_{0}) = y$, $Z(t_{0}) = z$,
$ T_{0} \leq t_{0} \leq s \leq T, \; \;  x \in \bar{G}.$

 \subsection{Numerical method} \label{section3.2}
 In this subsection we modify our Algorithm \ref{algorithm2.1} to construct a Markov chain $(t_{k}, X_{k},Y_{k},Z_{k})_{k\geq 0}$ with $X_{0}= x$, $Y_{0}= 1$, $Z_{0}= 0$  to approximate the solution $u(t_{0},x)$ of (\ref{eq:2})-(\ref{eq:4}) at $(t_{0},x) \in \bar{Q}$. We approximate RSDEs (\ref{eq:8}) according to Algorithm~\ref{algorithm2.1} and complement it by an approximation of (\ref{eq:9}) and (\ref{eq:10}). If the intermediate step $X_{k+1}^{'}$ introduced in Algorithm~\ref{algorithm2.1}, belongs to $\bar{G}$ then we use the Euler scheme:
 \begin{align}
      Y_{k+1} &= Y_{k}  + hc(t_{k},X_{k})Y_{k}, \label{mdy1}\\
      Z_{k+1} &= Z_{k} + hg(t_{k}, X_{k})Y_{k}. \label{mdz1}
 \end{align}
If $X_{k+1}^{'} \notin \bar{G}$ then
 \begin{align}
        Y_{k+1} &= Y_{k}  + hc(t_{k},X_{k})Y_{k} + 2r_{k+1}\gamma(t_{k+1}, X_{k+1}^{\pi})Y_{k}\nonumber \\ & \;\;\;\; +          2r_{k+1}^{2}\gamma^{2}(t_{k+1},X_{k+1}^{\pi})Y_{k},\label{eq:16}\\
     Z_{k+1} &= Z_{k} + hg(t_{k}, X_{k})Y_{k} - 2r_{k+1}\psi(t_{k+1}, X^{\pi}_{k+1})Y_{k} \nonumber \\& \;\;\;\; - 2r_{k+1}^{2}\psi(t_{k+1}, X_{k+1}^{\pi})\gamma(t_{k+1}, X_{k+1}^{\pi})Y_{k},\label{eq:17}
\end{align}
where $X_{k+1}^{\pi}$ is the projection of $X_{k+1}^{'}$ on $\partial G$ and $ r_{k+1} = \dist (X_{k+1}^{'}, X_{k+1}^{\pi})$ which is the shortest distance between $X_{k+1}^{'}$ and $X_{k+1}^{\pi}$.

 The approximation (\ref{eq:16})-(\ref{eq:17}) is derived via numerical analysis (see \rs{Lemma~\ref{lemma2.3} and} Theorem~\ref{conp}) aimed at obtaining first-order weak convergence of the proposed algorithm.
 We write this modified algorithm as Algorithm~\ref{algorithm3.1}.
\vspace{-2ex}
\begin{algorithm}
\caption{Algorithm to approximate (\ref{eq:8})-(\ref{eq:10}) }\label{algorithm3.1}
\begin{steps}
\begin{em}
\item \begin{textit} Set $X_{0} = x$, $Y_{0} = 1$, $Z_{0} = 0$, $X_{0}^{'}= x$, $k=0$.\end{textit}
\item \begin{textit}Simulate $\xi_{k+1}$ and find $X_{k+1}^{'}$ using (\ref{eq:11}).\end{textit}
       \item \textbf{If} ${ X_{k+1}^{'} \in \bar{G}} $ then
$X_{k+1} = X_{k+1}^{'}$ and calculate $Y_{k+1}$ and $Z_{k+1}$ according to (\ref{mdy1}) and (\ref{mdz1}), respectively,\;
\textbf{else} find $X_{k+1}$, $Y_{k+1}$ and $Z_{k+1}$ according to (\ref{eq:15}), (\ref{eq:16}) and (\ref{eq:17}), respectively.

\item \textbf{If} $k+1=N$ then \textbf{stop}, \textbf{else} put $k := k+1$ and \textbf{return} to Step 2.
\end{em}
\end{steps}

\end{algorithm}
\vspace{-2ex}

We note that Algorithm \ref{algorithm3.1} can be applied to the Robin problem (\ref{eq:2})-(\ref{eq:4}) in the layer-wise manner.
Recall (see \cite[Chapters 7-8]{49} and also references therein) that layer methods are deterministic numerical methods for PDEs which are constructed using probabilistic representations of the PDEs' solutions together with the weak-sense approximation of SDEs. Based on Algorithm \ref{algorithm3.1}, we can write a layer method  which is a deterministic method with fictitious nodes for (\ref{eq:2})-(\ref{eq:4}). In the one-dimensional case such a layer method was proposed in \cite{81} (see also \cite[Section 8.4.4]{49}). It was applied to a one-dimensional semilinear parabolic PDE with Neumann boundary condition. Thanks to the results of our paper, we now have a probabilistic representation of that layer method from \cite{81} and hence can prove its global order of convergence. Furthermore, based on Algorithm \ref{algorithm3.1}, we can construct a layer method for a multi-dimensional semilinear Neumann problem, which is simpler than the layer method proposed in \cite{81} (see also \cite[Section 8.5]{49}) for the multi-dimensional case which used the weak approximation from \cite{1}.
This connection of Algorithms~\ref{algorithm2.1} and \ref{algorithm3.1} with layer methods also provides further intuition for the approximation of $X(t)$ near the boundary: from the weak-sense \rf{perspective}, we replace the directional derivative in the Robin boundary condition with the central finite-difference using a fictitious node outside the domain.

\subsection{Finite-time convergence theorem}\label{subsection3.3}
We state the main theorem of this section which gives the estimate for the weak-sense error of Algorithm~\ref{algorithm3.1} at finite time.
\begin{theorem}\label{conp}
 The weak order of accuracy of Algorithm \ref{algorithm3.1} is  $\mathcal{O}(h)$ under Assumptions \ref{as:3}-\ref{as:2} and \ref{a01}-\ref{ass4}, i.e., for sufficiently small $h>0$
 \begin{equation}
    \mid \mathbb{E}(\varphi(X_{N})Y_{N} + Z_{N}) - u(t_{0},X_{0}) \mid \;  \leq \; Ch, \label{eq:th4.4}
\end{equation}
where $u(t,x)$ is solution of (\ref{eq:2})-(\ref{eq:4}) 
and $C$ is a positive constant independent of $h$.
\end{theorem}

We will prove this theorem in the next subsection.  The scheme of the proof is roughly as follows (cf. \cite{1,49}). We first prove two lemmas on weak local errors of Algorithm~\ref{algorithm3.1}: Lemma~\ref{lemma2.2} gives order ${\cal{O}}(h^2)$ for the one-step approximation for the intermediate step  $X_{k+1}^{'}$ (i.e., of the Euler approximation)  and Lemma~\ref{lemma2.3} gives local order ${\cal{O}}(h^{3/2})$ for $X_{k+1}$ when  $X_{k+1}^{'}$ goes outside $\bar{G}$. The number of steps when $X_{k+1}^{'} \in \bar{G}$ is obviously ${\cal{O}}(1/h)$. The sense of Lemma~\ref{bl} is that the average number of steps when $X_{k+1}^{'} \notin \bar{G}$ is ${\cal{O}}(1/\sqrt{h})$. Appropriately combining the three lemmas, we get  first order convergence as stated in the theorem.

\subsection{Proof of Theorem~\ref{conp}}\label{subsection3.4}

This subsection is devoted to analysis of the error incurred while numerically solving the Robin problem  (\ref{eq:2})-(\ref{eq:4})  using Algorithm \ref{algorithm3.1} and the probabilistic representation (\ref{eq:7}). In Subsection~\ref{sec4.1} we prove two lemmas regarding the one-step approximation corresponding to Algorithm~\ref{algorithm3.1}. In Subsection~\ref{section2.4} we prove a lemma on  the average number of steps when $X_{k+1}^{'} \notin \bar{G}$.  Theorem~\ref{conp} itself is proved in Subsection~\ref{ssec:thm31prf}. We introduce the additional notation to be used in the future analysis: $u_{k+1} = u(t_{k+1},X_{k+1})$, $u_{k}= u(t_{k},X_{k})$, $u_{k+1}^{\pi}=u(t_{k+1},X_{k+1}^{\pi})$, $u^{\prime}_{k+1} = u(t_{k+1},X_{k+1}^{'})$, $a_{k} =a(t_{k}, X_{k})$, $b_{k}= b(t_{k},X_{k})$, $c_{k}= c(t_{k},X_{k})$, $g_{k}= g(t_{k},X_{k})$, $\psi_{k+1}^{\pi} = \psi(t_{k+1}, X_{k+1}^{\pi})$, $\gamma_{k+1}^{\pi}= \gamma(t_{k+1},X_{k+1}^{\pi})$, $\sigma_{k} = \sigma(t_{k}, X_{k})$, $Y_{k+1}^{'}  = Y_{k}+hc_{k}Y_{k}$, and $Z^{'}_{k+1} = Z_{k}+hg_{k}Y_{k}$. For $d$-dimensional vectors  $V_{j}$, $j =1,\dots,p$,
we denote the $p$-th  spatial derivative of a smooth function $v(x)$ evaluated in
the directions  $V_{j}$ by $D^{p}v(x)[V_{1},\dots,V_{p}]$:
\begin{equation*}\label{notationequation}
  D^{p}v(x)[V_{1},\dots,V_{p}] =  \sum_{i_1,\dots,i_p=1}^d\frac{\partial^{p} }{\partial x^{i_{1}}\dots\partial x^{i_{p}}} v(x) \, \prod\limits_{j=1}^{p}V_{j}^{i_{j}}.
\end{equation*}

It is not difficult to see that in Algorithm~\ref{algorithm2.1}, under Assumptions~\ref{as:3}-\ref{as:2}, $r_{k} := \dist(X_{k}^{'},X_{k}^{\pi}) = \mathcal{O}(h^{1/2})$ whenever $X_{k+1}^{'}\in \bar{G}^{c}$. Under Assumption \ref{as:3}, we can introduce a $d-1$-dimensional surface outside $G$ parallel to $\partial G$  whose distance to $\partial G$ is $r$ which is large enough so that for all $k$, we have $r_{k} \leq r $ and $r=\mathcal{O}(h^{1/2})$. We denote this surface as $\mathbb S_{-r}$, and we denote the layer between the two surfaces $\partial  G$ and $\mathbb S_{-r}$ as $G_{-r}$ and also $Q_{-r}:=[T_{0},T)\times G_{-r}$.
\subsubsection{One-step approximation}\label{sec4.1}
In this subsection, we will prove two lemmas. Lemma~\ref{lemma2.2} estimates the error of the one-step approximation in moving from $X_{k}$ at time layer $t_{k}$ to $X_{k+1}^{'}$ at time layer $t_{k+1}$, i.e. it is about the local weak error of the Euler scheme.  Lemma~\ref{lemma2.3} estimates the error of the one-step approximation at the same time layer $t_{k+1}$ in moving from $X_{k+1}^{'} $ to $X_{k+1}$ given that $X_{k+1}^{'}$ goes outside $\bar{G}$.

We will use the following result. It is known \cite[Proposition 1.17]{19} that under Assumption~\ref{as:3} the solution $u(t,x) \in C^{2,4}(\bar{Q})$ can be extended to  a function $u(t,x) \in C^{2,4}(\bar{Q}\cup\bar{Q}_{-r})$. This extension of $u(t,x)$ and its derivatives will be used in proofs where we need to expand $u(t,x)$ around $x^{\pi} \in \partial G$ when $ x \in \bar{G}^{c}$. We pay attention to the fact that $u(t,x)$ and its derivatives are uniformly bounded for $(t,x)\in \bar{Q} \cup \bar{Q}_{-r}$. 

\begin{lemma} \label{lemma2.2}
Under Assumptions \ref{as:3}-\ref{as:2} and \ref{a01}-\ref{ass4}, the one-step error of Algorithm~\ref{algorithm3.1} associated with moving from $X_k$ to $X^{'}_{k+1}$ is estimated as
\begin{equation*}
    \Big|\mathbb{E} \Big( u_{k+1}^{\prime}Y_{k+1}^{'} + Z_{k+1}^{'} - (u_{k}Y_{k} + Z_{k})\big| X_{k},Y_{k},Z_{k}\Big)\Big| \leq CY_{k}h^{2}, \;\;\;\; k=0,1,\dots,N-1,
\end{equation*}
where  $C$ is a positive constant independent of $h$.
\end{lemma}
\begin{proof}
This is a standard result (see e.g. \cite{49}) for the weak Euler approximation (\ref{eq:11}), (\ref{mdy1})-(\ref{mdz1}) used in Algorithm~\ref{algorithm3.1} but we include its proof here as we will refer to it later (in Section~\ref{section6}), where we will need to take into account dependence of the error on time.

We expand $u_{k+1}^{\prime} $  firstly around $(t_{k},X_{k+1}^{'})$ and then around $(t_{k},X_{k}) $:
\begin{align*}
u_{k+1}^{\prime}
 = u_{k} + h\frac{\partial u_{k}}{\partial t} + h(b_{k}\cdot\nabla)u_{k} + \frac{h}{2}(a_{k}:\nabla\nabla)u_{k} +  R_{1,k+1} + R_{2,k+1},
\end{align*}
so that
\begin{align*}
    R_{1,k+1} &= (h^{2}/2)D^{2}u_{k}[b_{k},b_{k}] 
    + (h^{3}/6)D^{3}u_{k}[b_{k},b_{k},b_{k}]  + (h^{2}/2)D^{3}u_{k}[b_{k},\sigma_{k}\xi_{k+1},\sigma_{k}\xi_{k+1}] \\
    &\;\;  + (1/24)D^{4}u(t_{k},X_{r})[\delta_{k+1},\delta_{k+1},\delta_{k+1},\delta_{k+1}]  + h^{2}(b_{k}\cdot\nabla)D_{t}u_{k}
    \\ & \;\; + (h/2)D_{t}D^{2}u(t_{k},X_{r}^{'})[\delta_{k+1},\delta_{k+1}] + (h^{2}/2)D_{t}^{2}u(t_{r}, X_{k+1}^{'}),\\
R_{2,k+1} &=  h^{1/2}(\sigma_{k}\xi_{k+1}\cdot\nabla)u_{k} + (h/2)\big( D^{2}u_{k}[\sigma_{k}\xi_{k+1},\sigma_{k}\xi_{k+1}] - (a_{k}:\nabla\nabla)u_{k}\big)\\ & \;\; + h^{3/2}D^{2}u_{k}[b_{k},\sigma_{k}\xi_{k+1}] + ( h^{3/2}/6)D^{3}u_{k}[\sigma_{k}\xi_{k+1},\sigma_{k}\xi_{k+1},\sigma_{k}\xi_{k+1}] \\ & \;\; + (h^{5/2}/2)D^{3}u_{k}[b_{k},b_{k},\sigma_{k}\xi_{k+1}]   + h^{3/2}(\sigma_{k}\xi_{k+1}\cdot\nabla)D_{t}u_{k} , 
\end{align*}
where $D_{t}^{j} = \frac{\partial^{j}}{\partial t^{j}}$ with $j=1,2$, $\delta_{k+1} = hb_{k} + h^{1/2}\sigma_{k}\xi_{k+1} $, $t_{r} = t_{k}+\alpha_{1}h $, $X_{r}^{'}=X_{k} + \alpha_{2}\delta_{k+1}$ with some $\alpha_{1}$, $\alpha_{2} \in (0,1)$. Notice that
\begin{align*}
\mathbb{E}(\xi^{i}_{k+1}|X_{k},Y_{k},Z_{k}) &= 0,& 
\mathbb{E}(\xi^{i}_{k+1}\xi^{j}_{k+1}|X_{k},Y_{k},Z_{k}) &= 0,  \;\;\; i \neq j,\\ \mathbb{E}(\xi_{k+1}^{i}\xi_{k+1}^{j}\xi_{k+1}^{m})|X_{k},Y_{k},Z_{k}) &= 0, & \mathbb{E}((\xi_{k+1}^{i})^{2}|X_{k},Y_{k},Z_{k}) &= 1,
\end{align*}
    where $ i,j,m = 1,\dots d$, then  it is not difficult to deduce that  $\mathbb{E}(R_{2,k+1}|X_{k},Y_{k},Z_{k}) = 0$, $k = 0,\dots, N-1$. Further, again using moments of $\xi_{k+1}$, noticing $\mathbb{E}(\delta_{k+1})^2 = \mathcal{O}(h)$, $ k = 0,\dots,N-1$, and recalling that the function $u(t,x)$ is uniformly bounded for all $(t,x) \in [T_{0},T]\times \bar{G}\cup\bar{G}_{-r}$, we get
  $  |\mathbb{E}(R_{1,k+1}|X_{k},Y_{k},Z_{k})|\leq Ch^{2}$,   $k=0,\dots ,N-1$. Therefore, using equation (\ref{eq:5}), we obtain the desired bound as
\begin{align*}
   &\Big| \mathbb{E}\Big(u_{k+1}^{\prime}Y_{k+1}^{'}  + Z_{k+1}^{'} - (u_{k}Y_{k} + Z_{k}) \big|X_{k},Y_{k},Z_{k} \Big)\Big|  \\
   &= \Big| \mathbb{E}\Big(u_{k+1}^{\prime} \big(Y_{k} + hc_{k}Y_{k} \big) - u_{k}Y_{k} + Z_{k+1}^{'} - Z_{k}\big| X_{k},Y_{k},Z_{k} \Big)  \Big| \\ & \leq \Big|  \Big(u_{k} + h\frac{\partial u}{\partial t}(t_{k},X_{k}) + h(b_{k}\cdot\nabla)u_{k} + h\frac{(a_{k}:\nabla\nabla)}{2}u_{k} \Big)\big( Y_{k} + hc_{k}Y_{k}\big)  -u_{k}Y_{k}
   + hg_{k}Y_{k}\Big|  \\ & \;\;\; + \Big|  \mathbb{E}\Big(R_{1,k+1} + R_{2,k+1}\big| X_{k},Y_{k},Z_{k}\Big)( Y_{k} + hc_{k}Y_{k})\Big|    \leq CY_{k}h^{2}.
\end{align*}
\end{proof}

\begin{lemma} \label{lemma2.3}
Under Assumptions \ref{as:3}-\ref{as:2} and \ref{a01}-\ref{ass4}, the one-step error of Algorithm~\ref{algorithm3.1} near the boundary is estimated  for all $ k= 0,\dots,N-1$  as
\begin{equation}
    \Big| u_{k+1}Y_{k+1} + Z_{k+1} - (u_{k+1}^{\prime}Y_{k+1}^{'} + Z_{k+1}^{'})\Big| \leq CY_{k}r_{k+1}h\rf{I_{\bar{G}^{c}}(X_{k+1}^{'})\;\;\;\; a.s., }
    \label{eq:onestepbou}
\end{equation}
where $C$ is a positive constant independent of $h$.
\end{lemma}
\begin{proof} \rf{It is obvious that the left-hand side of (\ref{eq:onestepbou}) is equal to $0$ when $X_{k+1}^{'} \in \bar G$.

Let us consider the error on the event $\{X_{k+1}^{'} \in \bar{G}^{c} \}$.} We have \rf{
\begin{align*}
\big| u_{k+1}Y_{k+1} + Z_{k+1}  - (u_{k+1}^{\prime}Y_{k+1}^{'} + Z_{k+1}^{'})\big|  \\
= \big|\big(u_{k+1} - u_{k+1}^{\prime}\big)Y_{k+1} + u_{k+1}^{\prime} \big( Y_{k+1} - Y_{k+1}^{'}\big) + Z_{k+1} - Z_{k+1}^{'}\big|,
\end{align*}}
where we have three errors to analyse: $\Gamma_{1} :=\rf{  (u_{k+1} - u_{k+1}^{\prime})Y_{k+1}}$, $\Gamma_{2} :=\rf{ u_{k+1}^{\prime}(Y_{k+1}-Y_{k+1}^{'}) }$, and $\Gamma_{3} := \rf{Z_{k+1} - Z_{k+1}^{'}}$.

Taylor expansion of $u_{k+1}$ and $u_{k+1}^{\prime}$ around the boundary point $X_{k+1}^{\pi}$ of the same time layer $t_{k+1}$ gives
 \begin{align}
 u_{k+1} &= u_{k+1}^{\pi} + r_{k+1}(\nu\cdot\nabla)u_{k+1}^{\pi} + (r_{k+1}^{2}/2)D^{2}u_{k+1}^{\pi}[\nu,\nu] + R_{3,k+1},\label{2.26} \\  
 u_{k+1}^{\prime} &= u_{k+1}^{\pi} - r_{k+1}(\nu\cdot\nabla)u_{k+1}^{\pi} + (r_{k+1}^{2}/2)D^{2}u_{k+1}^{\pi}[\nu,\nu] + R_{4,k+1}\label{2.27},
 \end{align}
 where $\nu$ is evaluated at $X_{k+1}^{\pi}$ and for $j =3,4$,
 \begin{equation*}
   R_{j,k+1} = (-1)^{j+1}(r_{k+1}^{3}/6)D^{3}u(t_{k+1}, X_{k+1}^{\pi} + (-1)^{j+1}\alpha_{j} r_{k+1}\nu)[\nu,\nu,\nu],\;\;\;\alpha_{j} \in (0,1).
 \end{equation*}
  We notice that $(r_{k+1})^{m} \leq Ch^{m/2}$, $k=0,\dots,N-1$, for any $m \geq 1$. Recall $u(t,x) \in C^{2,4}( \bar{Q}\cup\bar{Q}_{-r})$ which implies that
 \begin{align*}
    \rf{ | R_{j,k+1}|} &\leq Cr_{k+1}^{3}, \;\;\;\;\;\;\; j =3,4,
 \end{align*}
 \rf{where $C>0$ is a nonrandom constant independent of $h$.}

 Using the expansions (\ref{2.26})-(\ref{2.27}) and substituting $Y_{k+1}$ from (\ref{eq:16}), $\Gamma_{1}$ becomes
\begin{flalign*}
\Gamma_{1} &= \rf{ 2r_{k+1}(\nu\cdot\nabla)u_{k+1}^{\pi}\big(Y_{k} + hc_{k}Y_{k} +2r_{k+1}\gamma_{k+1}^{\pi}Y_{k} + 2r_{k+1}^{2}(\gamma_{k+1}^{\pi})^{2}Y_{k}\big) + (R_{3,k+1}} \\
& \;\;\;\;- R_{4,k+1})Y_{k+1}  \\
&\;= \rf{2r_{k+1}Y_{k}(\nu\cdot\nabla)u_{k+1}^{\pi} +  4r_{k+1}^{2}\gamma_{k+1}^{\pi}Y_{k}(\nu\cdot\nabla)u_{k+1}^{\pi} + 2hr_{k+1}c_{k}Y_{k}(\nu\cdot\nabla)u_{k+1}^{\pi} }\\
& \;\;\;\;  + \rf{4r_{k+1}^{3}(\gamma_{k+1}^{\pi})^{2}Y_{k}(\nu\cdot\nabla)u_{k+1}^{\pi}  +  \big(R_{3,k+1} - R_{4,k+1}\big)Y_{k+1}}.
\end{flalign*}
Similarly $\Gamma_{2}$ gives
\begin{flalign*}
    \Gamma_{2} &= \rf{ \big(u_{k+1}^{\pi} - r_{k+1}(\nu\cdot\nabla)u_{k+1}^{\pi}\big)\big( 2r_{k+1}\gamma_{k+1}^{\pi}Y_{k} + 2r_{k+1}^{2}(\gamma_{k+1}^{\pi})^{2}Y_{k}\big) }\\
    & \;\;\;\; + \rf{ \Big(\frac{r_{k+1}^{2}}{2}D^{2}u_{k+1}^{\pi}[\nu,\nu]+R_{4,k+1}\Big)(Y_{k+1} - Y_{k+1}^{'})}\\
    & \;= \rf{2r_{k+1}\gamma_{k+1}^{\pi}Y_{k}u_{k+1}^{\pi} - 2r_{k+1}^{2}\gamma_{k+1}^{\pi}Y_{k}(\nu\cdot\nabla)u_{k+1}^{\pi} + 2r_{k+1}^{2}(\gamma^{\pi}_{k+1})^{2}Y_{k}u_{k+1}^{\pi}}  \\
    & \;\;\;\;-\rf{2r_{k+1}^{3}(\gamma_{k+1}^{\pi})^{2}Y_{k}(\nu\cdot\nabla)u_{k+1}^{\pi}
    +\Big(\frac{r_{k+1}^{2}}{2}D^{2}u_{k+1}^{\pi}[\nu,\nu]  + R_{4,k+1}\Big)(Y_{k+1}-Y_{k+1}^{'})},
\end{flalign*}
and, using the value of $Z_{k+1}$ and $Z_{k+1}^{'}$, $\Gamma_{3}$ becomes
\begin{flalign*}
\Gamma_{3} & =\rf{ - 2r_{k+1}\psi_{k+1}^{\pi}Y_{k} - 2r_{k+1}^{2}\gamma_{k+1}^{\pi}\psi_{k+1}^{\pi}Y_{k}}.
\end{flalign*}
Combining $\Gamma_{1}, \Gamma_{2}, \Gamma_{3}$ and using the boundary condition (\ref{eq:4}), we obtain
\begin{align*}
    \big| \Gamma_{1}& +  \Gamma_{2} +  \Gamma_{3}\big|  \leq  \rf{\Big|2hr_{k+1}c_{k}Y_{k}(\nu\cdot\nabla)u_{k+1}^{\pi} \Big|} \\
    & \;\;\;\; + \rf{\Big|2r_{k+1}^{3}(\gamma_{k+1}^{\pi})^{2}Y_{k}(\nu\cdot\nabla)u_{k+1}^{\pi} + \frac{r_{k+1}^{2}}{2}D^{2}u_{k+1}^{\pi}[\nu,\nu](Y_{k+1}-Y_{k+1}^{'})\Big|}\\
    & \;\;\;\;\;  + \rf{\big| R_{3,k+1}Y_{k+1} - R_{4,k+1}Y_{k+1}^{'}\big|
    \leq CY_{k}hr_{k+1}},
\end{align*}
\rf{which gives the error estimate in the case $\{X_{k+1}^{'} \in \bar{G}^{c} \}$}.
\end{proof}

 \subsubsection {Lemma on the number of steps when $X^{'}_{k} \notin \bar G$} \label{section2.4}
  Consider the Markov chain $(t_{k},X^{'}_{k})$ generated by  Algorithm~\ref{algorithm3.1}. Recall that $X_{k}^{'}$ can take values outside $\bar{G}$. If $X_{k}^{'} \in \bar{G}^{c}$, to calculate $X_{k+1}^{'}$ we first determine $X_{k}$ according to (\ref{eq:15}) and then simulate $\xi_{k+1}$ to find $X_{k+1}^{'}$ as in (\ref{eq:11}). %
 
\rf{Let $\rs{P_h=}P$ be the one-step transition operator for the  Markov chain $(t_{k},X_{k}^{'})$}, $\rs{k=0,\ldots,N}$:
\begin{equation*}
   \rs{ (P_h V)(t,x) = PV(t,x) := \mathbb{E}\big[ \big . V(t+h, X^{'}_{1}) \big | t_{0}=t, X_{0}^{'}=x\big],}
\end{equation*}
\rf{where $(t,x)$ is an arbitrary point in $[T_{0}, T)\times G$
\rs{and $V(t,x)$ is a function from $[T_{0}, T]\times \big(\bar G \cup \bar{G}_{-r}\big)$ to $\mathbb R$.}}
\rs{We note that
\begin{equation*}
\mathbb{E}\big[ \big . V(t_{k+1}, X^{'}_{k+1}) \big | X_{k}^{'}=x\big]
=PV(t_k,x).
\end{equation*}
}

Consider the boundary value problem associated with the Markov chain $(t_{k},X_{k}^{'})$:
\begin{align}
    q(x) PV(t,x) - V(t,x) &= -g(t,x),  &(t,x) \in [T_{0}, T-h]\times \big(\bar{G}\cup \bar{G}_{-r}\big), \label{eq:18} \\
    V(T,x) &= 0,       &x \in  \big(\bar{G} \cup \bar{G}_{-r}\big),\label{eq:19}
\end{align}
where $g(t,x) \geq 0$ and $q(x) > 0$. The solution to this problem \rs{starting from $(t,x)=(t_k,x)$} is given by \cite{20, 46, 49}:\rs{
\begin{equation}
    V(t_k,x) =  \mathbb{E}\bigg[ \bigg . \sum\limits_{i=k}^{N-1}g(t_{i},X_{i}^{'})\prod\limits_{j=k}^{i-1}q(X_{j}^{'}) \bigg | X_{k}^{'}=x\bigg]. \label{eq:20}
\end{equation}}

The next lemma is related to an estimate of the number of steps which the Markov chain $X_{k}^{'}$ spends in the layer $G_{-r}$ that lies outside the $\bar{G}$. It is used in proving the main convergence theorem  (Theorem \ref{conp}). 
\begin{lemma} \label{bl}
Under Assumptions \ref{as:3} - \ref{as:2}, for \rf{any} constant $K > 0$ and for sufficiently small $h$, the following inequality holds
\begin{equation*}
    \mathbb{E}\bigg(\sum\limits_{k=0}^{N-1}r_{k}I_{G_{-r}}(X_{k}^{'})\prod_{i=0}^{k-1}\big(1+Kr_{i}I_{G_{-r}}(X_{i}^{'})\big) \bigg)\leq C,
\end{equation*}
where $C$ is a positive constant independent of $h$.
\end{lemma}
\begin{proof}
\rf{Let $r_{0}(x)$ be the distance of $x$ from boundary $\partial G$. If we take $g(t,x) = r_{0}(x)I_{G_{-r}}(x)$ and $ q(x) = (1+Kr_{0}(x){I_{G_{-r}}(x)})  $ in  (\ref{eq:18}), then the solution to (\ref{eq:18})-(\ref{eq:19}) is (cf. (\ref{eq:20})):
\begin{equation*}
v(\rs{t_0},x) = \mathbb{E}\bigg( \bigg . \sum\limits_{k=0}^{N-1}r_{k}I_{G_{-r}}(X_{k}^{'})\prod_{i=0}^{k-1}(1+Kr_{i}I_{G_{-r}}(X_{i}^{'})) \bigg | \rs{ X_{0}^{'}=x} \bigg),
\end{equation*}
where $r_{k} = \dist(X_{k}^{'},\partial G)$ as before.
 If we can find a solution $V(t,x)$ to (\ref{eq:18})-(\ref{eq:19}) with a function $g(t,x)$ such that it satisfies the inequality
\begin{equation}
    g(t,x) \geq r_{0}(x) I_{G_{-r}}(x), \label{eq:21}
\end{equation}
for all $(t,x) \in [T_{0},T-h]\times(\bar{G}\cup \bar{G}_{-r})$, then we have
\begin{equation*}
    v(\rs{t_0},x) \leq V(\rs{t_0},x).
\end{equation*}
Introduce the function
\begin{equation}
   w(x) =
    \begin{cases}
    0, & x \in \bar{G}\textbackslash \bar{G}_{l}, \label{wx}\\
    \dist^{2}(x,\mathbb{S}_{l}), & x \in \bar{G}_{l}\cup \bar{G}_{-r}.
    \end{cases}
\end{equation}
Since Algorithm~\ref{algorithm3.1} takes steps according to (\ref{eq:11})-(\ref{eq:15}), we can write
 \begin{equation*}
 \Delta X^{'} := X^{'}-x = 2r_{0}(x)\nu I_{G_{-r}}(x) + \delta,
 \end{equation*}
  where the normal $\nu = \nu(x^{\pi})$ is evaluated at the projection of $x$ on $\partial G$, i.e. $  x^{\pi}$, and $\delta = b(t,x+2r_{0}(x)I_{G_{-r}}(x)\nu )h + \sigma(t,x+2r_{0}(x)I_{G_{-r}}(x)\nu)h^{1/2}\xi$.
Note that
 \begin{equation}
 |\mathbb{E}(\delta)| = \mathcal{O}(h), \,\,\,\,\, |\delta |^{2} = \mathcal{O}(h), \,\,\,\,\, |\Delta X^{'} |^{2} = \mathcal{O}(h).
 \label{eq:deltaprop}
 \end{equation}

 We have the following approximation \rs{when $x \in \bar{G_{l}} \cup G_{-r} $} (cf. \cite{1,49}):
\begin{equation}
      \dist(x+\Delta X^{'},\mathbb{S}_{l}) = \dist(x,\mathbb{S}_{l})+ \Bigg(\Delta X^{'}\cdot \frac{x-x^{\pi}_{l}}{| x-x^{\pi}_{l}|}\Bigg)+ \mathcal{O}(h), \label{eq:22}
   \end{equation}
   where $x^{\pi}_{l}$ is the projection of $x$ onto the surface $\mathbb{S}_{l}$.
 Hence, 
  we have 
   \begin{align}
   \mathbb{E}\big(\dist(x+\Delta X^{'},\mathbb{S}_{l})\big)^2  &= \mathbb{E} \bigg(\dist(x,\mathbb{S}_{l}) + \bigg(\Delta X^{'}\cdot\frac{x-x_{l}^{\pi}}{| x-x_{l}^{\pi}|}\bigg)+\mathcal{O}(h)\bigg)^{2} \nonumber  \\
   &= \Big(\dist(x,\mathbb{S}_{l})\Big)^{2} +  2\dist(x,\mathbb{S}_{l})\mathbb{E}\bigg(\Delta X^{'}\cdot \frac{x-x_{l}^{\pi}}{| x- x_{l}^{\pi}|}\bigg)
   + \mathcal{O}(h).\label{eq:23}
   \end{align}
 Also, define another function $U(t,x) $ as
\begin{equation}\label{utx}
    U(t,x) =
    \begin{cases}
    0, & (t,x) \in \{T\}\times \big(\bar{G} \cup \bar{G}_{-r}\big),\\
    e^{K_{1}(T-t)}e^{K_{2}w(x)}, & (t,x) \in [T_{0}, T-h] \times \big(\bar{G} \cup \bar{G}_{-r}\big),
    \end{cases}
\end{equation}
where $K_{1}$ and $K_{2}$ are positive constants which choice will be discussed later in the proof.

Applying the one step operator $P$ to $U$, for $t<T$ we get
\begin{equation*}
    PU(t,x) = e^{K_{1}(T-t-h)}\mathbb{E}\Big(e^{K_{2}w(X_{1}^{'})}\Big).
\end{equation*}
We note that if $ x $ as well as $ X_{1}^{'} \in \bar{G}_{l}\cup \bar{G}_{-r}$, then due to (\ref{eq:22}) and (\ref{eq:23}) we ascertain
\begin{align}
    PU(t,x) = e^{K_{1}(T-t)}e^{K_{2}w(x)} &\big(1 - K_{1}h + \mathcal{O}(h^{2})\big) \nonumber \\ & \times \mathbb{E}\Bigg( 1
    +  2K_{2}\dist(x,\mathbb{S}_{l})\bigg( \Delta X^{'}\cdot \frac{x-x^{\pi}_{l}}{| x-x^{\pi}_{l}|}\bigg) + \mathcal{O}(h)\Bigg),\label{eq:24}
\end{align}
where $\mathcal{O}(h^{2})$ depends on $K_{1}$ and $\mathcal{O}(h)$ depends on $K_{2}$.

Thereafter, we calculate $(1+Kr_{0}(x){I_{G_{-r}}(x)})PU(t,x)-U(t,x)$ at points $(t,x)$ lying in different regions identified by four different cases discussed below with the aim of finding $g(t,x)$ satisfying the inequality (\ref{eq:21}). Introduce the region $ S_{h} = \{ x | \dist(x,\mathbb{S}_{l}) < K_{3}h^{1/2}\} $ where $K_{3} $ is chosen so that for $x \in G_{l}\backslash S_{h}$, in one step transition, any of the $2^d$ realizations of $X_{1}^{'}$ cannot cross $\mathbb{S}_{l}$.
We also note that for $x \in S_{h}$, in one step transition $X_{1}^{'}$ may or may not cross the surface $\mathbb{S}_{l}$.  In the first case, i.e. in Case 1, we discuss the scenario when $x \in \bar{G}_{-r}$. In Case 2, we choose $x \in G_{l}\backslash S_{h}$ so that $X_{1}^{'}$ cannot cross $\mathbb{S}_{l} $ and remains in $G_{l}\cup G_{-r} $. In Case 3, we take $x \in \bar{G}\backslash(\bar{G_{l}}\cup S_{h})$  so that all realizations of $X_{1}^{'}$ also belong to $\bar{G}\backslash\bar{G}_{l}$. We examine the scenario when $x \in S_{h}$ in Case 4.
\begin{description}
   \item[Case 1:] $x \in \bar{G}_{-r}$. \\
 In this case $\dist(x,\mathbb{S}_{l}) = l + r_{0}(x)$.
 Since  $\mathbb{S}_{l} $ is parallel to $\partial G$, $\nu(x^{\pi})=\nu(x^{\pi}_l)$.
 Further, using (\ref{eq:24}), (\ref{eq:deltaprop}), $r_{0}^{2}(x) = \mathcal{O}(h)$, and $\Big(\nu(x^{\pi}_l)\cdot \frac{x-x_{l}^{\pi}}{|x-x_{l}^{\pi}|}\Big) = -1 $, we get
 \begin{align*}
      &(1+Kr_{0}(x))PU(t,x) -U(t,x) = e^{K_{1}(T-t)}e^{K_{2}w(x)}\Big(\big(1 + Kr_{0}(x)\big)\big(1-K_{1}h + \mathcal{O}(h^{2})\big)\\ & \;\;\;\;\times \Big( 1 + 4K_{2}r_{0}(x)l\Big(\nu\cdot \frac{x-x_{l}^{\pi}}{| x-x_{l}^{\pi}|}\Big) + \mathcal{O}(h)\Big)  - 1\Big) = \Big(-\underbrace{\big(4K_{2}r_{0}(x)l  - Kr_{0}(x)\big)}_{{1}} \\
      & -\underbrace{\big(K_{1} h - \beta_1(K_{2}) \mathcal{O}(h) + \beta_2(K_1,K_2)\mathcal{O}(h^{3/2}) \big)}_{{2}}\Big)e^{K_{1}(T-t)}e^{K_{2}w(x)},
\end{align*}
where $|\mathcal{O}(h^k)| \leq Ch^k$ with $C>0$ being independent of $h$, $K_1$ and $K_2$,  $\beta_1(K_2)$ is a function of $K_2$, and $\beta_2(K_1,K_2)$ is a function of $K_1$ and $K_2$.

\item[Case 2:]  $ x \in G_{l}\backslash S_{h} $ and hence all realizations of $X_{1}^{'} \in G_{l}\cup G_{-r}$.\\ 
In this case $\Delta X^{'} = \delta$, where $\delta = b(t,x)h + \sigma(t,x)h^{1/2}\xi$. Using (\ref{eq:deltaprop}) and (\ref{eq:24}), we obtain
\begin{align*}
    &PU(t,x)-U(t,x) = e^{K_{1}(T-t)}e^{K_{2}w(x)} \big(1 - K_{1}h + \mathcal{O}(h^{2})\big)\\ & \times \bigg( 1 +  2K_{2}\bigg(b(t,x)\cdot \frac{x-x_{l}^{\pi}}{| x-x_{l}^{\pi}|}\bigg)\dist(x,\mathbb{S}_{l})h  + \mathcal{O}(h)  \bigg)  - e^{K_{1}(T-t)}e^{K_{2}w(x)} \\
    & = \big(-h K_{1} + \beta_1(K_2) \mathcal{O}(h)+ \beta_2(K_1,K_2) \mathcal{O}(h^2)\big)e^{K_{1}(T-t)}e^{K_{2}w(x)},
\end{align*}
where $|\mathcal{O}(h^k)| \leq Ch^{k}$ with $C>0$ being independent of $h$, $K_1$ and $K_2$, $\beta_1(K_2)$ is a function of $K_2$, and $\beta_2(K_1,K_2)$ is a function of $K_1$ and $K_2$ (note that the functions $\beta_i$ here are different than the ones in Case 1).
\item[Case 3:] $x \in \bar{G}\backslash(G_{l}\cup S_{h})$ and all realizations of $ X_{1}^{'} \in \bar{G}\backslash\bar{G_{l}} $. \\
We have,
 $
 PU(t,x)-U(t,x)  = e^{K_{1}(T-t-h)} - e^{K_{1}(T-t)}
 = \big(-K_{1}h + \mathcal{O}(h^{2})\big)e^{K_{1}(T-t)}.
 $
 \item[Case 4:] $x \in S_{h}$. \\
 It can be observed that $ w(x) = \mathcal{O}(h)$ and $ w(x+\delta) = \mathcal{O}(h)$. Hence
 \begin{align*}
 PU(t,x)-U(t,x)  =  \big(-K_{1}h +K_2 \mathcal{O}(h)+\beta_2(K_1,K_2) \mathcal{O}(h^2)\big)e^{K_{1}(T-t)},
 \end{align*}
where $|\mathcal{O}(h^k)| \leq Ch^k$ with $C>0$ being independent of $h$, $K_1$ and $K_2$ and $\beta_2(K_1,K_2)$ is a function of $K_1$ and $K_2$.
\end{description}
Now firstly we analyze Case 1. We take $K_{2} > \frac{K+1}{4l}$ which ensures that  term $1$ is always greater than $r_{0}(x)$. Then we choose $K_{1}$ in a manner that not only term $2$ in Case 1 is positive but also $f(t,x) = -\big((1+Kr_{0}(x)I_{G_{-r}}(x))PU(t,x) -U(t,x)\big)$ is positive in Case 2 as well as in Case 4. It is evident from second and fourth case that such a choice of $K_{1}$ is dependent on $K_{2}$.
As one can see, Case 3 trivially satisfies the condition that $ f(t,x) = -\big(PU(t,x)-U(t,x)\big)$ is positive. We take $g(t,x) = f(t,x)$ which is greater than $r_{0}(x)$ in $G_{-r}$. As can be  easily observed, we have constructed a function $V(t,x) = U(t,x)$  which is a solution of (\ref{eq:18})-(\ref{eq:19}) with $q(x) = (1+Kr_{0}(x)I_{G_{-r}}(x))$ and $g(t,x)  \geq  I_{G_{-r}}(x)r_{0}(x)$.
Therefore, $v(\rs{t_0},x) \leq V(\rs{t_0},x)$ and the lemma is proved. }
\end{proof}

\begin{corollary} \label{corollary1}
Under Assumptions \ref{as:3} - \ref{as:2}, for \rf{any} constant $K>0$ the following inequalities hold
\begin{equation}\label{neweq3.29}
    \mathbb{E}\bigg(\sum\limits_{k=1}^{N}\prod_{i=0}^{k-1}(1+Kr_{i}I_{G_{-r}}(X_{i}^{'}))\bigg) \leq \frac{C}{h}
\end{equation}
\rs{and also
\begin{align}\label{neweq3.30}
    \mathbb{E}\bigg( \prod_{i=0}^{N-1}(1+Kr_{i}I_{G_{-r}}(X_{i}^{'}))  \bigg) \leq C,
\end{align}}
where $C$ is a positive constant independent of $h$.
\end{corollary}
\begin{proof}
Again consider the boundary value problem (\ref{eq:18})-(\ref{eq:19}) related to the Markov chain $(t_{k},X_{k}^{'})$ with the solution given by (\ref{eq:20}). If we chose $g(t,x) = 1$ and $q(x) = (1+r_{0}{I_{\bar{G}_{-r}}(x)})$ then the solution
 of the problem is
\begin{equation}\label{3.27}
     v(\rs{t_0},x) = \mathbb{E}\Bigg( \bigg . \sum\limits_{k=0}^{N-1}\prod_{i=0}^{k-1}\big(1+Kr_{i}I_{G_{-r}}(X_{i}^{'})\big) \bigg | \rs{ X_{0}^{'}=x} \Bigg).
\end{equation}
If we can find a solution to (\ref{eq:18})-(\ref{eq:19}) with a function $g(t,x) \geq 1$ then $ v(t,x) \leq V(t,x)$. Keeping this in mind, one can easily check if we take $g(t,x) = f(t,x)/h $ where $ f(t,x) $ is constructed according to the cases discussed in Lemma \ref{bl} then our aim to get $g(t,x) \geq 1$ is fulfilled. In turn we obtain $ v(\rs{t_0},x) \leq V(\rs{t_0},x)$ where $ V(t,x) = U(t,x)/h$.

We now show that $\mathbb{E}\big(\prod_{i=0}^{N-1}(1+\rs{K}r_{i}I_{G_{-r}}(X_{i}^{'}))\big) \leq C$, where $C $ is independent of $h$. Note that
\begin{align*}
    &\mathbb{E}\bigg(\sum\limits_{k=0}^{N-1}r_{k}I_{G_{-r}}(X_{k}^{'})\prod_{i=0}^{k-1}(1+ Kr_{i}I_{G_{-r}}(X_{i}^{'}))\bigg)
     = \frac{1}{K}\mathbb{E}\bigg(\sum\limits_{k=0}^{N-1}\Big( \prod\limits_{i=0}^{k}(1+Kr_{i}I_{G_{-r}}(X_{i}^{'})) \\ & - \prod\limits_{i=0}^{k-1}(1+Kr_{i}I_{G_{-r}}(X_{i}^{'}))\Big)\bigg)
     =\frac{1}{K}\mathbb{E}\bigg(\prod\limits_{i=0}^{N-1}\big(1+Kr_{i}I_{G_{-r}}(X_{i}^{'})\big) - 1\bigg),
\end{align*}
which on using Lemma~\ref{bl} gives \rs{(\ref{neweq3.30}). We have $v(t_0,x) \leq C/h$, where $v(t_0,x)$ is from (\ref{3.27}), which together with (\ref{neweq3.30}) yields  (\ref{neweq3.29})}.
\end{proof}

\subsubsection{Convergence theorem\label{ssec:thm31prf}}
  In Lemma \ref{lemma2.2} we have shown that the order of the one-step approximation in moving from ($X_{k}$, $Y_{k}$, $Z_{k}$) to ($X_{k+1}^{'}$, $Y^{'}_{k+1}$, $Z_{k+1}^{'}$) is $\mathcal{O}(h^{2}) $ and in Lemma \ref{lemma2.3} we have obtained the order $\mathcal{O}(hr_{k+1})$ of the one-step approximation in the case $X_{k+1}^{'} \notin \bar{G}$. Now, we combine these two lemmas along with Lemma \ref{bl} and Corollary \ref{corollary1} to obtain the weak order of convergence of Algorithm \ref{algorithm3.1}.

\begin{proof}[Proof of Theorem \ref{conp}]
We have
\begin{align*}
    &\big|  \mathbb{E}\big(\varphi(X_{N})Y_{N} + Z_{N}\big)  - u(t_{0}, X_{0})Y_{0} \big|
     = \Bigg| \mathbb{E}\Bigg( \sum\limits_{k=0}^{N-1}\Big(u_{k+1} Y_{k+1} + Z_{k+1} - \big(u_{k}Y_{k} + Z_{k}\big)\Big) \Bigg)\Bigg| \\
    & = \Bigg| \mathbb{E}\Bigg( \sum\limits_{k=0}^{N-1}u_{k+1}Y_{k+1} - u_{k+1}^{\prime}Y_{k+1}^{'} + u_{k+1}^{\prime}Y_{k+1}^{'} - u_{k}Y_{k} +  Z_{k+1} - Z_{k+1}^{'} + Z_{k+1}^{'} - Z_{k}\Bigg)\Bigg|.
\end{align*}
When $X_{k+1}^{'} \in \bar{G}$, we have $u_{k+1} = u_{k+1}^{\prime}$, $Y_{k+1}=Y_{k+1}^{'}$ and $Z_{k+1} = Z_{k+1}^{'}$, therefore
\begin{align}
    &\big| \mathbb{E}\big(\varphi(X_{N})Y_{N} + Z_{N}\big) - u(t_{0}, X_{0})Y_{0} \big| \leq \Bigg| \mathbb{E}\Bigg(\sum\limits_{k=0}^{N-1}\Big( u_{k+1}Y_{k+1} - u_{k+1}^{\prime}Y_{k+1}^{'} + Z_{k+1} \nonumber\\ & \;\;\; - Z_{k+1}^{'}\Big)I_{G_{-r}}\big(X_{k+1}^{'}\big) \Bigg)\Bigg| \nonumber  + \Bigg|\mathbb{E}\Bigg( \sum\limits_{k=0}^{N-1}\Big(u_{k+1}^{\prime}Y_{k+1}^{'} - u_{k}Y_{k} + Z_{k+1}^{'}-Z_{k}\Big)\Bigg)\Bigg|
   \nonumber \\ & \leq \Bigg| \mathbb{E}\Bigg( \sum\limits_{k=0}^{N-1}\rf{\Big(u_{k+1}Y_{k+1} - u_{k+1}^{\prime}Y_{k+1}^{'} + Z_{k+1}-Z_{k+1}^{'}\Big) I_{G_{-r}}\big(X_{k+1}^{'}\big)} \Bigg)\Bigg| \nonumber\\ & \;\;\;\;  + \Bigg| \mathbb{E}\Bigg(  \sum \limits_{k=0}^{N-1}\mathbb{E}\Big( u_{k+1}^{\prime}Y_{k+1}^{'}  - u_{k}Y_{k} + Z_{k+1}^{'} - Z_{k}\Big| X_{k}, Y_{k},Z_{k}\Big) \Bigg)\Bigg|. \label{ep4.16}
\end{align}
 Now we apply Lemma \ref{lemma2.2} and Lemma \ref{lemma2.3} to get
\begin{eqnarray}
 \big| \mathbb{E}\big(\varphi(X_{N})Y_{N} + Z_{N}\big)  - u(t_{0}, X_{0})Y_{0} + Z_{0}\big|    &\leq& Ch\Bigg| \mathbb{E}\Bigg( \sum\limits_{k=0}^{N-1}r_{k+1}Y_{k}I_{G_{-r}}\big( X_{k+1}^{'}\big) \Bigg)\Bigg| \nonumber \\ & +&  C h^{2}\Bigg|\mathbb{E} \sum\limits_{k=0}^{N-1}Y_{k}\Bigg|.\label{eq:30}
  \end{eqnarray}
From (\ref{mdy1}) and (\ref{eq:16}), we have
\begin{align*}
    Y_{k} &= Y_{k-1} + hc_{k-1}Y_{k-1} + 2r_{k}\gamma_{k}^{\pi}Y_{k-1}I_{G_{-r}}(X_{k}^{'}) + 2r_{k}^{2}\big(\gamma_{k}^{\pi}\big)^{2}Y_{k-1}I_{G_{-r}}(X_{k}^{'}),
\end{align*}
which, using the uniform boundedness  of $ \gamma(t,x)$ and $c(t,x)$,  $ (t,x) \in \bar{Q}$, and recalling $r_{k} = \mathcal{O}(h^{1/2})$, gives
\begin{align}
    Y_{k} & \leq Y_{k-1}\Big(1 + C_{1}h + C_{2}r_{k}I_{G_{-r}}\big( X_{k}^{'}\big)\Big) \leq Y_{k-1}\big(1 + C_{1}h\big)\big( 1 +  C_{2}r_{k}{I_{G_{-r}}(X_{k}^{'})}\big) \nonumber\\&
    \leq Y_{0}(1+C_{1}h)^{k}\prod_{i=1}^{k}\Big( 1 + C_{2}r_{i}I_{G_{-r}}(X_{i}^{'})\Big) \leq e^{C_{1}T}\prod_{i=1}^{k}\Big(1+C_{2}r_{i}I_{G_{-r}}(X_{i}^{'})\Big), \label{eq:31} 
\end{align}
where $C_{1}$ and $C_{2}$ are some positive constants. Then substituting (\ref{eq:31}) in (\ref{eq:30}), we obtain
\begin{align}
     \big| \mathbb{E}\big(\varphi(X_{N})Y_{N} + Z_{N}\big) - u(t_{0}, X_{0})\big| \nonumber & \leq  Ch\mathbb{E}\Bigg( \sum\limits_{k=1}^{N}r_{k}I_{G_{-r}}\big( X_{k}^{'}\big)\prod_{i=1}^{k-1}\Big(1+C_{2}r_{i}I_{G_{-r}}(X_{i}^{'})\Big)\Bigg) \nonumber \\ & \;\;\;\; +  C h^{2}\mathbb{E}\Bigg( \sum\limits_{k=1}^{N}\prod_{i=1}^{k-1}\big(1+C_{2}r_{i}I_{G_{-r}}(X_{i}^{'})\big)\Bigg),  \label{s2.33}
\end{align}
 which, together with Lemma \ref{bl},  (\ref{neweq3.29})  and the fact that $r_{N}=\mathcal{O}(h^{1/2})$,  implies the required result (\ref{eq:th4.4}).
\end{proof}

\begin{remark}\label{remark3.1}
We observe that Lemma~\ref{bl} 
does not require Assumption~\ref{as:1} and the proofs of Lemmas~\ref{lemma2.2}-\ref{lemma2.3} and Theorem~\ref{conp} rely only on sufficient smoothness of the solution $u(t,x)$ of (\ref{eq:2})-(\ref{eq:4}). This implies that Assumption~\ref{as:1} can be replaced by an appropriate hypoellipticity condition, but we do not pursue such a refinement of our results here.
We also note that Milstein's algorithm \cite{1} (see also \cite{49}) does intrinsically require the strong ellipticity (i.e.,  Assumption~\ref{as:1}) as its construction rests on changing local coordinates near the boundary.

\end{remark}

\begin{remark}
If we define the PDE (\ref{eq:2}) on $[0,T)\times \bar{G}$ (cf. (\ref{eq:47})-(\ref{eq:49}) and (\ref{eq:3.64})-(\ref{Nnonhb}) in the next section) then under Assumptions \ref{as:3}-\ref{as:2} and \ref{a01}-\ref{as:1} the solution $u(t,x)$ of (\ref{eq:2})-(\ref{eq:4}) belongs to $C^{2,4}([0,T)\times \bar{G})$ \cite{3,64} and $C(\bar{Q})$ \cite{lunbook}. Under these relaxed conditions, we can also prove the result of Theorem~\ref{conp} (cf. Lemmas~\ref{theorem3.9} and \ref{lemma3.12}).
\end{remark}

\section{Computing ergodic limits}\label{section6}
In Subsection~\ref{section2.3}, we set the scene required for introducing and analyzing time-averaging and ensemble-averaging estimators for computing ergodic limits. In Subsection~\ref{section5.1}, we introduce continuous time-averaging and ensemble-averaging estimators, and in Subsection~\ref{subsection4.3}, we present their numerical counterparts, which is computationally the main part of this section.  Subsections~\ref{section5.2} and \ref{nea} are devoted to error analysis of numerical time-averaging and ensemble-averaging estimators, respectively. 
Subsection~\ref{section5.5} addresses the question how close the stationary measure  $\mu$ of the RSDEs' (\ref{rsde}) solution $X(t)$ and \rs{a} stationary measure $\mu^{h}$ of the Markov chain $(X_k)_{k\geq 0}$ (constructed according to Algorithm~\ref{algorithm2.1}) are.

Here we consider computing ergodic limits using the simple-to-implement Algorithm~\ref{algorithm2.1}. At the same time, we note that for this purpose one can also use Milstein's algorithm \cite{1} (see also \cite[Chapter 6]{49}). As we mentioned in the introduction, weak first-order convergence of this algorithm at finite time was proved in \cite{1}. The theoretical results of this section on computing ergodic limits by Algorithm~\ref{algorithm2.1} can be transferred without any additional ideas required to the use of Milstein's algorithm for computing ergodic limits (see also Remark~\ref{remark3.1}).

 \subsection{Ergodic RSDEs and Poisson PDE} \label{section2.3}
This subsection is divided into four parts. As it was mentioned in the introduction, the main tool for obtaining continuous time-averaging estimators and analysis of errors of numerical time-averaging estimators is the Poisson equation with Neumann boundary condition. We discuss the existence and uniqueness of its solution in Subsection~\ref{subsection2.4.0}. We consider
ergodic limits with respect to the invariant density of the solution $X(t)$ of RSDEs in a bounded domain $G$ in Subsection~\ref{subsection2.4.1} and \rs{integrals} with respect to the normalised restriction of the invariant density of $X(t)$ on the boundary $\partial G$  in Subsection~\ref{subsection2.4.2}. We showcase our methodology by demonstrating how to sample from a given measure on $G$ or on $\partial G$ using Brownian dynamics (in other words, stochastic gradient system, which is also called Langevin equations in the fields of statistics and machine learning) with reflection on the boundary in Subsection~\ref{section2.4.3}.

\subsubsection{Poisson PDE with Neumann boundary condition}\label{subsection2.4.0}
Consider the Neumann problem for the Poisson equation
\begin{align}
    \mathscr{A}u(x) := \frac{1}{2}\sum\limits_{i,j=1}^{d}a^{ij}(x)\frac{\partial^{2}u(x) }{\partial x^{i}\partial x^{j}} + \sum\limits_{i=1}^{d}b^{i}(x)\frac{\partial u(x) }{\partial x^{i}} &= \phi_{1}(x),  & x \in G, \label{gpe1}\\
    (\nabla u(z) \cdot\nu(z)) &= \phi_{2}(z), & z\in \partial G. \label{gpe2}
\end{align}
We will need the following assumptions in addition to Assumption~\ref{as:3}.
\begin{assumption} \label{as:4}
The symmetric matrix $ a= \{a^{ij}\} $ in the operator $\mathscr{A} $
 is uniformly elliptic in $ \bar{G} $.
 \end{assumption}

\begin{assumption}\label{as:5}
  $a(x)$ and $b(x)$ are $ C^{2}(\bar{G})$ functions.
 \end{assumption}

\begin{assumption}
$\phi_1(x) \in C^2(\bar G)$ and $\phi_2(z) \in C^3(\partial G)$. \label{assu6.1}
\end{assumption}
 By $ \eta$ denote the co-normal vector on the boundary $\partial G$ whose direction cosines are
\begin{equation*}
    \cos(\eta(z), e^{i}) = \frac{1}{\alpha_{0}(z)}\sum\limits_{k=1}^{d}\frac{a^{ik}(z)}{2}\cos(\nu(z),e^{k}), \;\;\;\;\;\; z \in \partial G,
\end{equation*}
where $e^{i} $, $i =1, \dots, d$, represent the standard basis in the Cartesian notation, $\cos(\nu(z),e^{i})$ are the direction cosines of the inward normal $\nu(z)$ at $ z \in \partial G$, and $\alpha_{0}(z)$ is the normalization factor given by
\begin{equation*}
    \alpha_{0}(z) = \bigg(\sum\limits_{i=1}^{d}\bigg( \sum\limits_{k=1}^{d}\frac{a^{ik}(z)}{2}\cos(\nu(z),e^{k})\bigg)^{2}\bigg)^{1/2}, \;\;\;\;\;\; z \in \partial G.
\end{equation*}

Under Assumptions \ref{as:3}, \ref{as:4} and \ref{as:5}, there is a  unique solution $\rho(x)$ of the 
following \rf{Robin} problem \rf{for the stationary Fokker-Planck equation} \cite{50,45}:
\begin{equation} \label{sFP}
    \mathscr{A}^{*}\rho(x)  := \frac{1}{2}\sum\limits_{i,j=1}^{d}\frac{\partial^{2}}{\partial x^{i}\partial x^{j}}\big(a^{ij}(x)\rho(x)\big) - \sum\limits_{i=1}^{d} \frac{\partial }{\partial x^{i}}\big( b^{i}(x)\rho(x)\big) = 0, \;\;\;\;\; x \in G,
\end{equation}
with boundary condition
\begin{equation}\label{sFPbc}
    \alpha^{*}(z) (\nabla \rho(z)\cdot\eta^{*}(z))- \tilde{b}(z)\rho(z) = 0,\;\;\;\;\;\ z \in \partial G,
\end{equation}
where
$\eta^{*}(z),\;z\in \partial G$, is a unit vector whose direction cosines are
\begin{equation*}
    \cos(\eta^{*}(z), e^{i}) = \frac{2\alpha_{0}(z)}{\alpha^{*}(z)}\cos(\eta(z),e^{i}) -  \frac{\alpha(z)}{\alpha^{*}(z)}\cos(\nu(z),e^{i}).
\end{equation*}
\rf{Here} the normalization constant, $ \alpha^{*}(z)$, is such that $\sum\limits_{i=1}^{d}(\cos(\eta^{*}(z),e^{i}))^{2} = 1$, and
\begin{equation}
\alpha(z) = \alpha_{0}(z)\cos(\eta(z),\nu(z)) \rf{= \frac{1}{2}(\nu(z)\cdot a(z)\nu(z)), }\label{vrho}
\end{equation}
\rf{while} $\tilde{b}$ is \rf{given} in \cite[pp. 13-15]{50}.
\rf{Since the expression for $\tilde{b}$ is cumbersome  and not used in this paper, we do not provide it here}.
Note that $\rho(z)$, $z \in \partial G$, in (\ref{sFPbc}) is the trace of $\rho(x)$, $x \in \bar{G}$.
\rf{We note that (\ref{sFP})-(\ref{sFPbc}) is an adjoint homogeneous problem to  (\ref{gpe1})-(\ref{gpe2}).}

 As is known \cite{50,18}, solvability of the problem (\ref{gpe1})-(\ref{gpe2}) requires the compatibility (centering) condition
\begin{equation} \label{cc}
    \int_{G}\phi_{1}(x)\rho(x) dx +  \int_{\partial G} \phi_{2}(z) \alpha(z) \rho(z) dz = 0.
\end{equation}
If Assumptions \ref{as:3}, \ref{as:4}, \ref{as:5} and \ref{assu6.1} hold along with the centering condition (\ref{cc}) then the problem (\ref{gpe1})-(\ref{gpe2}) has a unique solution (up to an additive constant) $u(x) \in C^{4}(\bar{G})$ (see \cite{45,50}, \cite[Chapter 6]{104}, \cite[Theorem 3]{107}).

 \subsubsection{Ergodic limits in $G$\label{subsection2.4.1}}
 Consider the RSDEs
 \begin{align}
dX(s) = b(X(s))ds + \sigma (X(s))dW(s) + \nu (X(s))I_{\partial G}(X(s)) dL(s), \,\, X(0) = x. \label{rsde}
 \end{align}
 If Assumptions \ref{as:3}, \ref{as:4} and \ref{as:5} hold then there exists a unique invariant probability measure of the process $X(t)$ governed by the RSDEs (\ref{rsde}). Moreover, this measure is absolutely continuous with respect to  Lebesgue measure. We denote this invariant measure by $\mu(x)$ and its density by $\rho(x)$, $ x\in \bar{G}$, and we state that $ \rho(x) > 0$ and $\rho(x) \in C^{2}(\bar{G}) $ (see \cite{ 45} for more details). Indeed, $\rho(x)$ is the solution of the stationary Fokker-Planck equation (\ref{sFP}), and the restriction $\rho(z),\; z \in \partial G$, of the invariant density $\rho(x)$, satisfies the boundary condition (\ref{sFPbc}). We are interested in computing ergodic limits inside the domain, i.e., in calculating the following integral for some function $\varphi(x) \in C^{2}(\bar{G})$:
 \begin{equation}
  \bar{\varphi} = \int_{G} \varphi(x)\rho(x)dx. \label{q2.24}
\end{equation}
We consider the corresponding approximations in Subsection~\ref{subsection4.3}.

\subsubsection{Ergodic limits on $\partial G$}\label{subsection2.4.2}
There are two aspects \rs{related to} computing ergodic limits on the boundary $\partial G$ using RSDEs as explained below.
\begin{enumerate}[label= (\roman*)]
\item Consider the (restricted) \rs{function} $\rho(z)$ defined on $\partial G$. We denote $\int_{\partial G}\rho(z)dz$ as $\kappa$ and define
 \begin{equation}\label{eb2.25}
 \rho'(z) = \rho(z)/\kappa.
 \end{equation}
 One can notice that $\rho'(z)$ is a probability density on the boundary $\partial G$. We can calculate, for some function $\psi(z) \in C^{3}(\partial G)$,
 \begin{equation}
 \bar{\psi}^{'} = \int_{\partial G} \psi(z) \rho'(z)dz\rs{.} \label{q2.26}
 \end{equation}
 To find $\bar{\psi}^{'}$, we first need to calculate
 \begin{equation}\label{psibar}
 \bar{\psi} = \int_{\partial G} \psi(z) \rho(z) dz.
 \end{equation}

\item   Using $L(t)$ as the random time change process allows us to build a Markov process $\tilde{X}(t)=\tilde{X}_x(t)$ from the solution $X(t)=X_x(t)$ of RSDEs (\ref{rsde}).  We obtain this process $\tilde{X}(t)$ on the boundary $\partial G$ by putting $\tilde{X}(t) = X(L^{-1}(t))$, where $ L^{-1}(t)$ is the right continuous inverse of $L(t)$. \rf{Note that $\lim_{t \rightarrow \infty}L(t) = \infty$ (cf. (\ref{3.17})).} Denote by $\tilde{\mathscr{F}}$ and $\tilde{\mathscr{F}}_{t}$ the smallest sigma algebras which make $\{ \tilde{X}(t),\; 0 \leq t < \infty \}$ and $\{ \tilde{X}(s),\; 0 \leq s\leq t \} $ measurable, respectively. Further, let $\tilde{\mathbb{P}}$ be the probability measure defined by $\tilde{\mathbb{P}}(B) = \mathbb{P}\big(X\big(L^{-1}(\rf{\cdot})\big) \in B  \big)$, where $B \in \tilde{\mathscr{F}}$ (see \cite{30}). The process $\tilde{X}(t)$ is a strong Markov \cite[Theorems 1 and  2]{36} jump  \cite{53} process  on $\partial G$. It induces the following semigroup $\tilde{T}_{t}$ on $C(\partial G)$ \cite[ Theorem 9.1]{30}:
\begin{align}\label{semigrp}
    \tilde{T}_{t}f(x) &= \tilde{\mathbb{E}}f\big(\tilde{X}_{x}(t)\big) = \mathbb{E}f\big(X_{x}\big(L^{-1}(t)\big)\big),
\end{align}
where $f$ is a bounded measurable function on $\partial G$, $\rf{x \in \partial G}$,  $\tilde{\mathbb{E}}(\cdot)$ is expectation with respect to the probability measure $\tilde{\mathbb{P}}$. Under Assumptions \ref{as:3} and \ref{as:4}-\ref{as:5}, $\tilde{X}(t)$ has a unique stationary distribution on $\partial G$ \cite[p. 174]{18}, which we denote as $\tilde{\mu} $ and its density we denote as $\tilde{\rho}$. Hence, computing ergodic limits on the boundary $\partial G$ also means evaluating the integral $\int_{\partial G}\psi(z) \tilde{\rho}(z)dz$, which is  expectation with respect to the invariant law of $\tilde{X}(t)$:
\begin{equation}
       \tilde{\psi}  = \int_{\partial G}\psi(z)\tilde{\rho}(z)dz.
       \label{tildepsi}
\end{equation}
We will discuss the relationship between the densities $\tilde{\rho}(z)$ and $\rho(z)$ in Subsection \ref{subsection4.2.1}.

\end{enumerate}
We consider  approximations of $\bar{\psi}^{'}$ and $\tilde{\psi}$ in Subsection~\ref{subsection4.3}.

\subsubsection{Gradient system for sampling from a given measure with compact support} \label{section2.4.3}
This subsection highlights one of the major applications of this paper. Consider the simplified RSDEs
\begin{align}
 dX(s) &= b(X(s))ds + \sigma dW(s) + \nu (X(s))I_{\partial G}(X(s)) dL(s), \label{rgsde}
 \end{align}
 where $\sigma>0$ is a constant.

Under Assumptions \ref{as:3} and \ref{as:5}, the above RSDEs  has a stationary distribution whose density, $\rho(x)$, solves the following stationary Fokker-Planck equation
\begin{equation} \label{sgFP}
    \frac{\sigma^2}{2}\sum\limits_{i=1}^{d}\frac{\partial^{2}}{(\partial x^{i})^2}\rho(x) - \sum\limits_{i=1}^{d} \frac{\partial }{\partial x^{i}}\big( b^{i}(x)\rho(x)\big) = 0, \;\;\;\;\; x \in G,
\end{equation}
with boundary condition (note that $\tilde{b}(z)$ reduces to $b(z)$ in (\ref{sFPbc}), cf. \cite[pp. 13-15]{50}):
\begin{equation}\label{sgFPbc}
   \frac{\sigma^2}{2}(\nabla \rho(z)\cdot\nu(z)) - (b(z)\cdot\nu(z))\rho(z) = 0,\;\;\;\;\;\ z \in \partial G.
\end{equation}

Suppose we are given a probability density $\rho(x) \in C^{3}(\bar{G})$, $\rho(x)>0$ for $x \in \bar{G}$  from which we want to sample/with respect to which we would like to compute some integrals. As we mentioned in the introduction, such problems often occur in statistics \cite{101,102,103} and in molecular dynamics \cite{100}. To solve this problem, we can use the RSDEs (\ref{rgsde}) analogously to how SDEs (Brownian dynamics and Langevin equations) are used for this task in $\mathbb{R}^d$ (see e.g. \cite{RT96,40,100} and references therein). Indeed, if we 
take in (\ref{rgsde})
\begin{equation}\label{eb2.32}
b(x) = \frac{\sigma^2}{2}\nabla\log{\rho(x)},
\end{equation}
then the equations (\ref{sgFP})-(\ref{sgFPbc}) are trivially satisfied. The RSDEs (\ref{rgsde}), (\ref{eb2.32}) \rf{are} a stochastic gradient system with reflection, or, in other words, Brownian dynamics with reflection. Long time simulation of the gradient system (\ref{rgsde}), (\ref{eb2.32}) can be used for sampling from a given distribution having density $\rho(x)$  with the compact support $\bar{G}$.

As we have emphasized in the introduction, using the methodology developed in this paper we can also sample from the distribution having density $\rho'(z)$ with support on the boundary $\partial G$.
To illustrate the normalised restricted density $\rho'(z)$, let us look at a simple example. Consider $G = \{ x_{1}^{2} + x_{2}^{2} < 1\}$, $\partial G = \{ z_{1}^{2} + z_{2}^{2} = 1\}$, and the density function
\begin{equation*}
    \rho(x) = \frac{1}{\mathbb{Z}}(x_{1}^{2} + x_{2}^{2})e^{\beta x_{1}}, \, x \in \bar G,
\end{equation*}
where $\beta \geq 0$ is a constant and $\mathbb{Z}$ is the normalisation constant. Further, $\kappa = \int_{\partial G}\rho(z)dz = \frac{2\pi I_{0}(\beta)}{\mathbb{Z}}$, where $I_{0}(\beta)$ is the modified Bessel function of order $0$.
Then the normalised restriction $\rho'(z)$ is given by
   $ \rho'(z) = \frac{e^{\beta z_{1}}}{2\pi I_{0}(\beta)}, \, z \in \partial G.$
  We write $\rho'(z)$ in the polar coordinates to get
\begin{equation*}
    \rho'(\theta) = \frac{e^{\beta \cos(\theta)}}{2\pi I_{0}(\beta)}, 
\end{equation*}
which is the probability density function of von Mises' distribution. The motivation for considering this particular example lies in the fact that this distribution has a number of applications in directional statistics (see e.g. \cite{76}). Algorithm~\ref{algorithm2.1} described in Section~\ref{section2} as well as the estimators of Subsection~\ref{subsection4.3} can be used to sample from von Mises' distribution. The same approach is applicable to von Mises-Fisher, Bingham and Kent distributions or any other distribution on $d-1$ dimensional hyper-surface $\partial G$ (see a related numerical experiment, Experiment~\ref{experiment8.3}, in Subsection~\ref{section8.2}), which are widely used in bioinformatics, computer vision, geology and astrophysics (see, e.g. \cite{77}).

To conclude, the application of Algorithm~\ref{algorithm2.1}  to the gradient system with reflection (\ref{rgsde}), (\ref{eb2.32}) allows us to efficiently sample from any given distribution with the density $\rho(x)$ defined in $G$ and from its normalised restriction $\rho'(z)$ on $\partial G$.

 \subsection{Time-averaging and ensemble-averaging estimators} \label{section5.1}

In Subsection \ref{subsection4.2.1}, we obtain continuous time-averaging estimators for $\bar{\varphi}$ from (\ref{q2.24}) and for $\bar{\psi}^{'}$ from (\ref{q2.26}). In Subsection~\ref{subsection4.2.2}, we introduce continuous ensemble-averaging estimators for the same.  

\subsubsection{Time-averaging estimators}\label{subsection4.2.1}
{\bf Time-averaging estimator for $\bar{\varphi}$ from (\ref{q2.24}).} Consider the Neumann problem
\begin{align}
     &\mathscr{A}u(x) = \varphi(x) -\bar{\varphi},\;\;\;\;  x \in G, \label{n311}\\
     &(\nabla u(z)\cdot\nu(z))  = 0,\;\;\;\; z \in \partial G. \label{n312}
\end{align}
It is not difficult to verify that the compatibility condition (\ref{cc}) is satisfied for the problem (\ref{n311})-(\ref{n312}). Then, under Assumptions~\ref{as:3}, \ref{as:4} and \ref{as:5} for any $\varphi(x) \in C^{2}(\bar{G})$, there exists a solution $u(x) \in C^{4}(\bar{G})$ (see Subsection~\ref{subsection2.4.0}). Apply Ito's  formula to the function $u(X(t))$ with $X(t)$ from (\ref{rsde}):
\begin{eqnarray}
    u(X(t)) - u(x) &=& \int_{0}^{t}\mathscr{A}u(X(s))ds + \int_{0}^{t} \nabla u(X(s))\cdot\sigma(X(s))dW(s)\nonumber \\ & +& \int_{0}^{t} \big(\nabla u(X(s))\cdot\nu(X(s))\big)dL(s). \label{itof}
\end{eqnarray}
Rearranging the terms in (\ref{itof}), dividing by $t$, and using equation (\ref{n312}), we get
\begin{eqnarray}
    \frac{1}{t}\int_{0}^{t}\mathscr{A}u(X(s))ds  = - \frac{M(t)}{t} + \frac{u(X(t))-u(x)}{t}, \label{neq:37}
\end{eqnarray}
where $M(t) = \int_{0}^{t}\nabla u(X(s))\cdot\sigma(X(s))dW(s)$ is a  martingale with respect to $(\mathscr{F}_{t})_{t\geq 0}$. Then (\ref{n311}) implies
\begin{equation}
    \frac{1}{t}\int_{0}^{t}\varphi(X(s))ds - \bar{\varphi} = - \frac{M(t)}{t} + \frac{u(X(t))-u(x)}{t}. \label{n6.5}
\end{equation}
We know from Ito's isometry that
$    \mathbb{E}\big(M(t)^{2}\big)  \leq Ct,
$
where $C$ is some positive constant independent of $t$.  As a consequence of uniform boundedness of $u(x) $ in $\bar{G}$, we obtain
\begin{equation}
    \mathbb{E}\bigg(\frac{1}{t}\int_{0}^{t}\varphi(X(s))ds - \bar{\varphi}\bigg)^{2} \leq \frac{C}{t}.  \label{n6.6}
\end{equation}

 Kronecker's lemma \cite[Theorem 3.3]{22} implies $\frac{M(t)}{t}\rightarrow 0 $ a.s. as $t\rightarrow \infty$. Again uniform boundedness of $u(x)$,  $x \in \bar{G}$, yields that the last term on the right hand side of (\ref{n6.5}) goes to $0$ as $ t \rightarrow \infty$. Therefore, we have
\begin{align}
    \lim\limits_{t \rightarrow \infty}\frac{1}{t}\int_{0}^{t} \varphi(X(s)) ds = \bar{\varphi}\;\;\;\; a.s. \label{varphibar}
\end{align}
If we take expectation on both sides of (\ref{n6.5}), we get
\begin{equation}
    \bigg| \frac{1}{t}\mathbb{E}\bigg(\int_{0}^{t}\varphi(X(s))ds\bigg) - \bar{\varphi}\bigg| \leq \frac{C}{t}.  \label{bias}
\end{equation}
Consequently, it is natural to take $ \frac{1}{T}\int_{0}^{T}\varphi(X(s))ds  $ as a time-averaging estimator of $\bar{\varphi}$. We can also say based on (\ref{bias}) that
\begin{equation*}
    \bias\bigg(\frac{1}{T}\int_{0}^{T}\varphi(X(s))ds\bigg) = \mathcal{O}\Big(\frac{1}{T}\Big).
\end{equation*}
One may notice that by combining (\ref{n6.6}) and the above bound on the bias, we get
\begin{equation*}
    \var\bigg(\frac{1}{T}\int_{0}^{T}\varphi(X(s))ds\bigg) = \mathcal{O}\Big(\frac{1}{T}\Big).
\end{equation*}

{\bf Time-averaging estimator for $\bar{\psi}^{'}$ from (\ref{q2.26}).}
\rf{To obtain results \rs{related to} ergodic limits on the boundary, we will consider the PDE problem (\ref{gpe1})-(\ref{gpe2}) with $\phi_{2}(z) = \frac{1}{\alpha(z)}$ or $ \frac{\psi(z)-\bar{\psi}^{'}}{\alpha(z)}$, where $\alpha(z)$ depends on $a(z)$ ($a(z)$ is the restriction of $a(x)$ on $\partial G$), and hence we need a new assumption on $a(x)$ to ensure that $\phi_{2}(z) \in C^3(\partial G)$ and therefore to guarantee that the solution $u(x) \in C^{4}(\bar{G})$ (cf. Assumption \ref{assu6.1}).
\begin{assumption1}{4.2${}^{'}$}
 $a(x) \in  C^{3}(\bar{G})$  and $b(x) \in C^{2}(\bar{G})$ 
 in (\ref{gpe1}).\label{as4.2'}
\end{assumption1}}
Before we move ahead to discuss estimators for $\bar{\psi}^{'}$, we need to look at the asymptotic behavior of the integral $\int_{0}^{t}\frac{1}{\alpha(X(s))}dL(s)$. Consider the following Poisson equation with  Neumann boundary condition:
\begin{align}
    \mathscr{A}u(x) &= -\kappa, & x \in G, \label{eq:5.34}\\
    (\nu(z)\cdot\nabla u(z)) &= 1/\alpha(z), & z \in \partial G, \label{eq:5.35}
\end{align}
where $\kappa$  and $\alpha(z)$ are defined in Subsection~\ref{subsection2.4.2} (see (\ref{eb2.25})) and Subsection~\ref{subsection2.4.0} (see (\ref{vrho})), respectively. Note that under Assumption~\ref{as:4}, $\alpha(z) > 0$ for all $z \in \partial G$.
If we take $\phi_{1}(x) = -\kappa$ and $ \phi_{2}(z) = \frac{1}{\alpha(z)}$ in (\ref{gpe1})-(\ref{gpe2}) then it is not difficult to notice that compatibility condition (\ref{cc}) is satisfied.
This implies that under Assumptions~\ref{as:3} and \ref{as:4}, \rf{\ref{as4.2'}} the solution $u(x)$  of (\ref{eq:5.34})-(\ref{eq:5.35}) belongs to $C^{\rf{4}}(\bar{G})$.

Introduce the notation
\begin{align*}
  \rf{Z_{1}(t) = \int_{0}^{t}\frac{\psi(X(s))}{\alpha(X(s))}dL(s), \;\;\;\;\;\; Z_{2}(t) = \int_{0}^{t}\frac{1}{\alpha(X(s))}dL(s)}.
  \end{align*}

Applying Ito's  formula to  $u(X(t))$ with $u(x)$ being the solution of (\ref{eq:5.34})-(\ref{eq:5.35}) and acting analogously to how we obtained (\ref{n6.6}), (\ref{varphibar}) and (\ref{bias}) above, we get
\begin{align}
   &\mathbb{E}\bigg(\rf{\frac{1}{t}}Z_{2}(t) - \kappa \bigg)^{2} \leq \frac{C}{t}, \label{vakp} \\
    & \lim\limits_{t \rightarrow \infty}\frac{1}{t}Z_{2}(t) = \kappa \;\;\;\;\; a.s., \label{3.17} \\
      &   \bigg| \frac{1}{t} \mathbb{E}\big(Z_{2}(t)\big) - \kappa\bigg| \leq \frac{C}{t},  \label{318} \\
        &  \lim\limits_{t\rightarrow \infty}\frac{1}{t}\mathbb{E}\big(Z_{2}(t)\big) = \kappa. \label{eq:5.18}
\end{align}

  We use a similar reasoning to find an estimator for $\bar{\psi}$. For that purpose, first consider the problem
\begin{align}
\mathscr{A}u(x) &= 0,\;\;\; \;\;\;\;\;\;\;\;\;\;\;\;\;\;  &x \in G,  \label{p1} \\
(\nabla u(z)\cdot\nu(z)) &= \big(\psi(z) - \bar{\psi}^{'}\big)/\alpha (z),\;\;\;\;\;  &z \in \partial G. \label{p2}
\end{align}
As one can check, the compatibility condition (\ref{cc}) is verified. Then under Assumptions \ref{as:3} and \ref{as:4}, \rf{\ref{as4.2'}}, for any $\psi \in C^{3}(\partial G)$, there exists a solution of (\ref{p1})-(\ref{p2}) which is unique up to an additive constant (see Subsection~\ref{subsection2.4.0}). Therefore, Ito's formula again gives us (\ref{itof}), however in this case $\mathscr{A}\rf{u}(X(s))= 0$ and $(\nabla u(X(s))\cdot\nu(X(s))) = (\psi(X(s))-\bar{\psi}^{'})/\alpha(X(s))$. Consequently, after rearrangement and dividing by $t$, we have
\begin{align}
    \frac{1}{t}\int_{0}^{t}\frac{\psi(X(s))}{\alpha(X(s))}dL(s) - \bar{\psi}^{'}\frac{1}{t}\int_{0}^{t}\frac{1}{\alpha(X(s))}dL(s) = -\frac{M(t)}{t} + \frac{u(X(t)) -u(x)}{t}. \label{en}
\end{align}
Using (\ref{3.17}) and applying the same arguments as used above, we obtain
\begin{align}
   \lim\limits_{t\rightarrow \infty}\frac{1}{t}\rf{Z_{1}(t)} = \bar{\psi}^{'}\kappa = \bar{\psi} \label{psiestas} \;\;\;\;\;\;\; a.s.
\end{align}

If we take expectation on both sides of (\ref{en}) and use (\ref{318}), we get
\begin{equation*}
    \Big| \frac{1}{t}\mathbb{E}\rf{(Z_{1}(t))}- \bar{\psi} \Big| \leq  \frac{C}{t},
\end{equation*}
where $C$ is some positive constant independent of time $t$, and
\begin{equation}
    \lim\limits_{t\rightarrow \infty} \frac{1}{t}\mathbb{E}\rf{\big(Z_{1}(t)\big)} = \bar{\psi}. \label{eq:3.25}
\end{equation}
Moreover, one can also show
\begin{equation}
    \mathbb{E}\Big(\rf{\frac{Z_{1}(t)}{t}} - \bar{\psi}\Big)^{2} \leq \frac{C}{t}.
\end{equation}

Note that using (\ref{vakp}) and Chebyshev's inequality, we get
\begin{equation*}
     \mathbb{P}(Z_{2}(t) = 0) \leq \frac{C}{t}, \label{eq:prob0}
 \end{equation*}
 where $C >0$ is independent of $t$.

 \rf{Furthermore, from almost sure convergence in (\ref{3.17}), we have a set $\Omega_{0}\subset \Omega$ with $\mathbb{P}(\Omega_{0}) = 1$, and for every $\omega \in \Omega_{0}$ and for every $\epsilon > 0$, there  exists a $t_{0}(\omega, \epsilon)$ such that for all $t > t_{0}(\omega, \epsilon)$
  $ | Z_{2}(t)/t - \kappa | < \epsilon$,
 whence choosing  $\epsilon = \kappa/2$:
\begin{equation*}
    \frac{\kappa t}{2} < Z_{2}(t) < \frac{3 \kappa t}{2},
\end{equation*}
consequently, for sufficiently large $t$ we have  $ Z_{2}(t)>0$ a.s.}
Then, combining (\ref{3.17}) with (\ref{psiestas}), we get
\begin{align}
   \lim\limits_{t\rightarrow \infty}\rf{\big(\left. Z_{1}(t)\middle/Z_{2}(t)\right. \big)}&= \bar{\psi}^{'}\;\;\;\;\;\; a.s. \label{sidash1}
 \end{align}
\rf{Note that the limit (\ref{eq:5.18}) guarantees that for a sufficiently large $t$ we have $\mathbb{E}(Z_2(t))>\kappa t/2$. Therefore, by taking expectation of (\ref{en}) and using (\ref{318}), we get the following for sufficiently large $t$}
 \begin{align}
     \big| \mathbb{E}\rf{\big(Z_{1}(t)\big) / \mathbb{E}\big(Z_{2}(t)\big) }-\bar{\psi}^{'} \big| & \leq \frac{C}{t}.\label{sidash2}
 \end{align}
 Furthermore, \rf{we have 
 \begin{align*}
     (Z_{1}(t))^{2} = \bigg(\int_{0}^{t}\frac{\psi(X(s))}{\alpha(X(s))}dL(s)\bigg)^{2} \leq C\bigg(\int_{0}^{t}\frac{1}{\alpha(X(s))}dL(s)\bigg)^{2} \leq C(Z_{2}(t))^{2}\;\; a.s.
 \end{align*}
 which implies
 \begin{align*}
     \bigg(\frac{Z_{1}(t)}{Z_{2}(t)}\bigg)^{2}I_{(0,\infty)}\big(Z_{2}(t)\big) \leq C\;\; a.s.,
 \end{align*}
 where $C>0$ is independent of $t$. Therefore,
 \begin{align} \label{6.27nya}
     &\mathbb{E}\bigg(\frac{Z_{1}(t)}{Z_{2}(t)}I_{(0,\infty)}\big(Z_{2}(t)\big) - \bar{\psi}^{'}\bigg)^{2} 
     =\mathbb{E}\bigg(\frac{Z_{1}(t)/t}{Z_{2}(t)/t}I_{(0,\infty)}\big(Z_{2}(t)\big)
     -\frac{\bar{\psi}}{\kappa}I_{(0,\infty)}(Z_{2}(t)) \\ & \;\;\;\; -\frac{\bar{\psi}}{\kappa}I_{\{Z_2(t)=0\}}
     -\frac{Z_1(t)/t}{\kappa}I_{(0,\infty)}(Z_{2}(t)) + \frac{Z_1(t)/t}{\kappa}I_{(0,\infty)}(Z_{2}(t))
     \bigg)^{2}\nonumber \\
     & \leq C\Bigg(\mathbb{E}\bigg(\frac{Z_{1}(t)}{\kappa Z_{2}(t)}I_{(0,\infty)}(Z_{2}(t))
     \Big(\frac{Z_{2}(t)}{t} - \kappa\Big)\bigg)^{2} \nonumber \\ &\;\;\;\;
     +\mathbb{E}\Big(\frac{Z_{1}(t)}{t} - \bar{\psi}\Big)^{2}
     +\mathbb{P}(Z_{2}(t) = 0)\Bigg)
         \leq \frac{C}{t},\nonumber
  \end{align}
  where $C > 0$ is independent of $t$.}  \rf{In the similar manner, using Cauchy$-$Bunyakovsky$-$Schwarz inequality, we get
  \begin{align} \label{6.26nya}
      &\bigg|\mathbb{E}\bigg(\frac{Z_{1}(t)}{Z_{2}(t)}I_{(0,\infty)}(Z_{2}(t)) - \bar{\psi}^{'}\bigg)\bigg| 
      \leq \frac{1}{\kappa} \bigg|\mathbb{E}\bigg(\frac{Z_{1}(t)}{t} - \bar{\psi}\bigg) \bigg|
      + \frac{1}{\kappa} \bigg| \mathbb{E}\bigg( \Big(\frac{Z_{2}(t)}{t} -  \kappa\Big)\\ & \;\;\; \times \Big(\frac{Z_{1}(t)}{Z_{2}(t)}I_{(0,\infty)}(Z_{2}(t))  -\bar{\psi}^{'}\Big)\bigg)\bigg| + \frac{\bar{\psi}^{'}}{\kappa}\bigg|\mathbb{E}\bigg(\frac{Z_{2}(t)}{t} - \kappa\bigg)\bigg| +
      \bar{\psi}^{'}\mathbb{P}(Z_{2}(t) = 0)\nonumber\\
     &\leq \frac{1}{\kappa}\bigg|\mathbb{E}\bigg(\frac{Z_{1}(t)}{t} - \bar{\psi}\bigg)\bigg|  +  \frac{1}{\kappa}\Bigg(\mathbb{E}\bigg(\frac{Z_{1}(t)}{Z_{2}(t)}I_{(0,\infty)}(Z_{2}(t)) - \bar{\psi}^{'}\bigg)^{2}\mathbb{E}\bigg(\frac{Z_{2}(t)}{t} - \kappa\bigg)^{2}\Bigg)^{1/2}\nonumber \\ & \;\;\;\;
     +\frac{\bar{\psi}^{'}}{\kappa}\bigg|\mathbb{E}\bigg(\frac{Z_{2}(t)}{t} - \kappa\bigg)\bigg|
     + \bar{\psi}^{'}\mathbb{P}(Z_{2}(t) = 0) \leq \frac{C}{t} . \nonumber
  \end{align}
 }

For a fixed $T$, we can view
 \begin{align}
  \rf{\big(Z_{1}(T)/Z_{2}(T)\big)I_{(0,\infty)}(Z_{2}(T))}\label{n6.19}
 \end{align}
 as a time-averaging estimator for $\bar{\psi}^{'}$. It follows from the above analysis that the bias and variance of the estimator (\ref{n6.19}) is $\mathcal{O}\big(\frac{1}{T}\big)$.

{\bf Time-averaging estimator for $\tilde{\psi}$ from (\ref{tildepsi}).}
In Subsection~\ref{subsection2.4.2}, we discussed the invariant density $\tilde{\rho}(z)$ of $\tilde{X}(t)$, a jump process evolving on $\partial G$.
 From \cite[Lemma 2.1]{72} and (\ref{semigrp}), we have
\begin{align}\label{new1}
    \mathbb{E}\bigg( \frac{\rf{I_{(0,\infty)}(L(t))}}{L(t)}&\int_{0}^{t} \psi(X(s))dL(s) - \tilde{\psi}\bigg)^{2} \\ & =  \mathbb{E}\bigg(\frac{\rf{I_{(0,\infty)}(L(t))}}{L(t)}\int_{0}^{L(t)}\psi(X(L^{-1}(s)))ds - \tilde{\psi}\bigg)^{2} \nonumber \\ & =  \tilde{\mathbb{E}}\bigg(\frac{\rf{I_{(0,\infty)}(L(t))}}{L(t)}\int_{0}^{L(t)}\psi(\tilde{X}(s))ds - \tilde{\psi}\bigg)^{2}. \nonumber
\end{align}
\rf{Recalling that $\tilde X(s)$ is an ergodic process, $\lim_{t \rightarrow \infty}L(t) = \infty$ (cf. (\ref{3.17})) and also (\ref{tildepsi}), it is not difficult to see that the right-hand side of (\ref{new1}) converges to $0$ as $t \rightarrow \infty$. }
Introduce
\begin{align}
\tilde{\kappa} = \int_{\partial G} \alpha(z) \rho(z) dz\;\;\text{and}\;\; \tilde{\psi}_{0}  = \int_{\partial G}\psi(z)\alpha(z)\rho(z)dz, \label{tildekappa}
\end{align}
where $\alpha(z)$ is from (\ref{vrho}). It follows from (\ref{psiestas}) that
    \begin{equation}\label{new3}
      \tilde{\psi}_{0}   =  \lim\limits_{t\rightarrow \infty}\frac{1}{t}\int_{0}^{t}\psi(X(s))dL(s)\;\;\;\; a.s.  \;\;\text{and}\;\;\;\;\;
      \tilde{\kappa}  = \lim\limits_{t\rightarrow \infty}\frac{L(t)}{t}\;\;\;\; a.s.
  \end{equation}
 We can also obtain \rf{the following by the same procedure used to ascertain (\ref{6.27nya}):}
\begin{align}
\label{p2.27}
 \mathbb{E}\Bigg(\frac{\frac{1}{t}\int_{0}^{t} \psi(X(s))dL(s)}{\frac{L(t)}{t}}\rf{I_{(0,\infty)}(L(t))} - \frac{\tilde{\psi_0}}{\tilde \kappa}\Bigg)^{2}  \leq \frac{C}{t},
\end{align}
where $C$ is a positive constant independent of $t$.
Furthermore,
   \begin{equation}\label{new4}
     \rs{ \frac{\tilde{\psi}_{0}}{\tilde{\kappa}}} =  \lim\limits_{t\rightarrow \infty}\frac{1}{L(t)}\int_{0}^{t}\psi(X(s))dL(s)\;\;\;\; a.s.
  \end{equation}
 Comparing (\ref{new1}) and (\ref{p2.27}), we get
\begin{equation*}
  \tilde{\psi}= \tilde{\psi_0}/\tilde \kappa,
\end{equation*}
and by (\ref{tildepsi}) and (\ref{tildekappa}), we arrive at the relationships between the densities $\tilde{\rho}(z)$, $\rho(z)$ and \rf{$\rho'(z)$ (see (\ref{eb2.25}))}:
 \begin{equation}
 \tilde{\rho}(z) = \alpha(z)\rho(z)/\tilde{\kappa} \rf{=\alpha(z)\rho'(z)\kappa/{\tilde \kappa}}
 \end{equation}
 and
 \begin{equation}
\rho'(z) = \tilde{\kappa}\tilde{\rho}(z)/(\kappa\alpha(z)).
\end{equation}
  Based on (\ref{new3}) and (\ref{new4}),  we can take
  $$
     \frac{1}{t}\int_{0}^{t}\psi(X(s))dL(s),\;\;\;\;
      \frac{L(t)}{t}, \;\;\;\;
      \frac{\rf{I_{(0,\infty)}(L(t))}}{L(t)}\int_{0}^{t}\psi(X(s))dL(s),
 $$
      as the continuous time-averaging estimators for  $\tilde{\psi}_{0} $, $\tilde{\kappa} $, and $\tilde{\psi}$, respectively.

 \rf{
\begin{remark}
If we replace $\nu(X(s))$ by $a(X(s))\nu(X(s))/2$ (\rs{one should} note that $a(z) \nu(z)/|a(z) \nu(z)|$ represents the conormal direction at $\partial G$)
in RSDEs (\ref{rsde}) then the compatibility (centering) condition to be satisfied for the existence of the solution of corresponding Poisson PDE (\ref{gpe1}) with boundary condition $(a(z)\nu(z)\cdot \nabla u(z))/2 = \phi_{2}(z)$ is (see \cite{45}):
\begin{equation*}
 \int_{G}\phi_{1}(x)\rho(x)dx + \int_{\partial G}\phi_{2}(z)\rho(z)dz  = 0.
\end{equation*}
In this case, one can show that $\tilde{\kappa} = \kappa = \rs{\int_{\partial G}\rho(z)dz}$ and $\tilde{\psi}_{0} = \bar{\psi} = \rs{\int_{\partial G}\psi(z)\rho(z)dz}$ which implies $\rho'(z) = \tilde{\rho}(z)$.
In particular, this equality also holds for the RSDEs (\ref{rgsde}) when $\nu(X(s))$ is replaced by $\sigma^2\nu(X(s))/2$. See also Section~\ref{section7.2}.
\end{remark}
}

 \subsubsection{Ensemble-averaging estimators}\label{subsection4.2.2}

 It is known \cite{18} that reflected diffusion governed by (\ref{rsde}) satisfies Doeblin's condition \rf{under Assumptions~\ref{as:3}, \ref{as:4} and \ref{as:5}}, and therefore for any bounded and measurable function $\varphi$,
  \begin{align}
       \big|\mathbb{E}\varphi\big(X(t)\big) - \bar{\varphi}\big| \leq Ce^{-\lambda t} \label{eq:34},
\end{align}
where $ \bar{\varphi} =  \int_{G}\varphi(x)\rho(x)dx$, $\rho(x)$ is invariant density of $X(t)$ (see Subsection~\ref{section2.3}), and $C$ and $\lambda$ are positive constants independent of $t$. This implies that we can compute $\bar{\varphi}$ approximately by evaluating $\mathbb{E}\varphi(X(t))$ for sufficiently large $t $ using the Monte Carlo technique. 

Further, from (\ref{sidash2})-(\ref{6.26nya}), we infer that we can calculate $\bar{\psi}^{'}$ defined in (\ref{q2.26}) by evaluating
\begin{equation*}
     \left.\mathbb{E}\big(\rf{Z_1(t)}\big) \middle/\mathbb{E}\big(\rf{Z_2(t)}\big)\right. \;\; \text{or}\;\;  
     \left. \mathbb{E}\Big(\big( \rf{Z_1(t)}\middle/\rf{Z_2(t)} \big)\rf{I_{(0,\infty)}\big(Z_2(t)\big)}\Big) \right. 
\end{equation*}
for sufficiently large $t$.
\begin{remark}
At the end of Subsection \ref{subsection4.2.1} we introduced the continuous time-averaging estimator to find $\tilde{\psi}$. In the same manner the ensemble-averaging estimators for calculating the expectation with respect to the invariant law of $\tilde{X}(t)$ are
\begin{equation*}
    \mathbb{E}\bigg(\int_{0}^{t}\psi(X(s))dL(s)\bigg)\bigg/\mathbb{E}L(t)\;\;\;\;\;\;\text{and}\;\;\;\;\;\;\; \mathbb{E}\bigg(\frac{\rf{I_{(0,\infty)}(L(t))}}{L(t)} \int_{0}^{t}\psi(X(s)) dL(s)\bigg).
\end{equation*}
\end{remark}

\subsection{Numerical time-averaging and ensemble-averaging estimators}\label{subsection4.3}
 We now come to the central part of this section. We first introduce the discrete time-averaging estimators $\hat{\varphi}_{N}$ and $ \hat{\psi}_{N}^{'}$ to compute  $\bar{\varphi}$ and $\bar{\psi}^{'}$ (see (\ref{varphibar}) and (\ref{sidash1})), respectively, as well as $\hat{\kappa}_{N}$ and $\hat{\psi}_{N} $ to estimate $\kappa$ and $ \bar{\psi}$ (see (\ref{3.17}) and (\ref{psiestas})), respectively:
\begin{align}
    \hat{\varphi}_{N} &= \frac{1}{N}\sum\limits_{k=0}^{N-1}\varphi_{k}, \label{3.32}\\
\hat{\kappa}_{N} &=
     \frac{2}{Nh}\sum\limits_{k=0}^{N-1}\frac{r_{k+1}}{\alpha_{k+1}^{\pi}}I_{\bar{G}^{c}}(X_{k+1}^{'}),\label{3.33} \\
     \hat{\psi}_{N}   &=\frac{2}{Nh}\sum\limits_{k=0}^{N-1}\frac{r_{k+1}\psi_{k+1}^{\pi}}{\alpha_{k+1}^{\pi}}I_{\bar{G}^{c}}(X_{k+1}^{'}), \label{5.29}\\
     \hat{\psi}_{N}^{'} &= \frac{\hat{\psi}_{N}}{\hat{\kappa}_{N}} \rf{I_{(0,\infty)}(\hat{\kappa}_{N})} = \frac{ \sum\limits_{k=0}^{N-1}\frac{r_{k+1}\psi_{k+1}^{\pi}}{\alpha_{k+1}^{\pi}}I_{\bar{G}^{c}}(X_{k+1}^{'})}{\sum\limits_{k=0}^{N-1}\frac{r_{k+1}}{\alpha_{k+1}^{\pi}}I_{\bar{G}^{c}}(X_{k+1}^{'})}\rf{I_{(0,\infty)}(\hat{\kappa}_{N})},\label{5.39}
 \end{align}
where $X_{k}$ is the approximation of  RSDEs (\ref{rsde}) according to Algorithm~\ref{algorithm2.1},  $X^{'}_{k+1} $ is the auxiliary step from Algorithm~\ref{algorithm2.1}, $r_{k+1} = \dist(X_{k+1}^{'},X_{k+1}^{\pi})$, $\varphi_{k} = \varphi(X_{k}) $, $ \psi_{k+1}^{\pi} = \psi(X_{k+1}^{\pi})$, $\alpha_{k+1}^{\pi} = \alpha(X_{k+1}^{\pi})$, and $X_{k+1}^{\pi}$ is the projection of $X_{k+1}^{'}$ on $\partial G$.
In Subsection~\ref{section5.2}, we prove (Theorems~\ref{thrm3.2} and \ref{thrm6.5} and Lemma~\ref{aapl4.4}) that the bias of all the above numerical time-averaging estimators is  ${\cal{O}}(h+1/T)$ and that the second moment of the error of the estimators (Theorems~\ref{variancephi} and \ref{variancepsidash} and Lemma~\ref{aapl4.5}) is  ${\cal{O}}(h^2+1/T)$.

Further, we take the expectation $\mathbb{E}(\varphi_{N})$ as a discretized ensemble-averaging estimator for computing $\bar{\varphi}$. For calculating $\bar{\psi}^{'}$, we take the following as discretized ensemble-averaging estimators:
\begin{equation} \label{6.65n}
      \left.\mathbb{E}\biggl(\sum\limits_{k=0}^{N-2}\frac{r_{k+1}\psi_{k+1}^{\pi}}{\alpha_{k+1}^{\pi}}I_{\bar{G}^{c}}(X_{k+1}^{'})\biggr)\middle/ \mathbb{E}\biggl(\sum\limits_{k=0}^{N-2}\frac{r_{k+1}}{\alpha_{k+1}^{\pi}}I_{\bar{G}^{c}}(X_{k+1}^{'})\biggr) \right.
      \end{equation}
      or
 \begin{align}  \label{6.66n}
      \left.\mathbb{E}\Bigg(\biggl( \sum\limits_{k=0}^{N-1}\frac{r_{k+1}\psi_{k+1}^{\pi}}{\alpha_{k+1}^{\pi}}I_{\bar{G}^{c}}\big(X_{k+1}^{'}\big)\;\middle/\;\sum\limits_{k=0}^{N-1}\frac{r_{k+1}}{\alpha_{k+1}^{\pi}}I_{\bar{G}^{c}}\big(X_{k+1}^{'}\big)\biggr)\rf{I_{(0,\infty)}(\hat{\kappa}_{N})}\Bigg).\right.
\end{align}
In Subsection~\ref{nea} we prove (Theorem~\ref{theorem6.14}) that the error of the estimator $\mathbb{E}(\varphi_{N})$ is  ${\cal{O}}(h+e^{-\lambda T})$ for some $\lambda >0$ and (Theorem~\ref{theorem3.13}) that the error of the estimator (\ref{6.65n}) for $\bar{\psi}^{'}$  is ${\cal{O}}(h+1/T)$. \rf{The upper index of the sums in (\ref{6.65n}) is $N-2$ due to the error analysis (see Remark~\ref{rem:Nminus2}).}
The error of (\ref{6.66n}) is same as the bias of (\ref{5.39}).

\begin{remark}\label{remark5.2}
At the end of Subsection \ref{subsection4.2.1}, $\frac{\rf{I_{(0,\infty)}(L(t))}}{L(t)}\int_{0}^{t}\psi(X(s))dL(s)$
was introduced as the continuous time-averaging estimator to calculate the expectation with respect to the invariant law of $\tilde{X}(t)$, i.e. $\tilde{\psi}$. To approximate $\tilde{\psi}$, we take the discrete time-averaging estimator $ \tilde{\psi}_{N}$  as
\rf{
\begin{align*}
        \tilde{\psi}_{N} &= \frac{\tilde{Z}_{1,N}}{\tilde{Z}_{2,N}}I_{(0,\infty)}(\tilde{Z}_{2,N}),
 \end{align*}}
 and discretized ensemble averaging estimators as
 \rf{
\begin{equation*}
      \frac{\mathbb{E}(\tilde{Z}_{1,N-1})}{\mathbb{E}(\tilde{Z}_{2,N-1})},\;\;\; \text{or}\;\;\;
  \mathbb{E}\bigg(\frac{\tilde{Z}_{1,N}}{\tilde{Z}_{2,N}}I_{(0,\infty)}(\tilde{Z}_{2,N})\bigg),
\end{equation*}}
where $ \rf{\tilde{Z}_{1,N}} = \sum\limits_{k=0}^{N-1}r_{k+1}\psi_{k+1}^{\pi}I_{\bar{G}^{c}}(X_{k+1}^{'}) $ and $\rf{\tilde{Z}_{2,N}}= \sum\limits_{k=0}^{N-1}r_{k+1}I_{\bar{G}^{c}}(X_{k+1}^{'})$.
\end{remark}
\subsection{Error analysis for numerical time-averaging estimators}  \label{section5.2}

In this subsection our aim is to establish the closeness of numerical time-averaging estimators  obtained through the approximation of RSDEs (\ref{rsde}) by Algorithm~\ref{algorithm2.1} to their corresponding ergodic limits. 
Our approach to the error analysis of numerical time-averaging estimators is analogous to the one used in \cite{41} in the case of the usual SDEs.

For brevity, we write $u_{k}=u(X_{k})$, $u_{k}^{\prime}= u(X_{k}^{'}) $, $b_{k} = b(X_{k})$, $\sigma_{k} = \sigma(X_{k}) $, $\nu_{k}^{\pi}= \nu(X_{k}^{\pi}) $, $\alpha_{k}^{\pi} = \alpha(X_{k}^{\pi})$, $\phi_{1,k} = \phi_{1}(X_{k}) $ and $\phi_{2,k}^{\pi} = \phi_{2}(X_{k}^{\pi})$, where $X_{k}^{'}$ and $X_{k}^{\pi}$ are as introduced in Algorithm \ref{algorithm2.1}.  From Proposition 1.17 in \cite{19}, it is known that under Assumption~\ref{as:3}, $u(x) \in C^{4}( \bar{G})$ can be extended to a function in $C^{4}(\bar{G}\cup\bar{G}_{-r})$, where $G_{-r}$ was introduced in the beginning of Subsection~\ref{subsection3.4}. 
This also implies that $u(x)$ and its derivatives up to fourth order are uniformly bounded on $\bar{G}\cup\bar{G}_{-r}$.

We write $u(X_{k+1}) - u(X_{k})$ as
\begin{align*}
    u_{k+1} - u_{k} &= u_{k+1} - u_{k+1}^{\prime} + u_{k+1}^{\prime} - u_{k}
 = (u_{k+1}- u_{k+1}^{\prime})I_{\bar{G}^{c}}(X_{k+1}^{'})  + u_{k+1}^{\prime} - u_{k} 
\end{align*}
since $u_{k+1} = u_{k+1}^{\prime}$ if $X_{k+1}^{'} \in \bar{G}$. Using Taylor expansion, we get
\begin{align}
    u_{k+1} - u_{k} 
    & =  \Big(2r_{k+1}\phi_{2,k+1}^{\pi} + R_{5,k+1} + R_{6,k+1} \Big)I_{\bar{G}^{c}}(X_{k+1}') + h\phi_{1,k} + R_{7,k+1} + R_{8,k+1}, \label{5.28}
\end{align}
where
\begin{align*}
        R_{5,k+1} &= (r_{k+1}^{3}/6)D^{3}u(X_{k+1}^{\pi} + \alpha_{1}r_{k+1}\nu_{k+1}^{\pi})[\nu_{k+1}^{\pi},\nu_{k+1}^{\pi},\nu_{k+1}^{\pi}], \\
        R_{6\rs{,}k+1} &= -(r_{k+1}^{3}/6)D^{3}u(X_{k+1}^{\pi}- \alpha_{2}r_{k+1}\nu_{k+1}^{\pi})[\nu_{k+1}^{\pi},\nu_{k+1}^{\pi},\nu_{k+1}^{\pi}], \\
                R_{7,k+1} &= (h^{2}/2)D^{2}u_{k}[b_{k},b_{k}] + (h^{2}/2)D^{3}u_{k}[b_{k},\sigma_{k}\xi_{k+1},\sigma_{k}\xi_{k+1}] + (h^{3}/6)D^{3}u_{k}[b_{k},b_{k},b_{k}] \\ & \;\;\;\; + (1/24)D^{4}u(X_{k}+\alpha_{3}\delta_{k+1})[\delta_{k+1},\delta_{k+1},\delta_{k+1},\delta_{k+1}]\\
            R_{8,k+1} &= h^{1/2}(\sigma_{k}\xi_{k+1}\cdot\nabla)u_{k} + (h/2)\big(D^{2}u_{k}[\sigma_{k}\xi_{k+1},\sigma_{k}\xi_{k+1}] - (a_{k}:\nabla\nabla)u_{k}\big)
                     \\ & + h^{3/2}D^{2}u_{k}[b_{k},\sigma_{k}\xi_{k+1}] + (h^{3/2}/6)D^{3}u_{k}[\sigma_{k}\xi_{k+1},\sigma_{k}\xi_{k+1},\sigma_{k}\xi_{k+1}] \\ & \;\;\;\; + (h^{5/2}/2)D^{3}u_{k}[b_{k},b_{k},\sigma_{k}\xi_{k+1}],
    \end{align*}
    with \rs{$\delta_{k+1} = b_{k}h + \sigma_{k}\xi_{k+1}h^{1/2}$} and $\alpha_{1}, \; \alpha_{2}, \; \alpha_{3} \in (0,1)$.

    Let us estimate the error terms $ R_{j,k+1}$, $ j=5,\dots, 8$.
It is not difficult to deduce that  for any $k = 0, \dots,N-1$:
\begin{equation}\mathbb{E}(R_{8,k+1}|X_{k}) = 0.  \label{R8}
\end{equation}
 Recall that the functions $b(x)$, $\sigma(x) \in C^{2}$ for $x \in \bar{G}$  and $u(x) \in C^{4}$,  $ x \in   \bar{G}\cup \bar{G}_{-r}$. Then, using boundedness of $\xi_{k+1}$, we obtain for all $k=0,\dots ,N-1$:
    \begin{align}
     \rf{|R_{i,k+1}|}&\leq Cr_{k+1}^{3}\;\; a.s.,\;\;\;\;\; i=5,6, \label{eq:5.26} \\
     \rf{|R_{7,k+1}|} & \leq Ch^{2}\;\; a.s., \label{eq:5.27}
    \end{align}
where $C$ is a positive constant independent of $T$ and $h$. Proceeding in the same way as above, we can also get the following estimates for all $k=0,\dots ,N-1$:
\begin{align}
    \rf{|R_{i,k+1}R_{j,k+1}|}&\leq Cr_{k+1}^{6}\;\;a.s.,\;\;\;\;\;\; i,j=5,6, \label{b3.35}\\
    \rf{|R_{7,k+1}^{2}| }& \leq Ch^{4}\;\;a.s.,\;\;\;  \label{b3.36}\\
    \rf{|R_{8,k+1}^{2}|} & \leq Ch\;\;a.s.,\;\;\; \label{b3.37}
\end{align}
where $C$ is a positive constant independent of $T$ and $h$.

 The next lemma is related to an estimate of the number of steps which the Markov chain $X_{k}^{'}$ spends in the layer $G_{-r}$. We note that in comparison with the analogous Lemma~\ref{bl} the estimate in the below lemma has explicit (linear) dependence on time $T$, which is important for the error analysis of the numerical time-averaging estimators introduced in the previous subsection.

\begin{lemma} \label{sjlm}
Under Assumptions \ref{as:3} and \ref{as:5}, the following inequality holds 
for sufficiently large $T$:
\begin{equation*}
\mathbb{E}\bigg(\sum\limits_{k=1}^{N}r_{k}I_{G_{-r}}(X_{k}^{'})  \bigg)\leq CT, \label{eq:42}
\end{equation*}
where $C$ is a positive constant independent of $T$ and $h$.
\end{lemma}
\begin{proof}
  The proof follows the same arguments as those in Lemma~\ref{bl}, but with (cf. (\ref{utx}))
 \begin{equation*}
    U(t,x) =
    \begin{cases}
    0, & (t,x) \in \{T\}\times \big(\bar{G} \cup \bar{G}_{-r}\big),\\
    \rs{\big(}K(T-t) + \rf{w(x)}\rs{\big)/l}, & (t,x) \in [0, T-h] \times \big(\bar{G} \cup \bar{G}_{-r}\big),
    \end{cases}
\end{equation*}
and with an appropriate choice of $K$. \rs{Here} $w(x)$ is from (\ref{wx}) \rs{and $l$ is the distance between $\mathbb{S}_{l}$ and $\partial G$.}
\end{proof}

We will first consider estimates for bias of the estimator $\hat{\varphi}_{N}$ (see (\ref{3.32})) and second moment of its error  followed by respective  error estimates of $\hat{\kappa}_{N}$, $\hat{\psi}_{N}$ and $\hat{\psi}_{N}^{'}$ (see (\ref{3.33})-(\ref{5.39})).

\begin{theorem} \label{thrm3.2}
Under Assumptions \ref{as:3} and \ref{as:4}-\ref{as:5}, the following estimate holds for $\varphi \in C^{2}(\bar{G})$:
\begin{equation}
    \mid \mathbb{E}(\hat{\varphi}_{N}) - \bar{\varphi}\mid  \leq C\Big( h+\frac{1}{T}\Big), \label{est1}
 \end{equation}
where $\hat{\varphi}_{N}$ is from (\ref{3.32}), $\bar{\varphi}$ is from (\ref{q2.24}), and $C>0$ is independent of $T$ and $h$.
\end{theorem}
\begin{proof}
Consider the Neumann problem for Poisson equation, (\ref{n311})-(\ref{n312}), i.e., the problem (\ref{gpe1})-(\ref{gpe2}) with $\phi_{1}(x) = \varphi(x) - \bar{\varphi}$ and $\phi_{2}(z) = 0$. Then (\ref{5.28}) becomes
\begin{align*}
    u_{k+1} - u_{k}     & = h( \varphi_{k} -\bar{\varphi})  +  R_{7,k+1} + R_{8,k+1} + \Big(R_{5,k+1} + R_{6,k+1} \Big)I_{\bar{G}^{c}}(X_{k+1}^{'}),
\end{align*}
and, summing over the first $N$ terms, we get
\begin{align} \label{eqn3.45}
    \sum\limits_{k=0}^{N-1}\Big(u_{k+1}-u_{k}\Big) &= h\sum\limits_{k=0}^{N-1}\Big(\varphi_{k}-\bar{\varphi}\Big) + \sum\limits_{k=0}^{N-1}\Big(R_{7,k+1} + R_{8,k+1}\Big) \nonumber  \\ & \;\;\;\; + \sum\limits_{k=0}^{N-1}\Big(\big(R_{5,k+1} + R_{6,k+1}\big)I_{\bar{G}^{c}}\big(X_{k+1}^{'}\big)\Big),
\end{align}
whence by reordering terms, taking expectation on both sides, dividing by $T=Nh$, and using (\ref{R8}), we obtain
\begin{align*}
   |\mathbb{E}(\hat{\varphi}_{N}) - \bar{\varphi}| & \leq \frac{|\rs{\mathbb{E}}u_{N} - u_{0}|}{T} + \frac{1}{Nh}\sum\limits_{k=0}^{N-1}\big|\mathbb{E}\rf{R_{7,k+1}}\big| \\ & \;\;\;\; + \frac{1}{T}\sum\limits_{k=0}^{N-1}\mathbb{E}\Big(\rf{\big(R_{5,k+1} + R_{6,k+1}\big)} I_{\bar{G}^{c}}\big(X_{k+1}^{'}\big)\Big).
\end{align*}
Notice that $|\rs{\mathbb{E}}u_{N}-u_{0}| \leq C $  and using (\ref{eq:5.26})-(\ref{eq:5.27}), we obtain
\begin{align*}
    |\mathbb{E}(\hat{\varphi}_{N}) - \bar{\varphi}| & \leq C\Big(\frac{1}{T} + h \Big) + \frac{C}{T}h\mathbb{E}\bigg(\sum\limits_{k=0}^{N-1}r_{k+1}I_{\bar{G}^{c}}\big(X_{k+1}^{'}\big)\bigg),
\end{align*}
which by Lemma \ref{sjlm} gives the quoted result.
\end{proof}
\begin{theorem} \label{variancephi}
Under Assumptions \ref{as:3} and \ref{as:4}-\ref{as:5}, the following estimate holds  for $\varphi \in C^{2}(\bar G)$:
\begin{equation*}
     \mathbb{E}\big(\hat{\varphi}_{N} - \bar{\varphi}\big)^{2}  \leq C\Big( h^{2} +\frac{1}{T}\Big), 
 \end{equation*}
where $C$ is a positive constant independent of $T$ and $h$.
\end{theorem}
\begin{proof} The proof is based on two steps. We first square both sides of (\ref{eqn3.45}) and use the estimates (\ref{b3.36})-(\ref{b3.37}). The second step is to apply the following inequality obtained by using (\ref{eq:5.26}) and (\ref{b3.35}):
\begin{align*}
    &\mathbb{E}\bigg(\sum\limits_{k=0}^{N-1} R_{k+1} I_{\bar{G}^{c}}\big(X_{k+1}^{'}\big)\bigg)^{2}
       =  \mathbb{E}\bigg( \sum\limits_{k=0}^{N-1} \rf{R_{k+1}^{2}}I_{\bar{G}^{c}}\big(X_{k+1}^{'}\big)\bigg) \\ & \;\;\;\; + 2\mathbb{E}\bigg(\sum\limits_{k=1}^{N}\mathbb{E}\bigg(\sum\limits_{j = k+1}^{N} R_{j}I_{\bar{G}^{c}}\big(X_{j}^{'}\big) \big| X_{k}^{'} \bigg)  R_{k}I_{\bar{G}^{c}}\big(X_{k}^{'}\big)\bigg)
     \\  & \leq        C h^{5/2} \mathbb{E}\bigg(\sum\limits_{k=0}^{N-1} r_{k+1} I_{\bar{G}^{c}}\big(X_{k+1}^{'}\big)\bigg) + Ch\mathbb{E}\bigg(\sum\limits_{k=1}^{N}\mathbb{E}\bigg(\sum\limits_{j = k+1}^{N}r_{j}I_{\bar{G}^{c}}\big(X_{j}^{'}\big)\big| X_{k}^{'}\bigg)R_{k}I_{\bar{G}^{c}}\big(X_{k}^{'}\big)\bigg),
   \end{align*}
  where $R_{k+1} := R_{5,k+1} + R_{6,k+1}  $ and $C$ is a positive constant independent of $T$ and $h$. Then one can obtain the desired result  using Lemma~\ref{sjlm} and combining the above stated two steps.
\end{proof}

 Assumption~\ref{as:5} suffices for proving the previous two theorems. However, in the subsequent lemmas and theorems \rf{we will need Assumption~\ref{as4.2'}}.

Now we proceed to error analysis for time-averaging estimators \rs{related to} ergodic limits on the boundary.  The next lemma is proved analogously to Theorem~\ref{thrm3.2}.

\begin{lemma}\label{aapl4.4}
Under Assumptions \ref{as:3}, \ref{as:4} and \ref{as4.2'}, the following hold for any $\psi \in C^{3}(\partial G)$:
\begin{align}
    \vert \mathbb{E}(\hat{\kappa}_{N}) - \kappa\vert  &\leq C\Big( h+\frac{1}{T}\Big),\label{kappa} \\
     \vert \mathbb{E}(\hat{\psi}_{N}) - \bar{\psi}\vert  &\leq C\Big( h+\frac{1}{T}\Big), \label{psi}
\end{align}
where $\hat{\kappa}_{N}$ is from (\ref{3.33}), $\kappa$ is from (\ref{eb2.25}), $\hat{\psi}_{N}$ is from (\ref{5.29}), $\bar{\psi}$ is from (\ref{psibar}), and $C$ is a positive constant independent of $T$ and $h$.
\end{lemma}

The next lemma is proved analogously to Theorem~\ref{variancephi}.

\begin{lemma}\label{aapl4.5}
Under Assumptions  \ref{as:3}, \ref{as:4} and \ref{as4.2'}, the following hold  for any $\psi \in C^{3}(\partial G)$:
\begin{align}
     \mathbb{E}\big(\hat{\kappa}_{N} - \kappa\big)^{2}  &\leq C\Big( h^{2}+\frac{1}{T}\Big), \label{kappavariance} \\
      \mathbb{E}\big(\hat{\psi}_{N} - \bar{\psi}\big)^{2}  & \leq C\Big( h^{2}+\frac{1}{T}\Big),  \label{psivariance}
\end{align}
where $C$ is a positive constant independent of $T$ and $h$.
\end{lemma}
\rf{Using Chebyshev's inequality and (\ref{kappavariance}), we have
\begin{equation}
    \mathbb{P}(\hat{\kappa}_{N} = 0) \leq C\Big( h^{2}+\frac{1}{T}\Big),
    \label{eq:kappa0}
\end{equation}
where $C>0$ is independent of $T$ and $h$.}
\begin{theorem}\label{variancepsidash}
Under Assumptions \ref{as:3}, \ref{as:4} and \ref{as4.2'}, the following holds for any $\psi \in C^{3}(\partial G)$:
\begin{equation*}
     \mathbb{E}\big(\hat{\psi}^{'}_{N} - \bar{\psi}^{'}\big)^{2}  \leq C\Big( h^{2}+\frac{1}{T}\Big),
\end{equation*}
where $\hat{\psi}^{'}_{N}$ is from (\ref{5.39}), $\bar{\psi}^{'}$ is from (\ref{q2.26}), and $C$ is a positive constant independent of $T$ and $h$.
\end{theorem}
\begin{proof}
Since $\psi(z) \in C^{3}(\partial G)$, we have (cf. (\ref{3.33}) and (\ref{5.29})):
\begin{align*}
\hat{\psi}_{N}^{2} = \bigg(\rf{\frac{2}{Nh}}\sum\limits_{k=0}^{N-1}\frac{r_{k+1}\psi^{\pi}_{k+1}}{\alpha^{\pi}_{k+1}}I_{\bar{G}^{c}}(X_{k+1}^{'})\bigg)^{2} \leq C \bigg(\rf{\frac{2}{Nh}}\sum\limits_{k=0}^{N-1}\frac{r_{k+1}}{\alpha^{\pi}_{k+1}}I_{\bar{G}^{c}}(X_{k+1}^{'})\bigg)^{2} = C \hat{\kappa}^{2}_{N}.
\end{align*}
Under Assumption \ref{as:4}, $\alpha(z) > 0$ for all $z \in \partial G$ (see (\ref{vrho})) which implies 
\begin{align}
    \big(\hat{\psi}_{N}^{2}/\hat{\kappa}_{N}^{2}\big)\rf{I_{(0,\infty)}(\hat{\kappa}_{N})} \leq C.  \label{tb3.63}
\end{align}
We have
\begin{align*}
&\mathbb{E}\big(\hat{\psi}_{N}^{'} - \bar{\psi}^{'}\big)^{2}  = 
\mathbb{E}\bigg(\frac{\hat{\psi}_{N}\kappa - \bar{\psi}\hat{\kappa}_{N}}{\hat{\kappa}_{N}\kappa}\rf{I_{(0,\infty)}(\hat{\kappa}_{N}) - \frac{\bar{\psi}}{\kappa}I_{\{\hat{\kappa}_{N} = 0\} }} \bigg)^{2} \\ & = \mathbb{E}\bigg(\frac{(\hat{\psi}_{N} - \bar{\psi})\hat{\kappa}_{N} - (\hat{\kappa}_{N} - \kappa)\hat{\psi}_{N}}{\hat{\kappa}_{N}\kappa}\rf{I_{(0,\infty)}(\hat{\kappa}_{N})} \rf{- \frac{\bar{\psi}}{\kappa}I_{\{\hat{\kappa}_{N} = 0\} }}\bigg)^{2} \\& \leq \rf{C}\bigg(\mathbb{E}\big(\hat{\psi}_{N} - \bar{\psi}\big)^{2} + \mathbb{E}\bigg(\big(\hat{\kappa}_{N} - \kappa\big)^{2}\Big(\frac{\hat{\psi}_{N}}{\hat{\kappa}_{N}}\Big)^{2}\rf{I_{(0,\infty)}(\hat{\kappa}_{N})}\bigg) \rf{+ \mathbb{P}(\hat{\kappa}_{N} = 0)}\bigg).
\end{align*}
Using (\ref{tb3.63}), we get
\begin{align*}
    \mathbb{E}\big(\hat{\psi}_{N}^{'} - \bar{\psi}^{'}\big)^{2} & \leq C \big(\mathbb{E}\big(\hat{\psi}_{N} - \bar{\psi}\big)^{2} + \mathbb{E}\big(\hat{\kappa}_{N} - \kappa\big)^{2} + \rf{\mathbb{P}(\hat{\kappa}_{N} = 0)}\big),
\end{align*}
which by (\ref{kappavariance}), (\ref{psivariance}) and \rf{(\ref{eq:kappa0})} gives the desired result.
\end{proof}

\begin{theorem}\label{thrm6.5}
Under Assumptions  \ref{as:3}, \ref{as:4} and \ref{as4.2'}, the following holds for any $\psi \in C^{3}(\partial G)$:
\begin{equation}
    \mid \mathbb{E}(\hat{\psi}^{'}_{N}) - \bar{\psi}^{'}\mid  \leq C\Big( h+\frac{1}{T}\Big), \label{g5.57}
\end{equation}
where $C$ is a positive constant independent of $T$ and $h$.
\end{theorem}
\begin{proof}
We have
\begin{align*}
&|\mathbb{E}\hat{\psi}_{N}^{'} - \bar{\psi}^{'}| = \bigg|\mathbb{E}\Big(\frac{\hat{\psi}_{N}}{\hat{\kappa}_{N}}\rf{I_{(0,\infty)}(\hat{\kappa}_{N})}\Big) - \frac{\bar{\psi}}{\kappa}\bigg|
 = \bigg|\mathbb{E}\Big(\frac{\hat{\psi}_{N}\kappa - \bar{\psi}\hat{\kappa}_{N}}{\hat{\kappa}_{N}\kappa}\rf{I_{(0,\infty)}(\hat{\kappa}_{N})}\Big) -\rf{\frac{\bar{\psi}}{\kappa}I_{\{\hat{\kappa}_{N} =0\}}}\bigg| \\ & = \bigg|\mathbb{E}\Big(\frac{(\hat{\psi}_{N} - \bar{\psi})\hat{\kappa}_{N} + (\kappa-\hat{\kappa}_{N})\hat{\psi}_{N}}{\hat{\kappa}_{N}\kappa}\rf{I_{(0,\infty)}(\hat{\kappa}_{N})}\Big)-\rf{\frac{\bar{\psi}}{\kappa}I_{\{\hat{\kappa}_{N} =0\}}}\bigg|
\leq \frac{1}{\kappa}|\mathbb{E}(\hat{\psi}_{N})-\bar{\psi}| \\ & \;\;\;\; + \frac{1}{\kappa}\Big|\mathbb{E}\Big((\kappa-\hat{\kappa}_{N})\frac{\hat{\psi}_{N}}{\hat{\kappa}_{N}}\rf{I_{(0,\infty)}(\hat{\kappa}_{N})}\Big)\Big| + \bar{\psi}^{'}\mathbb{P}(\hat{\kappa}_{N}=0) \\
& \leq \frac{1}{\kappa}|\mathbb{E}(\hat{\psi}_{N})-\bar{\psi}|  + \frac{1}{\kappa}\big|\mathbb{E}(\kappa -\hat{\kappa}_{N})(\hat{\psi}_{N}^{'} - \bar{\psi}^{'})\big| + \rs{\frac{\bar{\psi}^{'}}{\kappa}}\big|\mathbb{E}(\kappa -\hat{\kappa}_{N})\big|+\rf{\bar{\psi}^{'}\mathbb{P}(\hat{\kappa}_{N}=0)}.
\end{align*}
By the Cauchy$-$Bunyakovsky$-$Schwarz inequality, \rs{we get $\big|\mathbb{E}(\kappa -\hat{\kappa}_{N})(\hat{\psi}_{N}^{'} - \bar{\psi}^{'})\big| \leq  \big(\mathbb{E}(\kappa -\hat{\kappa}_{N})^{2}\big)^{1/2}\big(\mathbb{E}(\hat{\psi}_{N}^{'} - \bar{\psi}^{'})^{2}\big)^{1/2}$.} Then, using (\ref{kappa})-(\ref{kappavariance}), \rf{(\ref{eq:kappa0})} and Theorem \ref{variancepsidash}, we obtain (\ref{g5.57}).
\end{proof}

\begin{remark}
\rf{In Remark~\ref{remark5.2}, we introduced the numerical time-averaging estimators for \rs{the ergodic limit}  $\tilde{\psi}$, for which we can prove the following error estimates analogously to the proofs of Theorems~\ref{variancepsidash} and \ref{thrm6.5} for $\bar{\psi}^{'}$,
\begin{align*}
    |\mathbb{E}\tilde{\psi}_{N} - \tilde{\psi}| &\leq C\bigg(h +\frac{1}{T}\bigg),\;\;\;\;\;\;
   \mathbb{E}(\tilde{\psi}_{N} - \tilde{\psi})^{2}  \leq C\bigg(h^{2} + \frac{1}{T}\bigg),
\end{align*}
where $C>0$ is independent of $T$ and $h$.}

\end{remark}

\subsection{Error analysis for numerical ensemble-averaging estimators}\label{nea}
In this subsection our aim is to estimate errors when we approximate the stationary averages $\bar{\varphi}$ and $\bar{\psi}^{'}$ using discretized ensemble-averaging estimators introduced in Subsection~\ref{subsection4.3}. We split our discussion in two parts. Firstly, we will consider the error of the ensemble-averaging estimator for $\bar{\varphi}$. The proof exploits the backward Kolmogorov (parabolic) equation (cf. the similar approach in \cite{42} for the case of SDEs in $\mathbb{R}^d$). In the second part, we will estimate the error of the numerical ensemble-averaging estimator for $ \bar{\psi}^{'}$. We take $t_{0} =0 $, therefore $Q = [0,T) \times G $.

Consider the parabolic equation with non-zero terminal condition and homogeneous boundary condition:
\begin{align}
    \frac{\partial u}{\partial t}(t,x) + \mathscr{A} u(t,x) &= 0, \;\;\;\; (t,x) \in [0,T)\times \bar{G}, \label{eq:47}\\
    u(T,x) &= \varphi(x), \;\;\;\; x \in \bar{G}, \label{eq:48}\\
    \frac{\partial u}{\partial \nu}(t,z) &= 0, \;\;\;\; 0 \leq t < T,\;\;\; z \in \partial G. \label{eq:49}
\end{align}

If  Assumptions~\ref{as:3} and \ref{as:4}-\ref{as:5} hold with $\varphi(x) \in C^{4}(\bar{G})$ then the solution $u(t,x)$ of the problem (\ref{eq:47})-(\ref{eq:49})  belongs to $ C^{2,4}([0,T)\times \bar G)$  \cite{3,64}. 
Further, $u(t,x) \in C(\bar{Q})$ 
\cite{lunbook}. 
The probabilistic representation of the solution is given by (see Section~\ref{section3.1}):
\begin{equation*}
    u(0,x) = \mathbb{E}\varphi(X(T)), 
\end{equation*}
 where $ X(s) $ is the solution of the  RSDEs (\ref{rsde}).

 \noindent  Introduce the norm for $t \in [0,T)$
\begin{equation*}
    \langle u(t,x) \rangle_{\bar{G} }^{p} = \sum\limits_{2i+|j|=p}\max\limits_{\bar{G}}|D_{t}^{i}D_{x}^{j}u(t,x)|,
\end{equation*}
where $D^{j}_{x}u = \frac{\partial^{|j|}u }{\partial x^{j_{1}}\dots x^{j_{d}}} $ and $j$ is a multi-index. We make the following natural assumption for the solution of the Neumann problem (\ref{eq:47})-(\ref{eq:49}).
\begin{assumption}\label{a7}
There are positive constants $C $ and $\lambda$ independent of $T$ such that the following bound holds for all $t \in [0,T)$:
\begin{equation}\label{r5.67}
\sum\limits_{p=1}^{4}\langle u(t,x)\rangle_{ \bar{G} }^{p} \leq Ce^{-\lambda (T-t)}.
\end{equation}
\end{assumption}
A bound of the form (\ref{r5.67}) but with supremum over $ G$ is given in Theorem~4.1 of \cite{65} under Assumptions \ref{as:3}, \ref{as:4} and \ref{as:5}. We also mention that such a decaying bound for the gradient of the solution was proved for the heat equation with homogeneous Neumann boundary condition in \cite{71}. A bound similar to (\ref{r5.67}) is proved for the Cauchy case in \cite{42}.

From \cite[Proposition 1.17]{19}, if Assumptions \ref{as:3}, 
\ref{as:4}, \ref{as:5} and \ref{a7} hold then the solution $u(t,x) \in C^{2,4}([0,T)\times \bar{G})$ of problem (\ref{eq:47})-(\ref{eq:49}) can be extended to a function in $C^{2,4}([0,T)\times \bar{G}\cup\bar{G}_{-r})$ satisfying the bound (\ref{r5.67}).

Under Assumptions \ref{as:3}, 
\ref{as:4}, \ref{as:5} and \ref{a7}, the following error bounds hold for all $k = 0,1\dots,N-2$:
\begin{align}
\Big|\mathbb{E}\Big(u_{k+1}^{\prime}-u_{k}\big|X_{k} \Big)\Big| &\leq C h^{2}e^{-\lambda (T- t_{k})},\label{3.5}\\
    \rf{ \big| u_{k+1} - u_{k+1}^{\prime}\big|} &\leq Chr_{k+1}e^{-\lambda (T-t_{k+1})}I_{\bar{G}^{c}}(X_{k+1}^{'})\;\;\;\; a.s., \label{3.6}
\end{align}
where $C>0$ is independent of $T$ and $h$. The above bounds directly follow from Lemmas~\ref{lemma2.2} and \ref{lemma2.3} together with Assumption \ref{a7} if we note \rs{that} under Assumption \ref{a7}, 
the terms appearing in $R_{1,k+1}$ in Lemma~\ref{lemma2.2} will be bounded by $Ch^{2}e^{-\lambda(T- t_{k})}$  and the terms $R_{3,k+1}$ and $R_{4,k+1}$ in Lemma~\ref{lemma2.3}  by $Cr_{k+1}^{3}e^{-\lambda(T- \rs{t_{k+1}})}$.

Next, we state an appropriate lemma related to the average number of steps which the chain $X_{k}^{'}$ spends in $G_{-r}$.

\begin{lemma} \label{lemmaonexp}
  Under Assumptions \ref{as:3} and \ref{as:5}, the following inequality holds for \rf{any} $\lambda > 0$:
  \begin{equation}
      \mathbb{E}\bigg(\sum\limits_{k=1}^{N}e^{\lambda t_{k}}r_{k}I_{G_{-r}}(X_{k}^{'})\bigg) \leq Ce^{\lambda T}, \label{5.70}
  \end{equation}
  where $C$ is a positive constant independent of $T$ and $h$.
\end{lemma}
\begin{proof}
\rf{We} take $U(t,x)$ in Lemma~\ref{bl} as (cf. (\ref{utx}))
 \begin{equation*}
    U(t,x) =
    \begin{cases}
    0, & (t,x) \in \{T\}\times \big(\bar{G} \cup \bar{G}_{-r}\big),\\
   \rf{ Ke^{\lambda T} - e^{\lambda t}\big( K-  \frac{1}{l}w(x)\big)}, & (t,x) \in [0, T-h] \times \big(\bar{G} \cup \bar{G}_{-r}\big),
    \end{cases}
\end{equation*}
\rf{where $w(x)$ and $l$ are as in Lemma~\ref{bl}. Following   arguments similar to  the ones used in the proof of Lemma~\ref{bl} with an appropriate choice of $K$, we can obtain the desired result.}
\end{proof}

\begin{lemma} \label{theorem3.9}
  Under Assumptions \ref{as:3}, 
  \ref{as:4}, \ref{as:5} and \ref{a7}, the following inequality holds for $\varphi \in C^{4}(\bar{G})$:
 \begin{equation*}
    \mid \mathbb{E}(\varphi(X_{N})\rs{)} - u(0,x)\mid \;  \leq \; Ch,
\end{equation*} 
where $u(0,x)$ is the solution of (\ref{eq:47})-(\ref{eq:49}) and $C > 0$ is 
independent of $T$ and $h$.
\end{lemma}
\begin{proof}
Analogously to the proof of Theorem \ref{conp}, we have (cf. (\ref{ep4.16})):
\begin{align*}
    \big| \mathbb{E}\big(&\varphi(X_{N})\big) - u(0, X_{0})\big|
     \leq  \Bigg| \mathbb{E}\Bigg( \sum\limits_{k=0}^{N-2}\big(u_{k+1} - u_{k+1}^{\prime} \big)I_{G_{-r}}\big(X_{k+1}^{'}\rs{\big)} \Bigg)\Bigg| \\ & \;\;\;\; + \Bigg| \mathbb{E}\Bigg(  \sum \limits_{k=0}^{N-2}\mathbb{E}\Big( u_{k+1}^{\prime}  - u_{k}\Big| X_{k}\Big) \Bigg)\Bigg|  + |\mathbb{E}\varphi(X_{N}) - u(t_{N-1},X_{N-1})|,
\end{align*}
 then using (\ref{3.5}) and (\ref{3.6}), we obtain
\begin{align}
 \big|\mathbb{E}\big(\varphi(X_{N}) \big) -  \mathbb{E}\big(\varphi &(X(T))\big) \big|  \leq C h\Bigg| \mathbb{E}\Bigg( \sum\limits_{k=0}^{N-2}e^{-\lambda(T- t_{k+1})}r_{k+1}I_{G_{-r}}\big( X_{k+1}^{'}\big) \Bigg)\Bigg| \nonumber
 \\ & + Ch^{2}e^{-\lambda T}\sum\limits_{k=0}^{N-2}e^{\lambda t_{k}}  + |\mathbb{E}(\varphi(X_{N-1}) - u(t_{N-1},X_{N-1}))| + Ch.
 \end{align}
Notice that due to Taylor expansion of $u(t_{N},X_{N-1})$ around $(t_{N-1},X_{N-1})$, we have $|u(t_{N},X_{N-1})-u(t_{N-1},X_{N-1})| \leq Ch$, where $C$ is independent of $T$ under Assumption~\ref{a7}. Using Lemma~\ref{lemmaonexp}, we obtain
\begin{equation}
    \big|\mathbb{E}\big(\varphi(X_{N})\big)
    -\mathbb{E}\big(\varphi(X_{0,x}(T))\big)\big| \leq Ch. \label{eq:3.47}
\end{equation}
\end{proof}

The proof of the next theorem directly follows from combining (\ref{eq:34}) and (\ref{eq:3.47}).
\begin{theorem} \label{theorem6.14}
Under Assumptions \ref{as:3}, 
\ref{as:4}-\ref{as:5} and \ref{a7}, the following inequality holds for $\varphi \in C^4(\bar{G})$:
\begin{equation}\label{e5.73}
    \big| \bar{\varphi} - \mathbb{E}\big( \varphi(X_N) \big)\big| \leq C\Big(h  + e^{-\lambda T}\Big),
\end{equation}
where $C$ is a positive constant independent of $T$ and $h$.
\end{theorem}

 Now we move to estimating the error of the discretized ensemble-averaging estimators corresponding to (\ref{6.65n}). Notice that the error bound for the discretized ensemble-averaging estimator (\ref{6.66n}) has already been proved in Theorem \ref{thrm6.5}.

 Consider the parabolic equation with zero terminal condition and non-zero Neumann boundary condition:
\begin{align}
    \frac{\partial u}{\partial t}(t,x) + \mathscr{A} u(t,x) &= 0, \;\;\;\; (t,x) \in [0,T)\times \bar{G},  \label{eq:3.64} \\
             u(T,x) &= 0, \;\;\;\; x \in \bar{G}, \label{eq:3.65}\\
    \frac{\partial u}{\partial \nu}(t,z) &= \frac{\psi(z) - \bar{\psi}^{'}}{\alpha(z)}, \;\;\;\; 0\leq t < T,\;\;\; z \in \partial G.  \label{Nnonhb}
\end{align}
If Assumptions~\ref{as:3}, \ref{as:4} and \ref{as4.2'} hold with $\psi(z) \in C^{3}(\partial G)$ then there exists a solution $u(t,x) \in C^{2,4}([0,T)\times \bar G)$ of the problem (\ref{eq:3.64})-(\ref{Nnonhb}) \cite{3,64}. Further, $u(t,x) \in C(\bar{Q})$ \cite{lunbook}.
The probabilistic representation of the solution is given by (cf. (\ref{eq:7})):
\begin{equation}
    u(\rf{t},x) = \mathbb{E}\bigg(\int_{0}^{T\rf{-t}}\frac{\bar{\psi}^{'}-\psi(X(s))}{\alpha(X(s))}dL(s)\bigg), \label{y5.77}
\end{equation}
 where $ X(s) $ is the solution of the RSDEs (\ref{rsde}).

 We make the following natural assumption for the solution of (\ref{eq:3.64})-(\ref{Nnonhb}).
\begin{assumption} \label{as10}
The following inequality holds uniformly for $(t,x) \in [0,T)\times \bar{G}$:
\begin{equation}
   \sum\limits_{p=0}^{4}\sum\limits_{2i+|j|=p}\sup\limits_{[0,T)\times \bar{G}}|D_{t}^{i}D_{x}^{j}u(t,x)|\leq C, \label{equation3.75}
\end{equation}
where  $C$ is a positive constant independent of $T$.
\end{assumption}
With respect to Assumption~\ref{as10}, we note that the solution of the problem (\ref{eq:3.64})-(\ref{Nnonhb}) converges to the solution of a stationary problem \cite{70,69} when time $T$ goes to infinity (cf. the case of the problem (\ref{eq:47})-(\ref{eq:49}) where the solution converges to a constant). Also, note that the solution $u(t,x)$ can be expanded asymptotically in powers of $1/(T-t)$ (see Theorem~3 and Section 7 in \cite{70} for more details).

 From \cite[Proposition 1.17]{19}, if Assumptions~\ref{as:3}, \ref{as:4}, \ref{as4.2'} and \ref{as10} hold then the solution $u(t,x) \in C^{2,4}([0,T)\times \bar{G})$  can be extended to a function in $C^{2,4}([0,T)\times \bar{G}\cup\bar{G}_{-r})$ satisfying the bound (\ref{equation3.75}).

Now we state two results which directly follow from Lemmas~\ref{lemma2.2} and \ref{lemma2.3} provided Assumptions~\ref{as:3}, \ref{as:4}, \ref{as4.2'} and \ref{as10} hold along with $\psi(z) \in C^{3}(\partial G)$:
\begin{align}
     \Big|\mathbb{E}\Big(u'_{k+1}-u_{k}  \big|X_{k} \Big)\Big| &\leq C h^{2}, \label{lemma3.10}\\
     \rf{\big| u_{k+1} - u'_{k+1} + Z_{k+1} -Z_{k}\big|} &\leq  Cr_{k+1}^{3}I_{\bar{G}^{c}}(X_{k+1}^{'}), \;\;\; a.s.\label{lemma3.11}
\end{align}
where $C$ is independent of $T$ and $h$, and $k = 0,1\dots,N-2$.

Before proceeding further, we introduce the additional notation
\begin{align*}
    Z_{1,N-1} & = \sum\limits_{k=0}^{N-2}\frac{2r_{k+1}\psi_{k+1}^{\pi}}{\alpha_{k+1}^{\pi}}I_{\bar{G}^{c}}(X_{k+1}^{'}), & Z_{2,N-1} &= \sum\limits_{k=0}^{N-2}\frac{2r_{k+1}}{\alpha_{k+1}^{\pi}}I_{\bar{G}^{c}}(X_{k+1}^{'}).
\end{align*}
Recall that $\bar{\psi}^{'}=\lim\limits_{T \to \infty} \frac{\mathbb{E}Z_{1}(T)}{\mathbb{E}Z_2(T)}$ (see (\ref{sidash2})) and the ensemble-averaging estimator for $\bar{\psi}^{'}$ is equal to $\frac{\mathbb{E}Z_{1,N-1}}{\mathbb{E}Z_{2,N-1}}$ (see (\ref{6.65n})).

\begin{lemma} \label{lemma3.12}
  Under Assumptions~\ref{as:3}, \ref{as:4}, \ref{as4.2'} and \ref{as10}, the following inequality holds for $\psi \in C^{3}(\partial G)$:
 \begin{equation*}
    | \mathbb{E}(\bar{\psi}^{'}Z_{2,N-1} - Z_{1,N-1}) - u(0,x)| \;  \leq \; ChT,  \label{b5.81}
\end{equation*} 
where $u(0,x)$ is from (\ref{y5.77}) and $C$ is a positive constant independent of $T$ and $h$.
\end{lemma}
\begin{proof}
Analogously to the proof of Theorem \ref{conp} (cf. (\ref{ep4.16}) with
$Z_{N-1}=\bar{\psi}^{'}Z_{2,N-1} - Z_{1,N-1}$), we have (note that $u(T,x)=0$ and $Z_0=0$):
\begin{align*}
    &\big| \mathbb{E}\big(\bar{\psi}^{'}Z_{2,N-1} - Z_{1,N-1}\big)  - u(0, x) - Z_{0}\big| \leq  \Bigg| \mathbb{E}\Bigg(  \sum \limits_{k=0}^{N-2}\mathbb{E}\Big( u_{k+1}^{\prime}  - u_{k} \big| X_{k}, Z_{k}\Big) \Bigg)\Bigg|  \\
     &  + \Bigg| \mathbb{E}\Bigg( \sum\limits_{k=0}^{N-2}\Big(u_{k+1} - u_{k+1}^{\prime} + Z_{k+1}-Z_{k}\Big)I_{G_{-r}}\big(X_{k+1}^{'}\big) \Bigg)\Bigg| \\ & + \big|\mathbb{E}\big( u(t_{N},X_{N-1}) - u_{N-1}\big)\big|. 
\end{align*}
Note that  $\big|\mathbb{E}\big( u(t_{N},X_{N-1}) - u_{N-1}\big)\big| \leq Ch$ due to Taylor expansion of $u(t_{N},X_{N-1})$ around $(t_{N-1},X_{N-1}).$ 
Then, using  (\ref{lemma3.10}), (\ref{lemma3.11}) and Lemma \ref{sjlm}, we obtain
\begin{align*}
 \big| \mathbb{E}\big(\bar{\psi}^{'}Z_{2,N-1}-Z_{1,N-1}  \big) & - u(0, x) - Z_{0}\big|  \leq Ch\Bigg| \mathbb{E}\Bigg( \sum\limits_{k=0}^{N-2}r_{k+1}I_{G_{-r}}\big( X_{k+1}^{'}\big) \Bigg)\Bigg| \\ & +  C h^{2}(N-1) +Ch \leq  ChT.\nonumber
  \end{align*}
\end{proof}
\begin{remark} \label{rem:Nminus2}
\rf{The solution $u(t,x)$ of (\ref{eq:3.64})-(\ref{Nnonhb}) belongs to  $C^{2,4}([0,T)\times \bar{G})$ and $C(\bar{Q})$, which does not allow us to use Taylor expansion in the previous lemma as we have done in the proof of Theorem~\ref{conp} (cf. Lemma~\ref{lemma2.2}) in the final step, i.e., from $k=N-1$ to $k=N$. Instead, in the above proof we use the fact $u(T,x) = 0$, take the estimator $\mathbb{E}(Z_{1,N-1})/\mathbb{E}(Z_{2,N-1})$ and expand only in time $t$ in the final step from $k=N-1$ to $k=N$. This also, in particular, explains why we have used $N-2$ and not $N-1$ in the estimator (\ref{6.65n}).}
\end{remark}
\begin{theorem}\label{theorem3.13}
      Under Assumptions~\ref{as:3}, \ref{as:4}, \ref{as4.2'} and \ref{as10}, the following inequality holds for $\psi \in C^{3}(\partial G)$:
 \begin{equation}\label{v5.82}
    \bigg| \frac{\mathbb{E}(Z_{1,N-1})}{\mathbb{E}(Z_{2,N-1})} - \bar{\psi}^{'}\bigg|  \leq \; C\bigg(h+\frac{1}{T}\bigg),
\end{equation} 
where $\bar{\psi}^{'}$ is from (\ref{q2.26}) and $C$ is a positive constant independent of $T$ and $h$.
\end{theorem}
 \begin{proof}
 From Lemma \ref{lemma3.12}, we have $\mathbb{E}\big(Z_{1,N-1} - Z_{2,N-1}\bar{\psi}^{'} +u(0,x)\big) = \mathcal{O}(h)T$,
 whence by rearranging the terms we get
 \begin{equation*}
      \frac{\mathbb{E}(Z_{1,N-1})}{\mathbb{E}(Z_{2,N-1})} - \bar{\psi}^{'}  =  \frac{\mathcal{O}(h)T - u(0,x)}{\mathbb{E}(Z_{2,N-1})}.
 \end{equation*}
 Then, using (\ref{equation3.75}) and (\ref{kappa}), we ascertain
 \begin{equation*}
     \bigg| \frac{\mathbb{E}(Z_{1,N-1})}{\mathbb{E}(Z_{2,N-1})} - \bar{\psi}^{'}\bigg|  \leq \frac{C\big(h+\frac{1}{T}\big)}{\frac{\mathbb{E}(Z_{2,N-1})}{T}} \leq \frac{C\big(h+\frac{1}{T}\big)}{\kappa + \mathcal{O}\big(h + \frac{1}{T}\big) }  \leq C \bigg( h + \frac{1}{T}\bigg).
 \end{equation*}
 \end{proof}
\begin{remark}
\rf{For the discretized ensemble-averaging estimators introduced in Remark~\ref{remark5.2} for $\tilde{\psi}$, we can get the following error bound in the same manner   as we did in Theorem~\ref{theorem3.13} for $\bar{\psi}^{'}$,
\begin{align*}
    \bigg| \frac{\mathbb{E}(\tilde{Z}_{1,N-1})}{\mathbb{E}(\tilde{Z}_{2,N-1})} - \tilde{\psi}\bigg| \leq C\bigg(h+ \frac{1}{T}\bigg),
\end{align*}
where $C> 0$ is independent of $T$ and $h$.}
\end{remark}

 \subsection{Error analysis for a stationary measure of Algorithm \ref{algorithm2.1} \label{section5.5}}

 We can express Algorithm~\ref{algorithm2.1} (its part for $X_k$) as the following recursive formula
 \begin{align*}
    X_{k+1}  &= F(X_{k},\rf{\xi_{k+1}}) + 2r_{0}(F(X_{k},\rf{\xi_{k+1}}))\nu(F(X_{k},\rf{\xi_{k+1}}))I_{\bar{G}^{c}}(F(X_{k},\rf{\xi_{k+1}})) \\ &:= H(X_{k},\rf{\xi_{k+1}}),
 \end{align*}
 where $F(x,\rf{\xi})= x + b(x)h + \sigma(x)\xi h^{1/2}$, $r_{0}(F(x,\rf{\xi})) = \dist(F(x,\rf{\xi}), \partial G)$, $\nu(F(x,\rf{\xi})) = \nu (F^{\pi}(x,\xi))$,
 and $F^{\pi}(x,\rf{\xi})$
 is the projection of $F(x,\rf{\xi})$ on $\partial G$ for all realizations of $\xi$, $x \in \bar{G}$, and $F(x) \in \bar{G}\cup\bar{G}_{-r}$. Under Assumptions~\ref{as:3} and \ref{as:5}, $F(x,\rf{\xi})$, $F^{\pi}(x,\rf{\xi})$ and $\nu(F(x,\rf{\xi}))$ are continuous functions on $\bar{G}$ \rf{for every realization of $\xi$}. Note that for $y = F(x,\xi) \in \bar{G} \cup\bar{G}_{-r}$, $I_{\bar{G}^{c}}(y)$ is discontinuous but $r_{0}(y)I_{\bar{G}^{c}}(y)$ is continuous on $\bar{G} \cup\bar{G}_{-r}$ provided Assumption~\ref{as:3} holds. This implies that $H(x,\rf{\xi}) \in \bar{G}$ is a continuous function for all $x \in \bar{G}$ \rf{and for every realization of $\xi$}. Consider a function $g(x) \in C(\bar{G})$, where $C(\bar{G})$ denotes the class of continuous functions on $\bar{G}$. Now take a sequence $\{x_{n}\}_{n\in \mathbb{N}}$ such that  $x_{n} \rightarrow x$ as $n \rightarrow \infty$. From the above discussion, it is clear that $g(H(x_{n},\rf{\xi})) \rightarrow g(H(x,\rf{\xi})) $ a.s. as $n \rightarrow \infty$. Using the bounded convergence theorem, we obtain $\mathbb{E}g(H(x_{n},\rf{\xi})) \rightarrow \mathbb{E}g(H(x,\rf{\xi}))$ as $n \rightarrow \infty$. This shows that $Pg(x) \in C(\bar{G})$, where $Pg(x)$ is the one-step transition operator defined as $Pg(x) = \mathbb{E}(g(X_{1})|X_{0} = x )$. Hence  $(X_{k})_{k \geq 0}$ is a Feller chain \cite{108}. Since  the state space $\bar{G}$ is compact, it follows from the Krylov-Bogoliubov theorem  \cite{83} that there exists a stationary measure $\mu^{h}$ of $X_k$. We note that $\mu^{h}$ is, as a rule, not unique.

 In this subsection, our focus is on how close $\mu$ and $\mu^{h}$ are. To this end, introduce the 
 metric $D$ between two probability measures $\mu_{1}$ and $\mu_{2}$:
 \begin{equation*}
     D(\mu_{1},\mu_{2}) = \sup\limits_{f \in \mathcal{H}}\bigg| \int_{G}f(x)\mu_{1}(dx) - \int_{G}f(x)\mu_{2}(dx)\bigg|,
 \end{equation*}
 where $ \mathcal{H} = \{ f : \bar{G} \rightarrow \mathbb{R}\; \text{and}\; \mid f\mid^{(2)}_{\bar{G}} \leq 1 \}$. Now, we are in a position to state the theorem.

 \begin{theorem}
  Let Assumptions  \ref{as:3}, \ref{as:4} and \ref{as:5} hold. Suppose $\mu^{h}$ is a stationary measure of the Markov chain $(X_{k})_{k\geq 0}$ described in Algorithm~\ref{algorithm2.1}, then
  \begin{equation*}
      D(\mu,\mu^{h}) \leq Ch,
  \end{equation*}
  where $C$ is a positive constant independent of $h$.
 \end{theorem}
 \begin{proof}
 Denote by $\mathbb{E}_{x}$  the conditional expectation with conditioning on the initial point $X(0) = x$. We have
 \begin{equation*}
     \int_{G} f(x)\mu^{h}(dx) = \int_{G}\mathbb{E}_{x}f(X_{k})\mu^{h}(dx),
 \end{equation*}
 where  $X_{k}$ is from Algorithm~\ref{algorithm2.1} and $f \in \mathcal{H}$. Therefore, we can write
 \begin{equation*}
     \int_{G} f(x)\mu^{h}(dx) = \int_{G}\frac{1}{N}\sum\limits_{k=1}^{N}\mathbb{E}_{x}f(X_{k})\mu^{h}(dx).
 \end{equation*}
 Using Theorem~\ref{thrm3.2}, we get
 \begin{align*}
     \bigg| \int_{G}f(x)\big(\mu^{h}(dx) - 
     \mu(dx)\big)\bigg| & = \bigg|\int_{G}\bigg(\frac{1}{N}\sum\limits_{k=1}^{N}\mathbb{E}_{x}f(X_{k}) - \bar{f}\bigg)\mu^{h}(dx)\bigg| \leq C\bigg(h + \frac{1}{Nh}\bigg),
 \end{align*}
 where $\bar{f}$ is the expectation of $f$ with respect to the invariant measure $\mu$ and (in comparison with Theorem~\ref{thrm3.2}) the constant $C$ does not depend on $f$
 as here we consider only $f \in \mathcal{H}$.  By \rf{letting} $N \rightarrow \infty$ in the above equation, we obtain the stated result.
 \end{proof}

  \section{Solving elliptic PDEs with Robin boundary condition} \label{section5}
  In Subsection~\ref{subsection5.1}, we introduce elliptic PDEs with Robin boundary condition and discuss their link to reflected diffusion processes via the Feynman-Kac formula. We use Algorithm~\ref{algorithm3.1} to numerically solve elliptic PDEs with decay. However, the same algorithm does not work (we highlight the reason for that later) when employed to solve the Poisson PDE with Neumann boundary condition. To deal with that, we propose a new adaptive time-stepping algorithm in Subsection~\ref{subsection5.2}.  We state the two convergence theorems for the two cases (with decay and without decay) in Subsection~\ref{subsection5.3} and prove them in Subsection~\ref{subsection5.4}.

  \subsection{Probabilistic representation}\label{subsection5.1}
  Consider the following elliptic equation
\begin{equation}\label{eqn2.10}
\mathscr{A}u + c(x)u + g(x) = 0, \; \; \;  x \in G,
\end{equation}
with Robin boundary condition
\begin{equation}\label{eqn2.11}
\frac{ \partial u}{\partial \nu} + \gamma(z)u = \psi(z), \; \; \;  z\in \partial G,
\end{equation}
 where $\mathscr{A}$ was introduced in Subsection~\ref{subsection2.4.0}. 
We make the following assumptions in addition to Assumptions \ref{as:3}, \ref{as:4} and \ref{as:5}.

 \begin{assumption}\label{a07}
 $ g(x) \in C^{2}(\bar{G})$ and $ \psi(z) \in C^{3}(\partial G)$.
 \end{assumption}
 \begin{assumption}\label{a8}
 $c(x) \in C^{2}(\bar{G}) $ is negative for all $x \in \bar{G}$ and $\gamma(z) \in C^{3}(\partial G) $  is non-positive for all $z \in \partial G$.
 \end{assumption}
  If Assumptions \ref{as:3}, \ref{as:4}-\ref{as:5} and \ref{a07}-\ref{a8} are satisfied then the problem (\ref{eqn2.10})-(\ref{eqn2.11}) has a unique solution $u(x) \in C^{4}(\bar{G})$ (see \cite{50,104} and \cite[Theorem 3]{107}). 
  Further, the  following estimates hold (see \cite{44,107}):
 \begin{equation*}
     \mid u\mid^{(4)}_{\bar{G}} \leq C\big( \mid g \mid^{(2)}_{G} +  \mid\psi\mid^{(3)}_{\partial G}\big).  \label{eq:36}
 \end{equation*}

 The probabilistic representation of the solution of equations (\ref{eqn2.10})-(\ref{eqn2.11}) is given by \cite{18}:
\begin{equation*}\label{eq:2.14}
u(x) = \lim\limits_{T \rightarrow \infty}\mathbb{E}\Big(Z_{x}(T)\Big),
\end{equation*}
where $Z_{x}(s)$, $x \in \bar{G}$, is governed by the following RSDEs
\begin{align*}
 dY(s) &= c(X(s))Y(s)ds + \gamma(X(s))I_{\partial G}(X(s))Y(s)dL(s),  \, &Y(0) = 1, \\
 dZ(s) &= g(X(s))Y(s)ds - \psi(X(s))I_{\partial G}(X(s))Y(s)dL(s),
 \, &Z(0) = 0,
 \end{align*}
 and $X(s)$ is from (\ref{rsde}).
The diffusion matrix $ \sigma(x)$ in (\ref{rsde}) is related to $a(x)$ as
$ \sigma(x)\sigma(x)^{\top} = a(x).$

\textbf{The case $c(x) = 0$ and $\gamma(z) = 0$}. Consider the Poisson equation (\ref{gpe1}) with Neumann boundary condition (\ref{gpe2}). We discussed the existence and uniqueness of the solution of (\ref{gpe1})-(\ref{gpe2}) in Subsection~\ref{subsection2.4.0}. The probabilistic representation of the solution $u(x)$ has the form \cite{18,45}:
\begin{equation}
    u(x) = \lim\limits_{T \rightarrow \infty}\mathbb{E}Z_x(T) + \bar{u}, \label{eqb6.10}
\end{equation}
where $\bar{u} = \int_{G}u(x)\rho(x)dx$, $\rho(x)$ is the solution of the adjoint problem (\ref{sFP})-(\ref{sFPbc}) (note that $\rho(x)$ is the invariant density of $X(s)$ from (\ref{rsde})), and $Z_x(s)=Z(s)$ is governed by
\begin{align}
  dZ(s) = -\phi_1(X(s))ds - \phi_2(X(s))I_{\partial G}(X(s))dL(s),      \,\, Z(0) = 0,\label{n10new}
 \end{align}
 where $X(s)$ is given by (\ref{rsde}).

 \subsection{Numerical method for the Poisson equation}\label{subsection5.2}
 We will numerically solve elliptic PDE (\ref{eqn2.10}) with Robin boundary condition (\ref{eqn2.11}) using Algorithm~\ref{algorithm3.1} and will prove its first order of convergence when $T\rightarrow \infty$ in Subsection~\ref{subsection5.4}. However, if $c(x) = 0$ and $\gamma(z) = 0$, i.e. if we consider the Poisson problem (\ref{gpe1})-(\ref{gpe2}), and use Algorithm~\ref{algorithm3.1} to numerically solve it then the algorithm's error increases linearly with $N$ for  fixed $h \in (0,1)$, i.e., Algorithm~\ref{algorithm3.1} diverges for this case when $T \to \infty$. Therefore, this case needs a different approach and a new algorithm. This new algorithm (Algorithm~\ref{algorithm6.1}) constructed in this subsection is based on double partitioning of the time interval $[0,T]$ and is convergent in $h$ and $T$, while
  its computational cost grows with decrease of tolerance only slightly higher than linear. This idea is of potential interest in other stochastic numerics settings.

We discretize the interval $[0,T]$ as follows
\begin{equation*}
    T_{j} - T_{j-1} := \Delta_{j} T = N_{j}h_{j},
\end{equation*}
where $N_{j}$ is the number of steps and $h_{j}$ is the time step size in the interval $(T_{j-1}, T_{j}]$. The $h_{j}$ and $N_{j}$ are given by
\begin{equation*}
   h_{j} = \frac{h}{j^{\beta}} \;\;\;\;\; \text{and} \;\;\;\;\;  N_{j} = \floor*{\frac{\Upsilon}{h_{j}j^{\ell}}},
\end{equation*}
where $ 0 <\rf{\ell \leq \beta} \leq 1$, $h>0$ is a fixed sufficiently small number, and $\Upsilon > 0$ is a constant chosen independently of $T$. One can see that $\Delta_{j} T \leq \Upsilon/j^{\ell}$. 
Now we define a constant $\Lambda $ as the smallest natural number so that the following inequality is satisfied:
 \begin{equation}
 T \leq T_{\Lambda} : = \sum\limits_{j=1}^{\Lambda}\Delta_{j}T.
 \label{eq:Lamb}
    \end{equation}
\rf{We note that $\Lambda $ is well defined because of $\Delta_{j} T \geq \Upsilon/j^{\ell}-h_j$ and the condition $ 0 < \ell \leq \beta \leq 1$.}
The total number of steps until time $T_{j}$ is equal to
\begin{equation*}
    N^{'}_{j} = \sum\limits_{i=1}^{j} N_{i},
\end{equation*}
and set $N_{0}^{'} = 0$ and $N = N^{'}_{\Lambda}$. Note that $\Delta_{j}T$ is independent of $h$.

Consider an interval $(T_{j-1},T_{j}]$. In this interval, as in Algorithm~\ref{algorithm3.1}, while moving from $X_{ k}$ to $X_{k+1}^{'}$ in (\ref{eq:11}), we take the following step
\begin{equation}\label{b6.11}
    X_{k+1}^{'} = X_{k} + h_{j}b_{k} + h^{1/2}_{j}\sigma_{k}\xi_{k+1}, \;\;\;\;\;\;\; k = N^{'}_{j-1},\dots,N^{'}_{j}-1.
\end{equation}
\rs{Let us denote $ \phi_{1,k} = \phi_{1}(X_{k})$ and $ \phi_{2,k}^{\pi} = \phi_{2}(X_{k}^{\pi})$.} If $X_{k+1}^{'} \in \bar{G}$ then
\begin{equation}
X_{k+1} = X_{k+1}^{'}, \;\;\;\;\;\;\;\; Z_{k+1} = Z_{k}  - h_{j}\phi_{1,k}, \label{a6.12}
\end{equation}
else we take the reflection step as in Algorithm~\ref{algorithm3.1}:
\begin{equation}
    X_{k+1} = X_{k+1}^{'} + 2r_{j,k+1}\nu,\;\;\;\;\;\;\; k = N^{'}_{j-1},\dots,N^{'}_{j}-1, \label{a6.13}
\end{equation}
where $r_{j,k+1} $ denotes $\dist(X_{k+1}^{'},X^{\pi}_{k+1})$ in time interval $(T_{j-1},T_{j}]$, $X_{k+1}^{\pi}$ is projection of $X_{k+1}^{'}$ on $\partial G$ and $\nu $ is the inward unit normal at $X_{k+1}^{\pi} \in \partial G$.
   Analogous to (\ref{eq:17}), we write
\begin{equation}
    Z_{k+1} = Z_{k} - h_{j}\phi_{1,k} -2r_{j,k+1}\phi_{2,k+1}^{\pi},\;\;\;\;\;\;\; k = N^{'}_{j-1},\dots,N^{'}_{j}-1. \label{a6.14}
\end{equation}
\vspace{-5ex}
\begin{algorithm}
\caption{Algorithm to solve (\ref{gpe1})-(\ref{gpe2})}\label{algorithm6.1}
\begin{steps}
\begin{em}
\item \begin{textit} Set $X_{0} = x$, $Z_{0} = 0$, $T_{0} = 0$, $j=0$, $k=0$, $N^{'}_{0} = 0$.\end{textit}
\item  \textbf{If} $T_{j} < T$,
 set $j := j+1$, $h_{j} := \frac{h}{j^{\beta}}$, $N_{j}^{'} := N_{j-1}^{'} + \floor*{\frac{\Upsilon}{h_{j}j^{\ell}}} $, $T_{j} := T_{j-1} + h_{j}\floor*{\frac{\Upsilon}{h_{j}j^{\ell}}}$,  and \textbf{Goto} Step~3, \textbf{else} \textbf{Stop}.
\item \begin{textit}Simulate $\xi_{k+1}$ and find $X_{k+1}^{'}$ using (\ref{b6.11}).\end{textit}
\item \textbf{If} ${ X_{k+1}^{'} \in \bar{G}} $ then find $X_{k+1}$ and $Z_{k+1}$ according to (\ref{a6.12}),\\
\textbf{else} find $X_{k+1}$ and $Z_{k+1}$ according to (\ref{a6.13})-(\ref{a6.14}).
\item Put $k := k+1$ and \textbf{if} $k < N^{'}_{j}$ then \textbf{return} to Step 3 \textbf{else} \textbf{return} to Step 2.
\end{em}
\end{steps}

\end{algorithm}
\vspace{-5ex}
\subsection{Two convergence  theorems}\label{subsection5.3}
Theorem~\ref{thrm5.4} shows the estimate of error incurred while solving the Robin problem for elliptic PDE using Algorithm~\ref{algorithm3.1}. Theorem~\ref{theorem6.7} gives the error estimate when Algorithm~\ref{algorithm6.1} is employed to solve the Poisson PDE with Neumann boundary condition.
\begin{theorem}\label{thrm5.4}
  Under Assumptions \ref{as:3}, \ref{as:4}-\ref{as:5} and \ref{a07}-\ref{a8}, the following inequality holds for sufficiently small $h>0$:
 \begin{equation}
    \mid \mathbb{E}( Z_{N}) - u(x) \mid \;  \leq \; C\big(h + e^{-\lambda T}\big), \label{eq:th5.4}
\end{equation}
where $Z_{N}$ is calculated according to Algorithm~\ref{algorithm3.1} approximating the solution  $u(x)$ of $(\ref{eqn2.10})$-$(\ref{eqn2.11})$, and  $C$ and $\lambda $ are positive constants independent of $T$ and $h$.
\end{theorem}

\begin{theorem}\label{theorem6.7}
     Let \rs{$0 < \ell \leq  \beta \leq 1$} and $\ell/2 + \beta > 1$ and Assumptions~\ref{as:3},  \ref{as:4}, \ref{as:5}, \ref{assu6.1} hold along with the centering condition (\ref{cc}).
     Assume that $v(t,x_{0})=\mathbb{E}u(X_{t,x_{0}}(T))$ satisfies Assumption~\ref{as10}, where $u(x)$ is the solution of (\ref{gpe1})-(\ref{gpe2}) and $X_{t,x_{0}}(s)$ solves (\ref{rsde}). Then
      \begin{equation}
          \big| \mathbb{E}(Z_{N}) - \lim\limits_{T\rightarrow \infty}Z_{\rf{x_0}}(T) \big| \leq  C \big( h +  e^{-\lambda T} \big),  \label{eqnwa5.12}
      \end{equation}
    where $Z_{N}$ is calculated according to Algorithm~\ref{algorithm6.1},
    $Z(s) $ solves (\ref{n10new}), $C$ and $\lambda$ are positive constants independent of $T$ and $h$, and $N = N^{'}_{\Lambda}$.
\end{theorem}

The proofs follow the same procedure as in Subsection~\ref{subsection3.4}, i.e. we first obtain error estimates for local approximation and prove a lemma related to the number of steps the Markov chain $ X_{k}^{'}$ spends in $\bar{G}_{-r}$, and then based on them we prove the main convergence theorems. Computational complexity of Algorithm~\ref{algorithm6.1} together with an optimal choice of $\ell$ and $\beta$   is discussed at the end of Subsection~\ref{section6.3}.

We remark that, in Theorem~\ref{theorem6.7}, Assumptions~\ref{as:3},  \ref{as:4}, \ref{as:5}, \ref{assu6.1} and the compatibility condition~(\ref{cc}) guarantee (see Subsection~\ref{subsection2.4.0}) that the solution $u(x)$ of the considered  elliptic problem (\ref{gpe1})-(\ref{gpe2}) belongs to $C^4(\bar G)$. Hence we can consider the parabolic problem (\ref{eq:47})-(\ref{eq:49}) with the coefficients as in (\ref{gpe1})-(\ref{gpe2}) and with the terminal condition $u(x)$ instead of  $\varphi(x)$. Consequently, the solution $v(t,x)=\mathbb{E}u(X_{t,x}(T))$ of this parabolic problem can be assumed to satisfy Assumption~\ref{a7}, which is natural (see the discussion after Assumption~\ref{a7}). However, to prove Theorem~\ref{theorem6.7}, it is sufficient to assume a weaker assumption on $v(t,x)$, Assumption~\ref{as10}.

\begin{remark}
\rf{If we take $\phi_{1}(x) = -f(x) + \bar{f}$ and $\phi_{2}(z) = 0$ in (\ref{gpe1})-(\ref{gpe2}), where $\bar{f} = \int_{\bar{G}}f(x) \rho(x)dx$ and $\rho(x)$ is the invariant density of $(\ref{rsde})$, then, for any $f(x) \in C^{2}(\bar{G})$, the solution of (\ref{gpe1})-(\ref{gpe2}) is given by $u(x) = \lim_{T\rightarrow \infty}\mathbb{E}( Z(T)) = \int_{0}^{\infty}\mathbb{E}\big(f(X(s)) -\bar{f}\big)ds$ under Assumptions~\ref{as:3}, \ref{as:4} and \ref{as:5}. The solution of this stationary problem is of potential interest for certain applications  (see e.g. \cite{25}). }
\end{remark}

\subsection{Proofs}\label{subsection5.4}

We first (Subsection~\ref{sec:negc}) consider the case of (\ref{eqn2.10})-(\ref{eqn2.11}) with $c(x)<0,\; x \in G$, and $\gamma(z) \leq 0,\; z\in \partial G$. 
The case of (\ref{gpe1})-(\ref{gpe2}) with $c(x)=0,\;x\in G$ and $\gamma(z)=0,\;z\in \partial G$ is considered in Subsection~\ref{section6.3}.

We need the following notation, $u_{k+1} = u(X_{k+1})$, $u_{k}= u(X_{k})$, $u_{k+1}^{\pi}=u(X_{k+1}^{\pi})$, $u_{k}^{\prime} = u(X_{k}^{'})$, $\varphi_{k} = \varphi(X_{k})$,  $\psi_{k+1}^{\pi} = \psi(X_{k+1}^{\pi})$, $\gamma_{k+1}^{\pi}= \gamma(X_{k+1}^{\pi})$, $g_{k}= g(X_{k})$, $c_{k}= c(X_{k})$, $b_{k}= b(X_{k})$, $a_{k} =a( X_{k})$, $\sigma_{k} = \sigma(X_{k})$, where $X_{k}$, $Y_{k}$, $Z_{k}$, $X_{k}^{'}$, $Y_{k}^{'}$, $Z_{k}^{'}$, $X_{k}^{\pi}$ are appropriately as in Algorithm~\ref{algorithm3.1} in Subsection \ref{sec:negc} and Algorithm~\ref{algorithm6.1} in Subsection~\ref{section6.3}. Here $u(x)$ denotes the solution of (\ref{eqn2.10})-(\ref{eqn2.11}) in Subsection~\ref{sec:negc} and of (\ref{gpe1})-(\ref{gpe2}) in Subsection~\ref{section6.3}.

\subsubsection{The case $c(x)<0$ and $\gamma(z) \leq 0$ \label{sec:negc}} We first give error bounds on one-step approximations. Under Assumptions \ref{as:3}, \ref{as:4}-\ref{as:5} and \ref{a07}-\ref{a8}, the following hold for all $k = 0,1\dots,N-1$:
\begin{align}
     \Big|\mathbb{E}\Big(u_{k+1}^{\prime}Y_{k+1}^{'}+ Z_{k+1}^{'}-u_{k}Y_{k}  - Z_{k} \big|X_{k},Y_{k},Z_{k} \Big)\Big| &\leq C h^{2}Y_{k}, \label{l5.2} \\
     \rf{\big|u_{k+1}Y_{k+1} + Z_{k+1} - u_{k+1}^{'}Y_{k+1}^{'}  -Z_{k+1}^{'}  \big|} &\leq  Chr_{k+1}Y_{k}I_{\bar{G}^{c}}(X_{k+1}^{'})\;\;\; a.s., \label{l5.1}
\end{align}
where $C$ is positive constant independent of $h$ and $T$. With the change of notation as introduced in the beginning  of the current subsection, we can obtain (\ref{l5.2}) and (\ref{l5.1}) by following exactly the same procedure as in the proof of Lemmas~\ref{lemma2.2} and \ref{lemma2.3}, respectively. The independence of $C$ from $T$ trivially follows from the fact that we are dealing with the elliptic equation.

The next lemma gives a bound related to the number of steps which the chain $X_{k}^{'}$ spends in $\bar{G}^{c}$.
\begin{lemma}\label{sjlme}
  Under Assumptions \ref{as:3}, \ref{as:4} and \ref{as:5}, the following bound holds for $\lambda > 0$ and sufficiently small $h>0$:
  \begin{equation*}
      \mathbb{E}\bigg(\sum\limits_{k=0}^{N-1} r_{k+1}(1-\lambda h)^{k}I_{\bar{G}^{c}}\big(X_{k+1}^{'}\big)\bigg)\leq C,
  \end{equation*}
  where $C$ is a positive constant independent of $T$ and $h$.
\end{lemma}
\begin{proof}
For $\lambda > 0$, consider the elliptic problem
\begin{align*}
    &\mathscr{A}v(x) - \lambda v(x)= 0, \,\,\,\,  x  \in G, \\
   & (\nabla v(z)\cdot\nu(z))  = -1, \,\,\,\,  z  \in \partial G.
\end{align*}
In this case it follows from (\ref{mdy1}) and (\ref{eq:17}):
\begin{align}
    Y_{k} &= (1-\lambda h)^{k} \leq e^{-\lambda t_{k}}, \label{b5.1} \\
    Z_{k} &= 2\sum\limits_{i=0}^{k-1}r_{i+1}(1-\lambda h)^{i}I_{\bar{G}^{c}}(X^{'}_{i+1}). \label{b5.2}
\end{align}
 Using (\ref{l5.2}), (\ref{l5.1}) and (\ref{b5.1}), we get
 \begin{align*}
    \Big|\mathbb{E}\Big(v_{k+1}^{\prime}Y_{k+1}^{'}-v_{k}Y_{k} + Z_{k+1}^{'} - Z_{k} \big|X_{k},Z_{k} \Big)\Big| &\leq C h^{2}e^{-\lambda t_{k}}, \\     \rf{\big| v_{k+1}Y_{k+1} - v_{k+1}^{\prime}Y_{k+1}^{'}  + Z_{k+1}-Z_{k+1}^{'}  \big|} &\leq  C h  r_{k+1}e^{-\lambda t_{k}}I_{\bar{G}^{c}}(X_{k+1}^{'})\;\;\;a.s.,
    \end{align*}
Similarly to the proof of Theorem \ref{conp}, we obtain
    \begin{eqnarray}
 \big| \mathbb{E}\big(v(X_{N})Y_{N} + Z_{N}\big)  - v( X_{0})Y_{0} - Z_{0}\big| & \leq  Ch \mathbb{E}\Bigg( \sum\limits_{k=0}^{N-1}r_{k+1}e^{-\lambda t_{k}}I_{G_{-r}}\big( X_{k+1}^{'}\big) \Bigg) \nonumber  \\ & +  C h^{2} \sum\limits_{k=0}^{N-1}e^{-\lambda t_{k}}. \label{eq6.3}
   \end{eqnarray}
Using (\ref{b5.1})-(\ref{eq6.3}) and the facts that $r_{k+1} = \mathcal{O}(h^{1/2})$ and $v(x)$ is uniformly bounded for $x \in \bar{G}$, we get (note that $Z_{N} \geq 0$ for $\lambda h <1$):
\begin{align*}
      \mathbb{E}\bigg(&\sum\limits_{k=0}^{N-1} r_{k+1}(1-\lambda h)^{k}I_{\bar{G}^{c}}\big(X_{k+1}^{'}\big)\bigg)  = \rs{\frac{1}{2}\mathbb{E}(Z_{N})}  \\ &   \leq  \big|v(x)\big| + \big| \mathbb{E}v(X_{N})Y_{N} \big| + Ch \mathbb{E}\Bigg( \sum\limits_{k=0}^{N-1}r_{k+1}e^{-\lambda t_{k}}I_{G_{-r}}\big( X_{k+1}^{'}\big) \Bigg) + C h^{2} \sum\limits_{k=0}^{N-1}e^{-\lambda t_{k}} \\ & \leq C + C(h^{3/2} + h^{2}) \sum\limits_{k=0}^{N-1}e^{-\lambda t_{k}},
\end{align*}
which completes the proof.
\end{proof}

\begin{proof}[Proof of Theorem \ref{thrm5.4}]
Using the notation introduced earlier in this section, we begin \rs{the} analysis with an inequality analogous to (\ref{ep4.16}) from the proof of Theorem \ref{conp} \rs{combining which with} (\ref{l5.2}) and (\ref{l5.1}), we get
\begin{align}
 \big| \mathbb{E}\big(u(X_{N})Y_{N} + Z_{N}\big)  - u(X_{0})Y_{0} - Z_{0}\big|  &\leq Ch\Bigg| \mathbb{E}\Bigg( \sum\limits_{k=0}^{N-1}r_{k+1}Y_{k}I_{G_{-r}}\big( X_{k+1}^{'}\big) \Bigg)\Bigg| \nonumber \\& +  C h^{2}\Bigg|\mathbb{E}\Bigg( \sum\limits_{k=0}^{N-1}Y_{k}\Bigg)\Bigg|. \label{ne5.4}
  \end{align}
  From (\ref{mdy1}) and (\ref{eq:16}), we have
\begin{align*}
    Y_{k} &= Y_{k-1} + hc_{k-1}Y_{k-1} + 2r_{k}\gamma_{k}^{\pi}Y_{k-1}I_{G_{-r}}(X_{k}^{'}) + 2r_{k}^{2}\big(\gamma_{k}^{\pi}\big)^{2}Y_{k-1}I_{G_{-r}}(X_{k}^{'}).
\end{align*}
Under the assumption $ \gamma(z) \leq 0$ and sufficiently small $h>0$, we have $2r_{k}\gamma_{k}^{\pi}Y_{k-1}I_{G_{-r}}(X_{k}^{'})(1+r_{k}\gamma_{k}) \leq 0$. Hence
\begin{equation}
    Y_{k}  \leq Y_{k-1}\big(1 + c_{k-1}h\big)  \leq (1-\lambda h)^{k} \leq    e^{-\lambda t_{k}},
    \label{n5.3}
\end{equation}
where $\lambda = \min\limits_{x \in \bar{G}}|c(x)|>0$. Then substituting (\ref{n5.3}) in (\ref{ne5.4}), we obtain
\begin{align}
     \big|\mathbb{E}\big( Z_{N}\big) - u(x)\big|   &\leq  C\big| \mathbb{E}u(X_{N})\big|e^{-\lambda T} +  Ch\mathbb{E}\Bigg( \sum\limits_{k=0}^{N-1}r_{k+1}(1-\lambda h)^{k}I_{G_{-r}}\big( X_{k+1}^{'}\big)\Bigg) \nonumber \\ & \;\;\; +  C h^{2} \sum\limits_{k=0}^{N-1}e^{-\lambda t_{k}}. \label{s5.21}
\end{align}
 Applying Lemma \ref{sjlme} to (\ref{s5.21}), we get (\ref{eq:th5.4}).
\end{proof}

\subsubsection{The case $c(x) = 0$ and $\gamma (z) = 0$}\label{section6.3}
Consider the problem (\ref{gpe1})-(\ref{gpe2}) as described in Subsection~\ref{subsection2.4.0} together with the probabilistic representation (\ref{eqb6.10}), (\ref{rsde}), (\ref{n10new}).

Now we are in a position to state two results regarding the one-step approximation in the time interval $(T_{j-1},T_{j}]$. Under Assumptions~\ref{as:3}, \ref{as:4}, \ref{as:5}, \ref{assu6.1}, these results directly follow from (\ref{l5.2}) and (\ref{l5.1}) for all $ k = N^{'}_{j-1},\dots,N^{'}_{j}-1$:
\begin{align}
    |\mathbb{E}(u_{k+1}^{\prime} + Z_{k+1}^{'} - u_{k} - Z_{k}| X_{k}, Z_{k})| &\leq Ch_{j}^{2}, \label{a6.10}\\
    \rf{| u_{k+1} + Z_{k+1} - u_{k+1}^{\prime} - Z_{k+1}^{'}|} & \leq Ch_{j}r_{j,k+1}I_{\bar{G}^{c}}(X_{k+1}^{'})\;\;\;a.s., \label{a6.11}
\end{align}
 where $C$ is a positive constant independent of $T$ and $h$.

Recall that the layer $\bar{G}_{-r}$ was introduced in Subsection~\ref{subsection3.4}. We now state a lemma which is related to the number of steps of $X_{k}^{'}$ in $G_{-r}$. For that purpose, we first consider a surface $ \mathbb S_{l_{j}}$ which belongs to $G$ and is parallel to the boundary $\partial G$. The distance between $\partial G$ and $\mathbb S_{l_{j}}$ is
\begin{equation}
l_{j}=\min\{(\Delta_{j} T)^{1/2}, R/2, 1\}, \label{leq}
\end{equation}
where \rs{$R$ is the radius of the uniform interior sphere} inside $G$. We take $R/2$ in (\ref{leq}) so that a surface $\mathbb{S}_{l_{j}}$, which is parallel to $\partial G$, exists 
when $R$ is not sufficiently large.     We assume that $h>0$ is sufficiently small so that $l_{j} >> h_{j}^{1/2}$. By virtue of Assumption \ref{as:3},  $l_{j}$-neighbourhood of any point $x\in \mathbb S_{l_{j}}$ entirely belongs to $G$. The layer between $\mathbb S_{l_{j}}$ and $\partial G$ will be denoted as $G_{l_{j}}$.

\begin{lemma} \label{blu}
Under Assumptions \ref{as:3} and \ref{as:5}, the following inequality holds
\begin{equation}
    \mathbb{E}\Bigg(\sum\limits_{k=N^{'}_{j-1}+1}^{N_{j}^{'}} r_{j,k}I_{G_{-r}}(X_{k}^{'})\; \Big| \; X_{N_{j-1}^{'}}\Bigg)\leq C\frac{\Delta_{j} T}{l_{j}},\;\;\;\;\;\;\;\; j = 1,\dots,\Lambda, \label{eq:lem65}
\end{equation}
where $C$ is a positive constant independent of $\Delta_{j} T$ and $h$.
\end{lemma}
\begin{proof}
Let us fix $j\rs{\in}\{1,\dots,\Lambda\}$ and hence consider $t \in [T_{j-1},T_{j}]$. 
Define a function $U(t,x) $ as
\begin{equation*}
    U(t,x) =
    \begin{cases}
    0, & (t,x) \in \{T_{j}\}\times \big(\bar{G} \cup \bar{G}_{-r}\big),\\
    \frac{1}{l_{j}}\big(K(T_{j} - t) + w(x)\big), & (t,x) \in [T_{j-1}, T_{j}-h_{j}] \times \big(\bar{G} \cup \bar{G}_{-r}\big),
    \end{cases}
\end{equation*}
\rf{where $w(x)$ is from (\ref{wx}) but with distance from the surface $\mathbb{S}_{l_{j}}$}. Introduce the region $ S_{h_{j}} = \{ x \mid \dist(x,\mathbb S_{l_{j}}) < K_{1}h_{j}^{1/2}\} $. We choose $K_{1}$ so that for any point $x \in G_{l_{j}}\backslash S_{h_{j}}$, there is zero probability that in  one step transition any of the $2^{d}$ realizations of $X_{j,1}^{'}$ cross the surface $\mathbb S_{l_{j}}$, where $ X^{'}_{j,1}$ is constructed according to Algorithm~\ref{algorithm6.1} given $X^{'}_{j,0} = x$. Note that the region $S_{h_{j}}$ contains points $x$ starting from which one-step realizations of $X_{j,1}^{'}$ may or may not cross the surface $\mathbb S_{l_{j}}$. As we did in Lemma~\ref{bl}, we calculate $PU(t,x)-U(t,x)$ at points $(t,x)$ lying in different regions identified by four cases:  (I) $x \in G\backslash(G_{l_{j}}\cup S_{h_{j}})$, when $X_{j,1}^{'}$ remains in $G\backslash G_{l_{j}}$; (II) $x \in  S_{h_{j}}$; (III) $x \in G_{l_{j}}\backslash S_{h_{j}}$; (IV) $x \in \bar{G}_{-r}$. We analyze $PU(t,x) - U(t,x)$, calculated according to \rs{the} above four cases, analogously as in the proof of Lemma~\ref{bl} to obtain the desired bound.
\end{proof}
The next lemma gives an error estimate for Algorithm~\ref{algorithm6.1} applied to (\ref{gpe1})-(\ref{gpe2}) accumulated over an interval $(T_{j-1},T_{j}]$.

\begin{lemma}\label{lemma6.6}
Under Assumptions \ref{as:3} and \ref{as:4}-\ref{assu6.1}, the following inequality holds for all $j = 1,\dots,\Lambda$:
\begin{equation}
    \big|\mathbb{E}(u(X_{N^{'}_{j}}) + Z_{N^{'}_{j}} - u(X_{N^{'}_{j-1}}) - Z_{N^{'}_{j-1}}| X_{N^{'}_{j-1}}, Z_{N^{'}_{j-1}})\big| \leq C\frac{h_{j}\Delta_{j} T}{l_{j}}, \label{aap1}
\end{equation}
where $C>0$ is independent of $h_{j}$ and $\Delta_{j}T$.
\end{lemma}
\begin{proof}
We can write
\begin{align*}
     &\big| \mathbb{E}\big(u(X_{N_{j}^{'}}) + Z_{N_{j}^{'}}  - u(X_{N_{j-1}^{'}}) - Z_{N_{j-1}^{'}} | X_{N_{j-1}^{'}}, Z_{N_{j-1}^{'}}\big) \big| \\ &\leq  \Bigg| \mathbb{E}\Bigg( \sum\limits_{k=N^{'}_{j-1}}^{N^{'}_{j}-1}\big(u_{k+1} - u_{k+1}^{\prime}+ Z_{k+1}-Z_{k+1}^{'}\big)I_{G_{-r}}\big(X_{k+1}^{'}\big)\bigg| X_{N^{'}_{j-1}},Z_{N^{'}_{j-1}} \Bigg)\Bigg| \\
     & \;\;\;\;\;\; + \Bigg| \mathbb{E}\Bigg(  \sum \limits_{k=N^{'}_{j-1}}^{N^{'}_{j}-1}\mathbb{E}\Big(u_{k+1}^{\prime} + Z_{k+1}^{'} - u_{k} - Z_{k} \Big| X_{k}, Z_{k}\Big) \bigg| X_{N^{'}_{j-1}}, Z_{N^{'}_{j-1}} \Bigg)\Bigg|.
\end{align*}
Using  the inequalities (\ref{a6.10}), (\ref{a6.11}), Lemma \ref{blu}, and (\ref{leq}), we ascertain (\ref{aap1}).
\end{proof}

The probabilistic representation (\ref{eqb6.10}) of the solution of (\ref{gpe1})-(\ref{gpe2}) contains two parts, $\lim\limits_{T\rightarrow \infty}Z(T) $ and $\bar{u}$. In Theorem~\ref{theorem6.7}, we estimate the error incurred in approximating $\lim\limits_{T\rightarrow \infty}Z(T) $, which proof is presented below. 

\begin{proof}[Proof of Theorem \ref{theorem6.7}]
 We have
\begin{equation*}
 \mathbb{E}\big(u(X_{N}) + Z_{N} - u(x)\big) =
   \mathbb{E}\bigg(\sum\limits_{j=1}^{\Lambda}\mathbb{E}\big(u(X_{N^{'}_{j}}) + Z_{N^{'}_{j}} - u(X_{N^{'}_{j-1}}) - Z_{N^{'}_{j-1}}\big| X_{N^{'}_{j-1}}, Z_{N^{'}_{j-1}}\big)\bigg).
  \end{equation*}
Using Lemma \ref{lemma6.6}, we get
 \begin{equation*}
     \big| \mathbb{E}\big(u(X_{N}) + Z_{N}\big) - u(x) \big|  \leq C \sum\limits_{j=1}^{\Lambda}\frac{h_{j}\Delta_{j} T}{l_{j}}.
  \end{equation*}
\rf{From (\ref{leq}), we know that 
either $l_{j} = (\Delta_{j}T)^{1/2}  $ or $l_{j} = (R/2) \wedge 1$.} Therefore, we have
\begin{equation}
    \big| \mathbb{E}\big(u(X_{N}) + Z_{N}\big)  - u(x) \big|  
    \leq C\sum\limits_{j=1}^{\Lambda}h_{j}\big(\Delta_{j} T + (\Delta_{j} T)^{1/2}\big).  \label{6.21q}
\end{equation}
Now, consider the parabolic problem (\ref{eq:47})-(\ref{eq:49}) whose solution we denote as $v(t,x)$ instead of $u(t,x)$ and whose terminal condition at $t = T_{\Lambda}$ is $u(x)$ instead of $\varphi(x)$. In the current setting we have assumed that the solution $v(t,x)$ satisfies the following bound
\begin{equation}
    \sum\limits_{p=0}^{4}\sum\limits_{2i+|l|=p}\sup\limits_{[0,T_{\Lambda})\times \bar{G}}|D_{t}^{i}D_{x}^{l}v(t,x)|\leq C,\label{cb}
\end{equation}
where $C>0$ is independent of $T_{\Lambda}$. Then, analogously to the previous lemma, we can prove
\begin{equation}
    |\mathbb{E}(v(t_{j},X_{N_{j}^{'}}) -v(t_{j-1},X_{N_{j-1}^{'}})|X_{N_{j-1}^{'}})| \leq C\frac{h_{j}\Delta_{j} T}{l_{j}},\;\;\; j \label{bov} =1,\dots,\Lambda,
\end{equation}
where $C>0$ is independent of $T_{\Lambda}$ (hence also of $T$) due to (\ref{cb}). Note that to obtain (\ref{bov}) in the interval $(T_{\Lambda-1},T_{\Lambda}]$, we use the facts that $\mathbb{E}u(X_{N}) = \mathbb{E}u(X_{N-1}) + \mathcal{O}(h_{\Lambda})$ and $|v(t_{N},X_{N-1}) - v(t_{N-1},X_{N-1})| \leq Ch_{\Lambda}$ (cf. the proof of Lemma~\ref{theorem3.9}). Therefore, we have (cf. (\ref{6.21q}))
\begin{equation*}
    |\mathbb{E}u(X_{N}) - \mathbb{E}u(X(T_{\Lambda}))|= |\mathbb{E}u(X_{N}) - v(0,x)|  \leq C\sum\limits_{j=1}^{\Lambda}h_{j}(\Delta_{j}T + (\Delta_{j}T)^{1/2}),
\end{equation*}
which together with (\ref{eq:34}) (though with the appropriate change of notation) gives
\begin{equation}
    \big| \mathbb{E}u(X_{N}) - \bar{u}  \big| \leq C\Big(\sum\limits_{j=1}^{\Lambda}h_{j}(\Delta_{j}T + (\Delta_{j}T)^{1/2}) + e^{-\lambda T}\Big), \label{6.23r}
\end{equation}
where $\bar{u}$ is from (\ref{eqb6.10}), and $C$ and $\lambda$ are positive constants independent of $T$.

Using (\ref{eqb6.10}), (\ref{6.21q}), (\ref{6.23r}) and the fact that $ 1 < \ell/2 + \beta \leq 3/2$, we obtain
\begin{align}
 \big| \mathbb{E}\big( Z_{N}\big)  &- \lim\limits_{T \rightarrow \infty} Z(T) \big| 
 = \big|\mathbb{E}\big( (u(X_{N}) + Z_{N} - u(x)) - (u(X_{N}) -\bar{u})\big)\big| \nonumber \\ & \leq C \Big( \sum\limits_{j=1}^{\Lambda}h_{j}\big(\Delta_{j} T + (\Delta_{j} T)^{1/2}\big) + e^{-\lambda T}\Big)  \leq C \bigg( h\sum\limits_{j=1}^{\Lambda} \Big(\frac{1}{j^{\ell/2 + \beta}} + \frac{1}{j^{\ell + \beta}}\Big) + e^{-\lambda T}
\bigg)\nonumber \\ &\;\;\;\;\;\; \leq C\bigg(\frac{h}{\ell/2 + \beta - 1} + h+ e^{-\lambda T}\bigg) .\label{t6.27}
\end{align}
  \end{proof}

We now discuss the cost of Algorithm~\ref{algorithm6.1}. Recall that the cost of Algorithms~\ref{algorithm2.1} and \ref{algorithm3.1}, in  all applications considered in this paper, is proportional to $1/h$. First, let us define (cf. (\ref{eqnwa5.12})) the tolerance of Algorithm~\ref{algorithm6.1} as $\text{tol} := \frac{1}{2}(h + e^{-\lambda T})$. Taking $h = e^{-\lambda T}$ (i.e., equal contributions from the two sources of the total error), we get $\text{tol} = h$ and $T = -\frac{\ln{h}}{\lambda}$. Further, we choose $\Lambda$ to guarantee $T \leq T_{\Lambda} \rs{\leq} \sum\limits_{j=1}^{\Lambda}\frac{\Upsilon}{j^{\ell}}$.
Let $\ell \neq 1$.  Then
 $ T \sim \Upsilon \Lambda^{-\ell + 1}$. Therefore, we need $\Lambda \sim |\ln{h}|^{1/(1-\ell)}  $ (we dropped $\lambda $ and $\Upsilon$ because here our only concern is to know how the cost grows with decrease of the tolerance tol via decrease of $h$).
  The cost of Algorithm~\ref{algorithm6.1} is appropriate to measure via the number of steps of a single trajectory of the algorithm which is equal to
\begin{equation*}
\text{Cost} = \sum\limits_{j=1}^{\Lambda} N_{j}  = \sum\limits_{j=1}^{\Lambda}\frac{\Delta_{j} T}{h_{j}} \sim \frac{1}{h}\sum\limits_{j=1}^{\Lambda} \frac{1}{j^{\ell-\beta}} \sim \frac{1}{h}\Lambda^{\beta - \ell + 1} \sim \frac{1}{h}|\ln{h}|^{1 + \frac{\beta}{1-\ell}}
=\frac{1}{\text{tol}}|\ln{\text{tol}}|^{1 + \frac{\beta}{1-\ell}}.
\end{equation*}
We need to choose $\ell$ and $\beta$ so that factor $\frac{\beta}{1-\ell}$ is small in the cost while the constant $\frac{1}{\ell/2 + \beta - 1}$ in the error (\ref{t6.27}) is not too big under the constraints $\ell/2 + \beta - 1 > 0$ \rs{and} $0< \ell \leq  \beta \leq 1$. In this case it is optimal to take $\beta = 1$ and small $\ell$. For example, $\ell = 1/10$ gives $\frac{\beta}{1-\ell} = 10/9$, $\frac{1}{\ell/2 + \beta - 1} = 20$ and $\text{Cost} \sim \frac{|\ln{h}|^{2.1}}{h}$.
Hence, we infer that with appropriate choice of $\ell$ and $\beta$, the cost of Algorithm~\ref{algorithm6.1} is only slightly higher than linear in $1/h$.

 We note that instead of using the two level partitioning of the time interval $[0,T]$ we could use the following single partition of $[0,T]$: $T =  \sum_{k=1}^{\Lambda}\frac{h}{k^{\beta}}$ with $ 0 <\beta < 1 $. For this single partition, by similar arguments as in the proof of  Theorem~\ref{theorem6.7}, we would only prove
\begin{equation*}
    \big| \mathbb{E}\big( Z_{N}\big)  - \lim\limits_{T \rightarrow \infty} Z(T) \big| \leq C \bigg( h^{3/2}\sum\limits_{j=1}^{\Lambda} \frac{1}{j^{3\beta/2}} + e^{-\lambda T}\bigg) \leq C\Big( \frac{2h^{3/2}}{3\beta - 2} + e^{-\lambda T}\Big), \, 2/3 < \beta <1.
\end{equation*}
Here again it is natural to define the tolerance tol $:= \frac{1}{2}(h^{3/2} + e^{-\lambda T})$ and again to take $h^{3/2} = e^{-\lambda T}$. Then the cost in terms of tol is
\begin{equation*}
\text{Cost}:=\Lambda \sim \Big(\frac{T}{h}\Big)^{1/(1-\beta)}
    \sim \frac{|\ln{\text{tol}}|^{1/(1-\beta)}}{\text{tol}^{2/(3(1-\beta))}}, \, \, 2/3 < \beta < 1,
\end{equation*}
which is significantly worse than linear increase in {1}/tol (e.g., for $\beta=2/3$, we have Cost $\sim |\ln{\text{tol}}|^{3}/$tol$^2$).
This emphasizes the importance of the idea of double partitioning of the time interval in Algorithm~\ref{algorithm6.1}.

 \section{Extensions}\label{section7}

In Section \ref{section7.1}, we present an algorithm to approximate RSDEs (\ref{eq:8}) with second order of weak convergence based on adaptive time stepping near the boundary. In Section~\ref{section7.2}, we generalise Algorithm~\ref{algorithm2.1} to approximate RSDEs in weak sense when reflection is in an inward oblique direction.

\subsection{Second-order approximation}\label{section7.1}

In this subsection we modify Algorithm \ref{algorithm2.1} to construct a second-order weak
approximation of the solution $X(t)$ of the RSDEs (\ref{eq:1}) on an
interval $[T_{0},T]$ (for simplicity, we do not include here a second-order
approximation of (\ref{eq:9})-(\ref{eq:10})). To obtain a second-order method for (\ref{eq:1}),
we need an approximation of weak local  order $3$ for SDEs without
reflection, which we will use inside the domain $G$, and we need a more accurate approximation of the reflection than in Algorithm~\ref{algorithm2.1}.
To achieve the \rf{latter}, we will introduce an \rf{adaptive (random)} time step so that
the auxiliary chain $X_{k}^{' }$ belongs to $\bar{G}\cup \bar{G}_{-r}$ with \rf{$r= \Theta(h)$} , where the notation $\bar{G}_{-r}$ was
introduced in Subsection~\ref{section2.4} \rf{and the Landau Big Theta notation, $\Theta(h)$, implies that there are constants $C_{1},\;C_{2}>0$ independent of $h$ such that $C_{1}h\leq r \leq C_{2}h$}. When $X_{k}^{' }\in $ $\bar{G}_{-r},$ we will use the same reflection as in Algorithm~\ref{algorithm2.1}, but because the
layer $\bar{G}_{-r}$ is of size \rf{$\Theta(h)$} in this subsection, the one-step
error of reflection will be $\mathcal{O}(h^{3}).$ Note that in Algorithm \ref{algorithm2.1} $%
X_{k}^{' }\in \bar{G}\cup \bar{G}_{-r}$ with $r=\mathcal{O}(h^{1/2}) $ and the
one-step error of reflection was $\mathcal{O}(h^{3/2})$.

Let us write a generic approximation of weak local order $3$ in the form
suitable for this section:
\begin{equation}
X=x+\delta (t,x;h;\xi ),  \label{s1}
\end{equation}%
where $h>0$ is a time step, $\xi $ is a random vector \rf{the  components of which} are
some bounded mutually independent random variables and $\delta $ is such that for all $x\in \bar{G%
}$%
\begin{equation}
\mathbb{E}f(X_{t,x}(t+h))-\mathbb{E}f(x+\delta (t,x;h;\xi ))=\mathcal{O}(h^{3}),  \label{s2}
\end{equation}%
with $X_{t,x}(t+h)$ being a solution of the SDEs (\ref{eq:1}) from which the
reflection term is excluded, and $f$ being an arbitrary sufficiently smooth
function with bounded derivatives. Note that $\delta $ and $\xi $ in (\ref%
{s1})  used in this subsection are different \rf{from} $\delta $ and $\xi $
associated with Algorithm~\ref{algorithm2.1} which have been used everywhere else in this paper except this Subsection~\ref{section7.1}. There
are various numerical approximations satisfying (\ref{s1})-(\ref{s2}) (see
e.g. \cite{49} and references therein).

Introduce a layer $S_{t,h}\in \bar{Q}$ such that if $(t,x)\in \bar{Q}%
\backslash S_{t,h}$ then $X$ from (\ref{s1}) belongs to $\bar{G}$ and if $%
(t,x)\in S_{t,h}$ then at least one realization $_{(i)}X$ of $X$ does not belong
to $\bar{G}.$ Consequently, if $(t,x)\in \bar{Q}\backslash S_{t,h},$ we can
use (\ref{s1}) to approximate (\ref{eq:1}).

If $(t,x)\in S_{t,h},$ we find the largest time step $\theta \in \left[
h^{2},h\right] $ \rf{such} that all realizations of $x+\delta (t,x;\theta ;\xi )$
belong to $\bar{G}\cup \bar{G}_{-r}$ with \rf{$r=\Theta(h)$}. Then we draw a realization of
$\xi $ and compute
\[
X^{' }=x+\delta (t,x;\theta ;\xi ).
\]%
If $X^{' }\in \bar{G},$ we set $X=X^{' },$ otherwise
\[
X=X^{'}+2r_{0}\nu (X^{\pi }),
\]%
where as usual $r_{0}=\mathrm{dist}(X^{' },X^{\pi})$ and $X^{\pi }$
is the projection of $X^{' }$ on $\partial G.$

Having the above two ingredients, we now can construct a Markov chain $X_{k}$
approximating the solution of (\ref{eq:1}) in the weak sense, which is formulated
as Algorithm~\ref{algorithm7.1}.

\begin{algorithm}
\caption{Algorithm for second-order approximation}\label{algorithm7.1}
\begin{steps}
\begin{em}
\item  Set $\tau _{0}=\rs{t_{0}},$ $X_{0}=x,$ $k=0.$

\item \textbf{If} $\tau _{k}\leq T-h,$ then $h_{k} = h;$ \textbf{else} $h_{k}=T-\tau
_{k}.$

\item \textbf{If} $\rf{(\tau_k,X_{k})\in \bar{Q}}\backslash S_{\tau _{k},h_{k}},$ then
simulate $\xi _{k+1}$ independently of $\xi _{1},\ldots ,\xi _{k}$ and
\begin{eqnarray*}
\tau _{k+1} &=&\tau _{k}+h_{k}, \\
X_{k+1} &=&X_{k}+\delta (\tau _{k},X_{k};h_{k};\xi _{k+1}),
\end{eqnarray*}
\;\;\;\;\;\;where $\delta$ satisfies (\ref{s2}).
\textbf{Goto} Step 6.

\item Find the largest time step $\theta _{k+1} \in [ h^{2},h_{k} ] $ so that all realizations of $X_{k}+\delta (\tau
_{k},X_{k}; \theta _{k+1}; \rs{\xi_{k+1}})$ belong to $\bar{G}\cup \bar{G}_{-r}$
with $\rf{r=\Theta(h)}$. Simulate $\xi_{k+1}$ independently of $\xi
_{1},\ldots ,\xi _{k}$ and
\begin{eqnarray*}
\tau _{k+1} &=&\tau _{k}+\theta _{k+1}, \\
X_{k+1}^{' } &=&X_{k}+\delta (\tau _{k},X_{k};\theta _{k+1};\xi _{k+1}).
\end{eqnarray*}

\item \textbf{If} $X_{k+1}^{' }\in \bar{G},$ then $X_{k+1}=X_{k+1}^{' }$,
\textbf{else} (the case $X_{k+1}^{' }\notin \bar{G}$)
\[
X_{k+1}=X_{k+1}^{'}+2r_{k+1}\nu (X_{k+1}^{\pi }),
\]

\;\;\;\;\;\;where $r_{k+1}=\mathrm{dist}(X_{k+1}^{' }, X_{k+1}^{'})$ and $X_{k+1}^{\pi
} $ is the projection of $X_{k+1}^{'}$ on $\partial G$.

\item \textbf{If} $\tau _{k+1}\geq T-h^{2}$ then the random number of steps $%
\varkappa =k+1$, the final state of the chain $X_{\varkappa }=X_{k+1}$ \;\;\;\;\;\; and STOP; \textbf{else} $k=k+1$ and \textbf{Goto} Step 2.
\end{em}
\end{steps}
\end{algorithm}

Using the ideas introduced to study first order convergence of Algorithm \ref{algorithm2.1} in
Section \ref{section4}, one can prove convergence of Algorithm~\ref{algorithm7.1} with weak order 2, i.e.
\[
|\mathbb{E}f(X_{t,x}(T))-\mathbb{E}f(X_{\varkappa })|\leq Ch^{2},
\]%
where $X_{t,x}(T)$, $T_{0}\leq t \leq T$, is the solution of (\ref{eq:1}) and $C>0$ is a constant
independent of $h.$  It can be shown that the average number of steps
of Algorithm~\ref{algorithm7.1} is ${\cal{O}}(1/h)$. \rf{The Markov chain $(\tau_k,X_k)$ on average spends ${\cal{O}}(1/h)$ steps in the domain $\bar{Q}\backslash S_{\tau _{k},h_{k}}$ and also 
${\cal{O}}(1/h)$ steps
when $(\tau _{k},X_k) \in S_{\tau _{k},h_{k}}$. The proof of the latter exploits the rule how $\theta_k$ is chosen in the algorithm 
and also that, thanks to the definition of $S_{t,h}$, if $(\tau_k,X_k) \in S_{\tau _{k},h_{k}} $ then at least one realisation of $X_{k+1} \notin \bar G$.}

\subsection{Oblique reflection}\label{section7.2}

Consider the RSDEs
\begin{align}\label{eq:A1}
    dX(s) &= b(s,X(s))ds + \sigma(s,X(s))dW(s) + \eta(X(s))I_{\partial G}(X(s)) dL(s),
\end{align}
where $\eta$ is a unit vector field belonging to class $C^{4}$, and we assume that there exists a constant $c_{0} > 0$ such that
\begin{equation}
    (\eta(z)\cdot\nu(z)) > c_{0},
\end{equation}
for all $z \in \partial G$. The uniqueness and existence of the strong solution of SDEs with reflection in the oblique direction  was proved in \cite{11} under weaker assumptions than Assumptions~\ref{as:3}-\ref{as:2}, but we will continue using  Assumptions~\ref{as:3}-\ref{as:2} to ensure first-order of weak convergence of our algorithms. We note that $\eta$ in this subsection is an arbitrary oblique direction, while in Section~\ref{section6} $\eta$ was the co-normal, however this should not lead to any confusion. For applications of (\ref{eq:A1}) see e.g. \cite{william2019}.

 In Section \ref{section4}, we proposed Algorithm \ref{algorithm3.1} to approximate the system (\ref{eq:8})-(\ref{eq:10}). Here, for simplicity we only consider approximation of the RSDEs (\ref{eq:A1}). Again we assume a uniform discretization of the time interval $[t_{0},T]$, $ t_{0} < \dots < t_{N}=T$, $h = (T-t_{0})/N$ and $t_{k+1} = t_{k} + h$. As was elaborated in Section~\ref{section2} while constructing  the Markov chain $(X_{k})_{k\geq 0}$, we take an auxiliary step
 \begin{equation}
     X_{k+1}^{'} = X_{k} + hb_{k} + h^{1/2} \sigma_{k}\xi_{k+1}, \label{eq:38}
 \end{equation}
 where $b_{k} = b(t_{k},X_{k})$, $\sigma_{k} = \sigma(t_{k},X_{k})$, and $\xi_{k+1} $ represents the same random vector as in (\ref{eq:11}). We follow the same idea as in Algorithm~\ref{algorithm2.1} of first taking an intermediate step, $ X_{k+1}^{'}$. If $ X_{k+1}^{'} \in \bar{G}$ then we set $X_{k+1} = X_{k+1}^{'}$, and if $X_{k+1}^{'} \in \bar{G}^{c} $ then we solve the following \rs{system of implicit equations}:
 \begin{align}
     X^{\pi}_{k+1} &= X^{'}_{k+1} + \dist(X^{'}_{k+1},X_{k+1}^{\pi})\eta(X_{k+1}^{\pi}), \label{3.11}\\
          X_{k+1}^{\pi} &\in \partial G, \label{3.13}
 \end{align}
 in order to find the projection $X_{k+1}^{\pi}$ of $X_{k+1}^{'}$ on $ \partial G$ along the direction $\eta(X_{k+1}^{\pi})$, where $\eta(X_{k+1}^{\pi})$ is the inward oblique direction at $X_{k+1}^{\pi}$.  By \cite[Proposition 1]{2} there exists a unique solution of the system of equations  (\ref{3.11})-(\ref{3.13}) under Assumption \ref{as:3} \rf{for all $X_{k+1}^{'} \in \bar{G}_{-r}$}.

 We denote $ r_{k+1}$ as the distance of $X_{k+1}^{'}$ from the boundary $\partial G$ along the direction $\eta(\rf{X_{k+1}^{\pi}})$, i.e. $r_{k+1}=\dist(X_{k+1}^{'}, X_{k+1}^{\pi})$.
 If $ X_{k+1}^{'}$ goes outside $ \bar{G}$,  the following symmetric step is taken:
 \begin{equation}
     X_{k+1} = X_{k+1}^{'} + 2r_{k+1}\eta(X_{k+1}^{\pi}).  \label{eq:41}
 \end{equation}
 This algorithm is presented in a formalized form as Algorithm~\ref{alg:euclid}. One can prove its first-order of weak convergence analogously to proofs for Algorithm~\ref{algorithm3.1} considered in earlier sections.

 \begin{algorithm}
\caption{Algorithm to approximate obliquely reflected diffusion}\label{alg:euclid}
\begin{steps}
\begin{em}
\item \begin{textit} Set $X_{0} = x$, $X_{0}^{'}= x$, $k=0$.\end{textit}
\item \begin{textit}Simulate $\xi_{k+1}$ and find $X_{k+1}^{'}$ according to (\ref{eq:38}).\end{textit}
\item \textbf{If}  ${ X_{k+1}^{'} \in \bar{G}} $ then $ X_{k+1} = X_{k+1}^{'}$,\;\textbf{else}
 \begin{enumerate}[label= (\roman*)]
\item  find the projection, $X_{k+1}^{\pi}$, of $X_{k+1}^{'}$ on $\partial G$ along the oblique direction $\eta(X_{k+1}^{\pi})$ according to (\ref{3.11})-(\ref{3.13});
\item compute $ r_{k+1} = dist(X_{k+1}^{'}, X_{k+1}^{\pi}) $ and find $X_{k+1}$ according to (\ref{eq:41}).
\end{enumerate}
\item \textbf{If} $k+1=N$ then \textbf{stop}, \textbf{else} put $k := k+1$ and \textbf{return} to Step 2.
\end{em}
\end{steps}
\end{algorithm}
\vspace{-3ex}
\section{Numerical experiments}\label{section8}

In this section we perform numerical experiments to support the theoretical results obtained in Sections~\ref{section4}-\ref{section5}.

To evaluate the expectation $\mathbb{E}\rf{\Gamma}$ where $\rf{\Gamma}$ is a generic random variable with some finite moments, we use the Monte Carlo technique in the usual fashion:
\begin{align}\label{montcest}
    \mathbb{E}\rf{\Gamma} \simeq \check{\rf{\Gamma}}_{M} := \frac{1}{M}\sum\limits_{m=1}^{M}\rf{\Gamma}^{(m)},
\end{align}
where $\rf{\Gamma}^{(m)}$ are independent realisations of $\rf{\Gamma}$. The Monte Carlo error  is computed via the sample variance:
\begin{equation*}
    \bar{D}_{M} = \frac{1}{M}\bigg(\frac{1}{M}\sum\limits_{m=1}^{M} (\rf{\Gamma}^{(m)})^{2} -\check{\rf{\Gamma}}_{M}^{2}\bigg).
\end{equation*}
The $95\% $ confidence interval for $\mathbb{E}\rf{\Gamma}$ is $\check{\rf{\Gamma}}_{M} \pm 2 \sqrt{\bar{D}_{M}}$.

\subsection{Parabolic PDE}

In this subsection we solve a parabolic PDE with Neumann boundary condition using Algorithm~\ref{algorithm3.1}. 
\begin{experiment}\normalfont
Take the circular domain $G=\{x_{1}^{2} + x_{2}^{2} < R^{2}\}$ and $\partial G=\{x_{1}^{2} + x_{2}^{2} = R^{2}\}$. Consider the  parabolic problem
    \begin{align}\label{expeq1}
            \frac{\partial u}{\partial t}
            + \frac{1}{2}\frac{\partial^{2}u}{\partial x_{1}^{2}} + 2\frac{\partial^{2} u}{\partial x_{2}^{2}}
            - x_{2}\frac{\partial u}{\partial x_{1}} + x_{1}\frac{\partial u}{\partial x_{2}}   + 5(1+\exp^{-(T-t)})\nonumber & \\ - (25-x_{1}^{2} - x_{2}^{2})\exp^{-(T-t)} = 0, \;\; (t,x) \in [0,T) \times G,
    \end{align}
with terminal condition
\begin{equation}\label{expeq2}
    u(T,x) = 2(25-x_{1}^{2} - x_{2}^{2}), \;\;\;\;\;\;\;  x \in \bar{G},
\end{equation}
and Neumann boundary condition
\begin{equation}\label{expeq3}
    \frac{\partial u}{\partial \nu} = 2R(1+\exp^{-(T-t)}), \;\;\;\;\;\; z \in \partial G \;\; \text{and} \;\; t \in [0,T).
    \end{equation}
The solution of the above problem is given by $u(t,x_{1},x_{2}) = (25-x_{1}^{2}-x_{2}^{2})(1+\exp^{-(T-t)})$. The exact solution at $(t, x_{1}, x_{2}) = (0,0,0)$ with \rf{$R = 2$} and $T= 1$ is $34.1970$ (4 d.p.). We note that the construction of the model problem in this experiment follows the same path as in \cite{82} (see also \cite[Chapter 6]{49}).

We consider the absolute error $e=|u(0,0,0) - \check{u}_{M}|$, where  $\check{u}_{M}$ is the Monte Carlo estimator for $\bar u(0,0,0)=\mathbb{E}[u(T,X_N)Y_N+Z_N]$ which approximates $u(0,0,0)$. Here $(X_N,Y_N,Z_N)$ is due to Algorithm~\ref{algorithm3.1} applied to the problem (\ref{expeq1})-(\ref{expeq3}). The results are presented in Table~\ref{exptab8.1} and Figure~\ref{exptab8.1}, which demonstrate that the numerical integration error incurred in solving the parabolic problem  (\ref{expeq1})-(\ref{expeq3}) is of order $\mathcal{O}(h)$ and hence it is consistent with the prediction of  Theorem~\ref{conp}.
\end{experiment}

\begin{minipage}[b]{.40\textwidth}
\captionof{table}{Numerical solution of the parabolic problem (\ref{expeq1})-(\ref{expeq3}) using Algorithm \ref{algorithm3.1}.}\label{exptab8.1}

\centering 

\setlength{\tabcolsep}{2.9pt}
\begin{tabular}{c c c c c} 
\hline\hline 
 $h$ & $M$ & $\check{u}_{M}(0,0,0) \pm 2 \sqrt{\bar{D}_{M}} $ & $e$\\ [0.5ex] 
\hline 
  0.1 & $10^{5}$ & $36.5301 \pm \; 0.0408$ & 2.3331 \\ 
  0.05 & $10^{6}$ & $35.4104 \pm \;0.0133$ & 1.2134\\
 0.025 & $10^{6}$ & $34.8223 \pm \;0.0135$ & 0.6253\\
 0.0125 & $10^{6}$ & $34.5120 \pm \;0.0136$ & 0.3150\\[1ex]
\hline 
\end{tabular}
\end{minipage}\qquad
 \hfill%
  \begin{minipage}[t]{.47\textwidth}
   \centering
        \includegraphics[width=6.4cm]{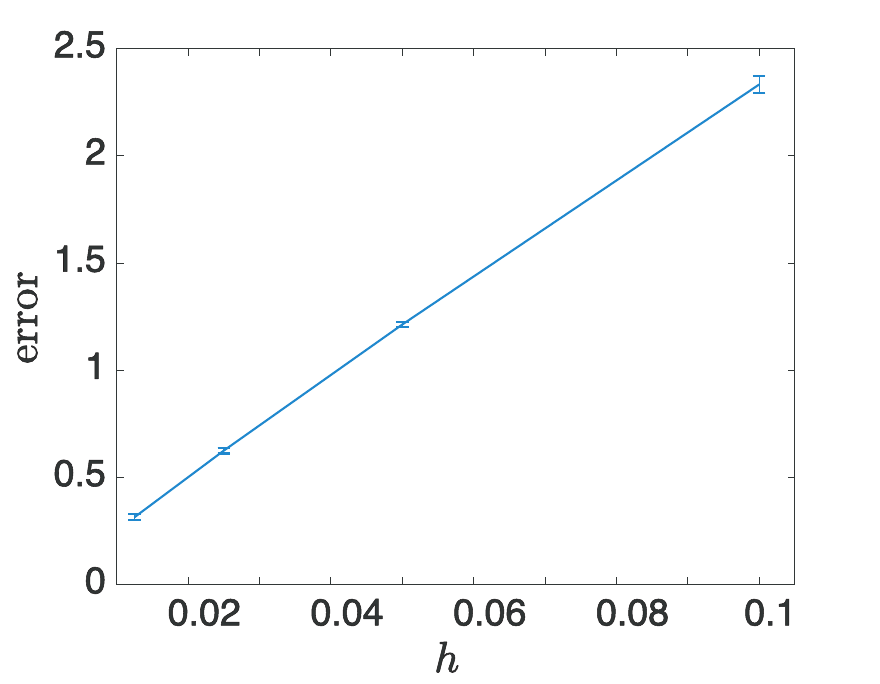}
    \captionof{figure}{Plot to show the first order of convergence of Algorithm~\ref{algorithm3.1} applied to (\ref{expeq1})-(\ref{expeq3}). Error bars correspond to the Monte Carlo error.}\label{expfig8.1}\end{minipage}\qquad \qquad \qquad

\subsection{Calculating ergodic limits}\label{section8.2}
In this subsection we approximate $\bar{\varphi}$ (cf. (\ref{q2.24})) and $\bar{\psi}^{'}$ (cf. (\ref{q2.26})) using the numerical time-averaging estimators  (cf. (\ref{3.32}) and (\ref{5.39})) and the numerical ensemble-averaging estimators (cf. (\ref{e5.73}) and (\ref{v5.82}) together with the Monte Carlo technique (\ref{montcest})) based on Algorithm \ref{algorithm2.1}.

There are three types of errors incurred while computing $\bar{\varphi}$ and $\bar{\psi}^{'}$ via numerical time-averaging estimators $\hat{\varphi}_{N}$ and $\hat{\psi}^{'}_{N}$ \cite{40,41}:  (i) the error due to discretization of RSDEs (\ref{rsde}) (the numerical integration error) which is estimated by $Ch$, (ii) the error incurred because integration is over finite time $T$ (the time truncation error) which is estimated by $C/T$, and (iii) the statistical error.
The statistical error is also controlled by the final finite time $T$ (see e.g. \cite{41} for further discussion on the statistical error).

In practice to find the statistical error of $\hat{\varphi}_{N}$ and $\hat{\psi}^{'}_{N}$, we simulate a long trajectory (here according to Algorithm~\ref{algorithm2.1}) and split that trajectory into $L$ blocks of time length $T'$, which means that the total length of the simulated trajectory is $T=LT'$.
We calculate $\hat{f}^{(i)}$ (here $\hat{f}^{(i)}$ is either  $\hat{\varphi}^{(i)}$ or $\hat{\psi}^{'}{}^{(i)}$) computed over the $i$th block, $i = 1,\dots,L$, and the statistical error is estimated as
\begin{equation}
    \bar{J}_{L} = \frac{1}{L}\bigg(\rf{\frac{1}{L}}\sum\limits_{i=1}^{L}\big(\hat{f}^{(i)}\big)^{2} - \Big(\frac{1}{L}\sum\limits_{i=1}^{L}\hat{f}^{(i)}\Big)^{2}\bigg).
\end{equation}
This implies that the $95\%$ confidence interval of $\mathbb{E}\hat{f}$ is given by $\hat{f} \pm 2\sqrt{\bar{J}_{L}}$, where the time-averaging estimator $\hat{f}$ is computed over the whole trajectory with the length $T=LT'$.

\begin{experiment}
\normalfont Consider the circular domain $ G=\{ x_{1}^{2} + x_{2}^{2} < R^{2}\} $ with boundary $\partial G=\{ x_{1}^{2} + x_{2}^{2}  = R^{2}\} $, where $R >0$ is the radius.
Consider the RSDEs
\begin{align} \label{exprsde1}
&\begin{bmatrix} dX_{1}(s)\\ dX_{2}(s) \end{bmatrix}= -\begin{bmatrix}
\frac{X_{1}(s)}{2} + \frac{X_{2}(s)}{4} \\ \frac{X_{1}(s)}{4} +\frac{X_{2}(s)}{2}
\end{bmatrix}ds \\&  \;\; + \begin{bmatrix}
     \sin(X_{1}(s)+X_{2}(s)) & \cos(X_{1}(s)+X_{2}(s)) \\
     \sin(X_{1}(s)+X_{2}(s)+\frac{\pi}{3}) & \cos(X_{1}(s) + X_{2}(s) +\frac{\pi}{3})
     \end{bmatrix}\begin{bmatrix} dW_{1}(s) \\ dW_{2}(s)\end{bmatrix} \nonumber - \frac{1}{R}\begin{bmatrix} X_{1}(s) \\ X_{2}(s)\end{bmatrix}dL(s) .
\end{align}
The invariant density of the RSDEs (\ref{exprsde1}) is
\begin{equation*}
    \rho(x_{1},x_{2}) = \frac{1}{2\pi(1-e^{-R^{2}/2})}e^{-(x_{1}^{2} + x_{2}^{2})/2}.
\end{equation*}
It is not difficult to verify that the density function $\rho(x_{1},x_{2})  $ satisfies the stationary Fokker-Planck equation (\ref{sFP}) with Neumann boundary condition (\ref{sFPbc}) corresponding to the RSDEs (\ref{exprsde1}).

In this experiment, we approximate the ergodic limit $\bar{\varphi}$ (cf. (\ref{q2.24})) with $\varphi(x_{1}, x_{2}) = x_{1}^{2} + x_{2}^{2}$ using numerical time-averaging estimator (cf. (\ref{3.32})) and discretized ensemble-averaging estimator (cf. (\ref{e5.73}) together with (\ref{montcest})) based on Algorithm \ref{algorithm2.1}. To evaluate the performance of Algorithm~\ref{algorithm2.1} for computing ergodic limits, we have the exact value of $\bar{\varphi}$:
\begin{equation*}
    \bar{\varphi} = \frac{2- 2e^{-R^{2}/2} - R^{2}e^{-R^{2}/2}}{1-e^{-R^{2}/2}}.
\end{equation*}
For $R = 2$, we have $\bar{\varphi} = 1.3739$ (4 d.p.).
We introduce $error_{\rm{ta}} = | \bar{\varphi} - \hat{\varphi}_{N}| $, which is the absolute error of the time-averaging estimator (\ref{q2.24}), and $error_{\rm{ea}} = | \bar{\varphi}-\check{\varphi}_{M}|$, where $\check{\varphi}_{M}$ is the Monte Carlo estimator (see (\ref{montcest})) of $\mathbb E \varphi(X_N) $ (see (\ref{e5.73})), i.e. $error_{\rm{ea}}$ is the absolute error of the ensemble-averaging estimator. The simulations are run by taking $(0,0)$ as the starting point in both time-averaging and ensemble-averaging.

In Table \ref{exptab8.2}, the second column corresponds to time-averaging estimation of $\bar{\varphi}$ and it also includes the statistical error. The fourth column corresponds to ensemble-averaging estimation of $\bar{\varphi}$ and it also includes the Monte Carlo error. Table~\ref{exptab8.2} and Figure~\ref{expfig8.2} verify the theoretical results of Theorem~\ref{thrm3.2} (see (\ref{est1})) and Theorem~\ref{theorem6.14} (see (\ref{e5.73})) as the absolute errors $error_{\rm{ta}}$ and $error_{\rm{ea}}$ are of order $\mathcal{O}(h)$.

\end{experiment}


\begin{minipage}{.60\textwidth}
\vspace{-37ex}
\captionof{table}{The parameters for calculating $\rs{\hat{\varphi}_{N}}$ are $L = 10^{4}$ and $T' = 10$  and the parameters for calculating $\check{\varphi}_{M}$ are $T =5$ and $M = 2 \times 10^{4}$.}\label{exptab8.2}
\begingroup
\setlength{\tabcolsep}{1.7pt}
\begin{tabular}{c c c c c } 
\hline\hline 
 $h$ & $\hat{\varphi}_{N} \pm 2\sqrt{\bar{J}_{L}}$ & $error_{\rm{ta}}$ & $\check{\varphi}_{M} \pm 2\sqrt{\bar{D}_{M}} $ & $error_{\rm{ea}}$ \\  
 [0.5ex] 
\hline 
 0.4 & 1.6590 $\pm$ 0.0073  & 0.2850  & 1.6651 $\pm$ 0.0167  & 0.2911\\
 0.25 & 1.5428 $\pm$ 0.0072& 0.1688 &  1.5403 $\pm$ 0.0164 & 0.1663 \\
 0.2 & 1.5039 $\pm$ 0.0072 & 0.1299 & 1.4989 $\pm$ 0.0161  & 0.1249\\
 0.1 & 1.4236 $\pm$ 0.0070 & 0.0496  & 1.4262 $\pm$ 0.0154  & 0.0522 \\
  [1ex]
\hline 
\end{tabular}
\endgroup
\end{minipage}
 \hfill%
  \begin{minipage}[b]{.30\textwidth}
  \hspace{-0.8cm}
        \includegraphics[width=4.7cm]{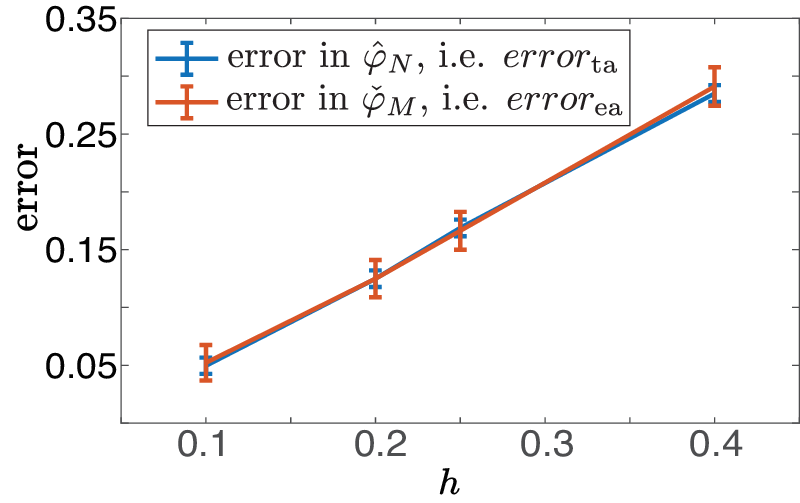}
 \captionof{figure}{Plot to show linear dependence of the errors in computing the ergodic limit, $error_{\rm{ta}}$ and $error_{\rm{ea}}$, on $h$.}\label{expfig8.2}
\end{minipage}

\begin{experiment}\label{experiment8.3}
\normalfont Consider the sphere $G=\{x_{1}^{2} +x_{2}^{2} + x_{3}^{2} < 1\}$ with boundary $\partial G=\{x_{1}^{2} + x_{2}^{2} + x_{3}^{2} = 1\}$ and the RSDEs
\begin{equation*}
    dX(s) = \frac{1}{|X(s)|}\Big(V - \frac{(V\cdot X(s))}{|X(s)|^{2}}X(s)\Big)ds + \sqrt{2}dW(s) - X(s)dL(s),
\end{equation*}
where $W(s)$ is a $3$-dimensional standard Wiener process and $V = (1/2,1/2,1/\sqrt{2})^{\top}$.  The invariant density of the above RSDEs is
\begin{equation*}
    \rho(x) = \frac{3}{4\pi \sinh{1}}e^{(V\cdot x)/|x|},
\end{equation*}
 and the corresponding normalised restricted density $\rho'(z)$ is
\begin{equation*}
    \rho'(z) = \frac{1}{4\pi \sinh{1}}e^{(V\cdot z)},
\end{equation*}
where $x=(x_{1},x_{2},x_{3}) \in \bar{G}$ and $z = (z_{1},z_{2},z_{3}) \in \partial G$. The function $\rho'(z)$ is the density of Fisher distribution with the concentration parameter $1$ and mean direction $V$ \cite{84}.

We are interested in calculating $\kappa$ (see (\ref{eb2.25})) and  $\bar{\psi}^{'}$  for $\psi(z) = z_{1} + z_{2} + z_{3}$. The exact value of $\kappa$ is $3$ and of $\bar{\psi}^{'}$ is 0.53438 (5 d.p.). We introduce the absolute errors $e_{1} = |\hat{\psi}^{'}_{N} - \bar{\psi}^{'} | $ and $e_{2} = | \hat{\kappa}_{N} - \kappa |$, where $\hat{\psi}^{'}_{N}$ is from (\ref{5.39}) and $\hat{\kappa}_{N}$ is from (\ref{3.33}).
\end{experiment}

\begin{minipage}{.53\textwidth}
\vspace{-28ex}
\captionof{table}{The parameters used are  $T' = 30$ and $L = 10^{4}$. The initial point taken for simulation is (-0.5,-0.5,-0.5). }\label{exptab8.3}
\centering 
\begingroup
\setlength{\tabcolsep}{1.7pt}
\begin{tabular}{c c c c c  } 
\hline\hline 
 $h$ & $\hat{\psi}_{N}^{'} \pm2\sqrt{\bar{J_{L}}}$ & $e_1$  & $\hat{\kappa}_{N}\pm2\sqrt{\bar{J_{L}}}$ &  $e_2$ \\ [0.5ex] 
\hline 
 0.1  & 0.5673 $\pm$ 0.0025 & 0.03294 & 3.1183 $\pm$ 0.0065 & 0.1183 \\
 0.075  & 0.5602 $\pm$ 0.0025 & 0.02577& 3.0813 $\pm$ 0.0071 & 0.0813\\
 0.05  & 0.5557  $\pm$ 0.0026 & 0.02127 & 3.0510 $\pm$ 0.0039 & 0.0510\\
 0.025  & 0.5457 $\pm$ 0.0026 & 0.01130 & 3.0227 $\pm$ 0.0040 & 0.0227\\ [1ex]
\hline 
\end{tabular}
\endgroup

\end{minipage}\qquad
 \hfill%
  \begin{minipage}[b]{.30\textwidth}
  \hspace{-0.5cm}
        \includegraphics[width=4.5cm]{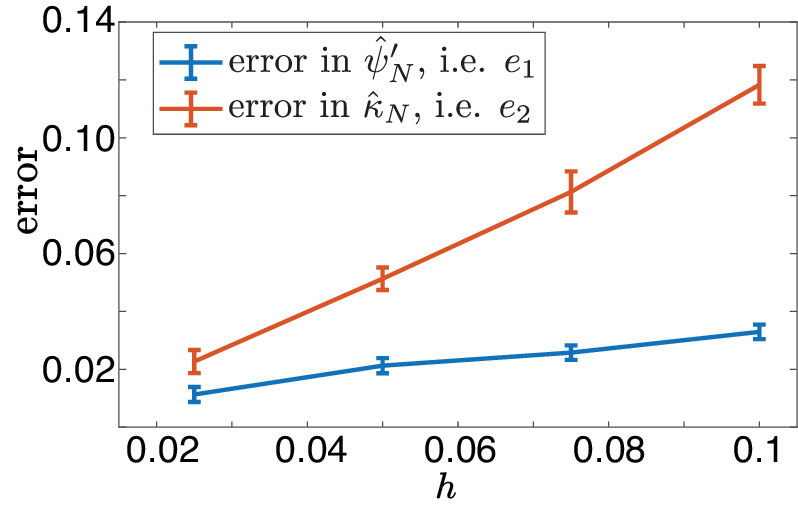}
    \captionof{figure}{Plot to show the dependence of the errors $e_{1}$ and $e_{2}$ on $h$.}\label{expfig8.3}
\end{minipage}\qquad \qquad \qquad

The second column in Table \ref{exptab8.3} corresponds to estimation of $\bar{\psi}^{'}$ and the fourth column - to estimation of $\kappa$  via the time averaging estimators (\ref{5.39}) and (\ref{3.33}), respectively. One can see from Table~\ref{exptab8.3} and Figure~\ref{expfig8.3} that the absolute errors $e_{1}$ and $e_{2}$ decay linearly with $h$ for sufficiently large  $T=LT'$, which verifies  Theorem~\ref{thrm6.5} and the estimate (\ref{kappa}) of Lemma~\ref{aapl4.4}.

\subsection{Elliptic PDEs}
In this subsection we present two numerical experiments to verify the theoretical results proved in Section~\ref{section5}. 

\begin{experiment} \normalfont
Consider the torus $ G =\{\big(\sqrt{x_{1}^{2} + x_{2}^{2}} - R\big)^{2} + x_{3}^2 < r^{2}\} $ with boundary $\partial G =\{\big(\sqrt{x_{1}^{2} + x_{2}^{2}} - R\big)^{2} + x_{3}^2 = r^{2}\} $, where $0 < r < R$, and the elliptic equation
    \begin{align}\label{expeleq}
    \frac{1}{2}\frac{\partial^{2}u}{\partial x_{1}^{2}} + & \frac{5}{2}\frac{\partial^{2} u}{\partial x_{2}^{2}}  + \frac{\partial^{2} u}{\partial x_{3}^{2}} + \frac{\partial^{2}u}{\partial x_{1}\partial x_{2}} + 2\frac{\partial^{2}u}{\partial x_{2}\partial x_{3}}
            - x_{3}\frac{\partial u}{\partial x_{1}} + x_{1}\frac{\partial u}{\partial x_{2}} +     x_{2}\frac{\partial u}{\partial x_{3}}     - 2u   +  2x_{3}^{3} \nonumber \\ & - 3x_{2}x_{3}^{2}   + 2x_{2}^{2} - 2x_{1}x_{2} + 2x_{1} - 5x_{3} - 5    = 0, \;\; (x_{1},x_{2},x_{3}) \in  G,
    \end{align}
with Neumann boundary condition
\begin{equation}\label{expelbc}
    \frac{\partial u}{\partial \nu} = \frac{1}{r}\bigg(\Big(\frac{R}{\sqrt{\rs{z_{1}^{2} + z_{2}^{2}}}} - 1\Big)\big(\rs{z_{1}} + 2\rs{z_{2}^{2}}\big) - 3\rs{z_{3}^{3}}\bigg), \;\;\;\;\;\; (z_{1},z_{2},z_{3}) \in \partial G.
    \end{equation}
 In this experiment we solve the elliptic equation (\ref{expeleq}) with boundary condition (\ref{expelbc}) and verify the estimate (\ref{eq:th5.4}) that the numerical integration error decays linearly with respect to step size $h$ for sufficiently large $T$. For this experiment, we choose $r= 2$ and $R= 4$. The exact solution is given by $u(x_{1},x_{2},x_{3}) = x_{1} + x_{2}^{2} + x_{3}^{3}$ and at $(x_{1}, x_{2}, x_{3}) = (1,2,1/2)$ solution is $5.125$. We denote the absolute error as $e = |u(1,2,1/2) - \check{u}_{M}|$, where $\check{u}_{M}$ is the Monte Carlo estimator (see (\ref{montcest})) of the approximation of the solution to (\ref{expeleq})-(\ref{expelbc}) computed using Algorithm~\ref{algorithm3.1}. Table \ref{exptab8.4} and Figure~\ref{fig8.4} confirm the result of Theorem \ref{thrm5.4}.
\end{experiment}

\begin{minipage}{.40\textwidth}
\vspace{-30ex}
  \captionof{table}{Results of numerical simulation for the elliptic problem (\ref{expeleq})-(\ref{expelbc}) with $T=4$.}\label{exptab8.4}
  \setlength{\tabcolsep}{1.7pt}
\begin{tabular}{c c c c} 
\hline\hline 
 $h$  & $M$ & $\check{u}_{M}$(1,2,1/2) $\pm$\;$2\sqrt{\bar{D}_{M}}$ & $e$\\ [0.5ex] 
\hline 
 0.1  & $10^{6}$ & 5.5144 $\pm$ \;0.0168 & 0.3894 \\ 
 0.08  & $10^{6}$ & 5.4615 $\pm$ \;0.0167 & \rs{0.3365}\\
 \rs{0.04}  & \rs{$10^{6}$} & \rs{5.2874 $\pm$ \;0.0163} & \rs{0.1624}\\
 0.01  & $10^{7}$ & 5.0820 $\pm$ \;0.0050 & 0.0430\\ [1ex] 
\hline 
\end{tabular}

\end{minipage}\qquad
 \hfill%
  \begin{minipage}[b]{.37\textwidth}
   \centering
        \includegraphics[width=5.5cm]{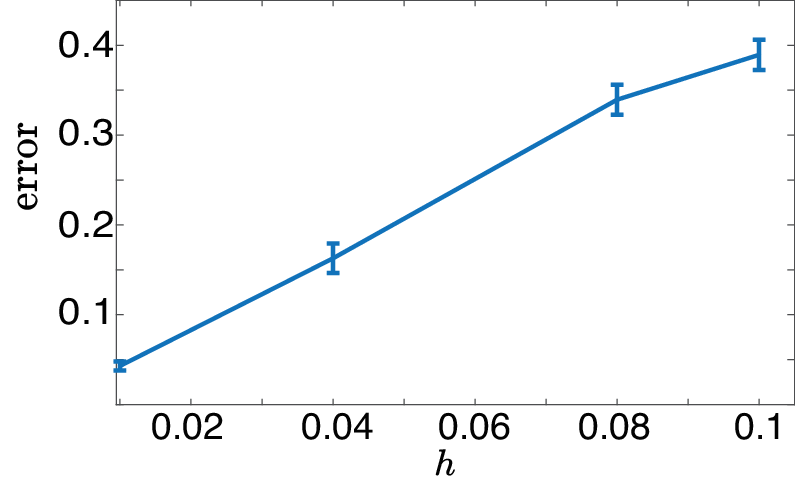}
    \captionof{figure}{Plot to show linear dependence of absolute error $e$ on $h$.}\label{fig8.4}
\end{minipage}\qquad \qquad \qquad

\begin{experiment}\label{experiment7.5} \normalfont
Consider the circular domain $ G=\{ x_{1}^{2} + x_{2}^{2} < 4\} $ with boundary $\partial G =\{x_{1}^{2} + x_{2}^{2}  = 4\}$ and the  Poisson PDE
    \begin{align}\label{expeq8.5.1}
    \frac{1}{2}\frac{\partial^{2}u}{\partial x_{1}^{2}} + \frac{1}{2}\frac{\partial^{2} u}{\partial x_{2}^{2}} + & \frac{1}{2}\frac{\partial^{2} u}{\partial x_{1} \partial x_{2}}
            - \Big(\frac{x_{1}}{2} + \frac{x_{2}}{4}\Big)\frac{\partial u}{\partial x_{1}} \nonumber \\ &  - \Big(\frac{x_{1}}{4} + \frac{x_{2}}{2}\Big)\frac{\partial u}{\partial x_{2}}         = 2 - x_{1}^{2} - x_{2}^{2} -  x_{1}x_{2}, \;\; (x_{1},x_{2}) \in  G,
    \end{align}
with Neumann boundary condition
\begin{equation}\label{expeq8.5.2}
    \frac{\partial u}{\partial \nu} = -4, \;\;\;\;\;\; (z_{1},z_{2}) \in \partial G.
    \end{equation}
Note that the underlying RSDE for (\ref{expeq8.5.1})-(\ref{expeq8.5.2}) is (\ref{exprsde1}).
It is not difficult to show that the compatibility condition (\ref{cc}) is satisfied for (\ref{expeq8.5.1})-(\ref{expeq8.5.2}). The exact solution is given by $u(x_1,x_2)=x_1^2+x_2^2$, $(x_1,x_2) \in \bar G$. In this experiment we use Algorithm \ref{algorithm6.1} to solve the Poisson problem (\ref{expeq8.5.1})-(\ref{expeq8.5.2}).  As we know from Subsection~\ref{section6.3}, we can numerically evaluate the solution of the Poisson PDE with Neumann boundary condition up to an additive constant $\bar{u}$. For comparing the exact solution and numerical solution, we computed the exact value of $\bar{u} = 1.374$ (3 d.p.). We evaluate $u(x_{1},x_{2})- \bar{u}$ at $(x_{1},x_{2}) = (\sqrt{2},\sqrt{2})$, \rf{the} exact value \rf{of which} is 2.626 (3 d.p.). We introduce the absolute error $e=|u(\sqrt{2},\sqrt{2}) - \bar{u} - \check{u}_{M}|$. Table~\ref{exptable8.5} and Figure~\ref{expfig8.5} show that $e=\mathcal{O}(h)$ that verifies the result of Theorem~\ref{theorem6.7}.
\end{experiment}
\begin{minipage}{.40\textwidth}
\vspace{-30ex}
\captionof{table}{Numerical solution of the Poisson problem (\ref{expeq8.5.1})-(\ref{expeq8.5.2}) at $(\sqrt{2},\sqrt{2})$ up to the additive constant $\bar u$ using Algorithm~\ref{algorithm6.1} with  parameters $T=5$, $\ell = 0.1$, $\beta = 1$, $\Upsilon = 1$.}\label{exptable8.5}
\setlength{\tabcolsep}{1.7pt}
\begin{tabular}{c c c c c} 
\hline\hline 
$h$ & $M$ & $\check{u}_{M}(\sqrt{2}$,$\sqrt{2}$)$\pm$ $2\sqrt{\bar{D}_{M}}$ & $e$\\ [0.5ex] 
\hline 
 0.5  & $10^{4}$ &3.781 $\pm$ \;0.111 &  1.155 \\
 0.4  & $10^{5}$ & 3.438 $\pm$ \;0.035 & 0.812\\
 0.3  & $10^{5}$ & 3.194 $\pm$ \;0.034 & 0.568\\
0.2 & $10^{5}$ & 2.998 $\pm$ \;0.035 & 0.372\\ [1ex] 
\hline 
\end{tabular}
\end{minipage}\qquad
 \hfill%
 \begin{minipage}[b]{.50\textwidth}
   \centering
        \includegraphics[width=6cm]{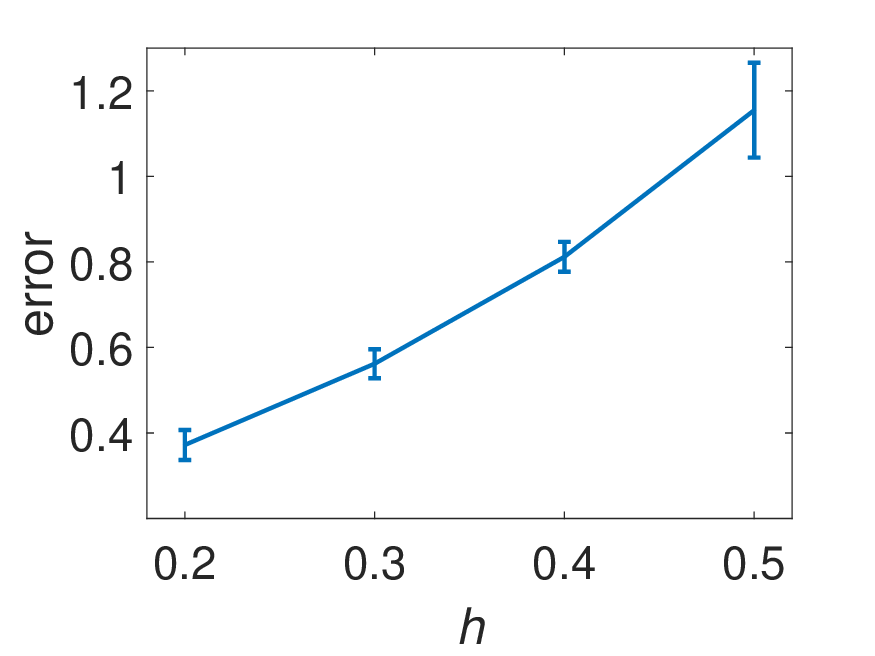}
    \captionof{figure}{Plot to show linear dependence of the absolute error $e$ on $h$. Error bars correspond to the Monte Carlo error from the third column of Table \ref{exptable8.5}.}\label{expfig8.5}
\end{minipage}\qquad \qquad \qquad

\section*{Acknowledgements}
BL and MVT thank Institute for Computational and Experimental Research in Mathematics (ICERM, Providence, RI) where this work was nucleated during a workshop hosted by ICERM.
BL  was  supported by EPSRC grant no. EP/P006175/1 and by the Alan Turing Institute (EPSRC EP/N510129/1) as a Turing Fellow. AS was supported by the University of Nottingham Vice-Chancellor's Scholarship for Research Excellence (International).
\rf{The authors thank anonymous referees for useful suggestions.}

\bibliographystyle{imsart-number} 
\bibliography{references}

\end{document}